\documentclass[a4paper,dvipsnames]{amsart}

\usepackage{hyperref}
\usepackage[dvipsnames]{xcolor} 
\usepackage{amssymb}
\usepackage{wasysym}
\usepackage{marvosym}
\usepackage{enumitem}
\usepackage{tikz}
\usepackage{tikz-cd}
\usepackage{mathtools}
\usepackage{mathscinet}
\usepackage{psfrag}
\usepackage{datetime}
\usepackage{url}
\usepackage{adjustbox}

\usetikzlibrary{decorations.pathmorphing}
\usetikzlibrary{decorations.pathreplacing}
\usetikzlibrary{positioning}
\usetikzlibrary{calc}
\usetikzlibrary{intersections}

\makeatletter
\tikzset{
    dot diameter/.store in=\dot@diameter,
    dot diameter=3pt,
    dot spacing/.store in=\dot@spacing,
    dot spacing=10pt,
    dots/.style={
        line width=\dot@diameter,
        line cap=round,
        dash pattern=on 0pt off \dot@spacing
    }
}
\makeatother

\numberwithin{equation}{section}

\makeatletter
\newcommand\@dotsep{4.5}
\def\@tocline#1#2#3#4#5#6#7{\relax
  \ifnum #1>\c@tocdepth 
  \else
    \par \addpenalty\@secpenalty\addvspace{#2}%
    \begingroup \hyphenpenalty\@M
    \@ifempty{#4}{%
      \@tempdima\csname r@tocindent\number#1\endcsname\relax
    }{%
      \@tempdima#4\relax
    }%
    \parindent\z@ \leftskip#3\relax \advance\leftskip\@tempdima\relax
    \rightskip\@pnumwidth plus1em \parfillskip-\@pnumwidth
    #5\leavevmode\hskip-\@tempdima{#6}\nobreak
    \leaders\hbox{$\m@th\mkern \@dotsep mu\hbox{.}\mkern \@dotsep mu$}\hfill
    \nobreak
    \hbox to\@pnumwidth{\@tocpagenum{\ifnum#1=1\fi#7}}\par
    \nobreak
    \endgroup
  \fi}
\AtBeginDocument{%
\expandafter\renewcommand\csname r@tocindent0\endcsname{0pt}
}
\def\l@subsection{\@tocline{2}{0pt}{2.5pc}{5pc}{}}
\makeatother

\newtheorem{theorem}{Theorem}[section]
\newtheorem{lemma}[theorem]{Lemma}
\newtheorem{proposition}[theorem]{Proposition}
\newtheorem{corollary}[theorem]{Corollary}

\theoremstyle{definition}
\newtheorem{definition}[theorem]{Definition}

\newtheorem{example}[theorem]{Example}

\theoremstyle{remark}
\newtheorem{remark}[theorem]{Remark}

\renewcommand{\emptyset}{\varnothing}

\newcommand{\nonconsec}[2]{\mathbf{K}_{#1}^{#2}}

\newcommand{\chainset}[2]{{\mathbf{J}_{#1}^{#2}}}
\newcommand{\floor}[1]{\lfloor #1 \rfloor}

\newcommand{\vorder}{\overset{v}{\leqslant}}
\newcommand{\vcover}{\overset{v}{\lessdot}}

\newcommand{\CL}{\mathcal{L}}
\newcommand{\CU}{\mathcal{U}}

\newcommand{\sse}{\mathcal{W}}
\newcommand{\sect}[1]{\sse(#1)}
\newcommand{\synth}{\sect{\mathcal{L}}}
\newcommand{\synthb}{\sect{\mathcal{L}'}}
\newcommand{\relcirc}[1]{(#1_{-} \cup x, #1_{+} \cup y)}

\newcommand{\tvplusx}{\mathcal{T}\backslash v^{+} \ast x}
\newcommand{\tvminusy}{\mathcal{T}\backslash v^{-} \ast y}

\newcommand{\simp}{\mathrm{Int}} 

\newcommand{\longmapsfrom}{\mathrel{\reflectbox{\ensuremath{\longmapsto}}}}

\newcommand{\preim}[1]{\widetilde{#1}} 
\newcommand{\cont}[1]{\overline{#1}}
\newcommand{\obs}[1]{{#1}^{o}} 
\newcommand{\xvy}{[x \rightarrow v \leftarrow y]} 
\newcommand{\contm}{[m - 1 \leftarrow m]} 
\newcommand{\conto}{[1 \rightarrow 2]} 

\newcommand{\conv}{\mathrm{conv}} 
\newcommand{\ve}[1]{\mathbf{#1}} 
\newcommand{\ip}[2]{\langle #1, #2 \rangle} 

\newcommand{\st}{:} 
\newcommand{\tup}[1]{\{#1\}} 
\newcommand{\stup}{\{} 
\newcommand{\etup}{\}} 
\newcommand{\set}[1]{\{\,#1\,\}} 
\newcommand{\bigset}[1]{\bigl\{#1\bigr\}} 
\newcommand{\adjset}[1]{\left\lbrace\, #1 \,\right\rbrace} 

\newcommand{\creal}[2]{\mathfrak{C}(#1, #2)}
\newcommand{\polytope}{\mathfrak{P}} 
\newcommand{\realisation}{|-|\colon \vertices \to \mathbb{R}^{n}}
\newcommand{\combpoly}{P} 
\newcommand{\geomtri}{\mathfrak{T}} 
\newcommand{\geomsimp}{|S|} 
\newcommand{\face}{\mathfrak{F}} 
\newcommand{\vertices}{V} 
\newcommand{\vertex}{v} 
\newcommand{\realset}{\mathfrak{X}} 

\newcommand{\fac}{\mathcal{F}}
\newcommand{\crc}{\mathcal{Z}}
\newcommand{\facets}[1]{\fac_{#1}} 
\newcommand{\circuits}[1]{\crc_{#1}} 
\newcommand{\cfacets}[2]{\fac(#1, #2)} 
\newcommand{\ccircuits}[2]{\crc(#1, #2)} 
\newcommand{\lfacets}[2]{\fac^{l}(#1, #2)} 
\newcommand{\ufacets}[2]{\fac^{u}(#1, #2)} 
\newcommand{\vfacets}[2]{\fac_{v}(#1, #2)} 
\newcommand{\vlfacets}[2]{\fac^{l}_{v}(#1, #2)} 
\newcommand{\vufacets}[2]{\fac^{u}_{v}(#1, #2)} 
\newcommand{\vcircuits}[2]{\crc_{v}(#1, #2)} 

\newcommand{\sm}{\sigma} 

\newcommand{\odd}[1]{#1_{o}} 
\newcommand{\even}[1]{#1_{e}} 

\colorlet{cyc}{gray} 
\colorlet{secclr}{purple} 

\renewcommand{\S}{S} 
\newcommand{\s}{s} 
\newcommand{\R}{R} 
\newcommand{\sR}{r} 
\newcommand{\T}{Q} 
\newcommand{\Tb}{Q'} 

\newcommand{\F}{F} 
\newcommand{\G}{G} 
\newcommand{\U}{H} 
\newcommand{\h}{h} 

\newcommand{\SU}{E} 
\newcommand{\su}{e} 

\newcommand{\J}{J} 
\newcommand{\sJ}{j} 
\newcommand{\K}{K} 
\newcommand{\Ja}{L} 

\newcommand{\sQ}{q} 

\newcommand{\aI}{i} 
\newcommand{\bI}{j} 
\newcommand{\cI}{k} 

\newcommand{\Wa}{W} 
\newcommand{\Wb}{W'} 

\newenvironment{proofenum}{\begin{enumerate}[wide, topsep=0pt,  labelindent=\parindent,
]}{\end{enumerate}}

\newenvironment{caseenum}{\begin{enumerate}[wide,  labelindent=\parindent]}{\end{enumerate}}

\usepackage[backend=biber,
			style=alphabetic,
			url=false,
			doi=false,
			isbn=false,
			giveninits=true,
			maxbibnames=99]{biblatex}
\DeclareFieldFormat{title}{#1} 
\renewbibmacro{in:}{} 
\title{The two higher Stasheff--Tamari orders are equal}
\author{Nicholas J. Williams}
\email{williams@ms.u-tokyo.ac.jp}
\address{Graduate School of Mathematical Sciences, The University of Tokyo, 3-8-1 Komaba, Meguro-ku, Tokyo 153-8914, Japan}
\thanks{I would like to thank my PhD supervisor Sibylle Schroll for her support and attention during my PhD studies. Thank you also to Hugh Thomas for interesting conversations. I would finally like to thank JSPS, Osamu Iyama, and the University of Tokyo, where I am currently supported by a JSPS International Short-Term Postdoctoral Research Fellowship. }

\begin{document}

\begin{abstract}
The set of triangulations of a cyclic polytope possesses two \emph{a priori} different partial orders, known as the \emph{higher Stasheff--Tamari orders}. The first of these orders was introduced by Kapranov and Voevodsky, while the second order was introduced by Edelman and Reiner, who also conjectured the two to coincide in 1996. In this paper we prove their conjecture, thereby substantially increasing our understanding of these orders. This result also has ramifications in the representation theory of algebras, as established in previous work of the author. Indeed, it means that the two corresponding orders on tilting modules, cluster-tilting objects and their maximal chains are equal for the higher Auslander algebras of type~$A$.
\end{abstract}

\maketitle

\tableofcontents

\section{Introduction}

One of the principal reasons to study triangulations is their ability to encode pertinent information combinatorially. One example here is the famous work of Gel{\cprime}fand, Kapranov, and Zelevinsky showing how extremal terms of $A$-discriminants are described by regular triangulations of Newton polytopes \cite{gkz-book}. Another example is given by the result of Stanley that the linear extensions of a poset correspond to simplices in a triangulation of its order polytope \cite{stanley_tpp}.

But nowhere is this ability in sharper relief than in the Tamari lattice \cite{friedman_tamari,huang-tamari,tamari_thesis}, a ubiquitous partial order encoded by triangulations of convex polygons. The Tamari lattice is inescapable when considering weak associativity conditions \cite{tamari}, where triangulations of convex polygons correspond to the different possibilities for performing a binary operation on a string. Homotopy associativity of $H$-spaces was studied by Stasheff \cite{stasheff} using the associahedron, a polytope whose $1$-skeleton is the Tamari lattice \cite{tamari_thesis,stasheff_met}. In mathematical physics, weak associativity conditions occur in open string field theory \cite{moore_freudenthal,kk_trees,hikko}, and  in the Biedenharn--Elliott identities \cite{biedenharn_racah,elliott_nuclear}. In algebra, triangulations of convex polygons correspond to clusters in the type~$A$ cluster algebra \cite{fz-y}, which is related to incarnations of the Tamari lattice as a partial order on tilting modules \cite{buan-krause} or torsion classes \cite{thomas_tamari} for the type~$A$ path algebra. The sequence counting the number of objects of the Tamari lattice is the Catalan numbers, which is known to enumerate over two hundred different sequences of combinatorial objects \cite{stanley-catalan}. The extensive reach of the Tamari lattice into different areas of mathematics is exhibited in the Tamari memorial festschrift \cite{tamari-festschrift}.

The first and second higher Stasheff--Tamari orders are two higher-dimensional versions of the Tamari lattice. Their objects are triangulations of cyclic polytopes, which are the higher-dimensional analogues of triangulations of convex polygons. The history of these orders is as follows. In 1991, Kapranov and Voevodsky \cite{kv-poly} defined an order on the set of triangulations of a cyclic polytope, called the \emph{higher Stasheff order}, to give natural examples of strictly ordered $n$-categories produced by a certain iterative construction. In 1996, Edelman and Reiner built upon this work by introducing the two \emph{a priori} different higher Stasheff--Tamari orders. Thomas later proved that the first higher Stasheff--Tamari order of Edelman and Reiner coincided with the higher Stasheff order of Kapranov and Voevodsky \cite{thomas-bst}. Edelman and Reiner further conjectured the two higher Stasheff--Tamari orders to coincide with each other \cite[Conjecture 2.6]{er}, a problem that has remained open since, despite several papers on the orders \cite{err,thomas,thomas-bst,rambau,rs-baues,rr}.

One especially beautiful facet of the higher-dimensional orders is that triangulations of $(n + 1)$-dimensional cyclic polytopes are assembled from maximal chains of triangulations of $n$-dimensional cyclic polytopes in the first higher Stasheff--Tamari order \cite{rambau}. In particular, the objects of the three-dimensional first higher Stasheff--Tamari order correspond to equivalence classes of maximal chains in the Tamari lattice, and the objects of the four-dimensional first higher Stasheff--Tamari order correspond to equivalence classes of maximal chains in the three-dimensional order, and so on. 

Generalisations of, and variations on, the Tamari lattice is a large subject in itself, and includes Tamari lattices in other Dynkin types \cite{thomas_btamari}, Cambrian lattices \cite{reading_cambrian}, lattices of torsion classes of cluster-tilted algebras \cite{g-mc}, $m$-Tamari lattices \cite{bpr_mtam,bfp_intervalno}, $\nu$-Tamari lattices \cite{pv_nu_tamari}, Dyck lattices \cite{knuth_apc4,dfrr_catalan}, generalised Tamari orders \cite{ronco_tamari}, and Grassmann--Tamari orders \cite{ssw}. However, the higher Stasheff--Tamari orders hold a particularly special position amongst these because, as we have seen, they encode higher-dimensional information hidden in the Tamari lattice itself, rather than being only variations on the Tamari lattice. This furthermore shows the virtues of viewing the Tamari lattice in terms of triangulations of convex polygons: it brings out these latent higher-dimensional structures which are obscured by other combinatorial interpretations.

Just as we have seen for triangulations of convex polygons, triangulations of cyclic polytopes describe phenomena across mathematics. Indeed, triangulations of cyclic polytopes are often used to define higher-dimensional analogues of structures that exist for lower-dimensional triangulations. The example of this \emph{par excellence} is the application of triangulations of cyclic polytopes to define higher Segal spaces \cite{dk-segal}, see also \cite{poguntke,djw}. In integrable systems, regular triangulations of cyclic polytopes describe the evolution of a class of solitary waves modelled by the Kadomtsev--Petviashvili equation \cite{dm-h,njw-hbo}, see also \cite{huang_thesis,kk,gpw}. The amplituhedron of Arkani-Hamed and Trnka \cite{amplituhedron} is a cyclic polytope for particular values of its parameters. Here, repeatedly applying the BCFW recursion \cite{bcfw} to compute scattering amplitudes produces a triangulation of this cyclic polytope \cite{what_is_amp}. In algebra, triangulations of cylic polytopes correspond to tilting modules and equivalence classes of ($d$-)maximal green sequences for the higher Auslander algebras of type~$A$ \cite{ot,njw-hst}. Triangulations of cyclic polytopes have also been shown to be in bijection with other combinatorial objects, such as snug partitions \cite{thomas} and persistent graphs \cite{fr}.The higher Stasheff--Tamari orders can be interpreted in terms of these combinatorial objects. A cyclic polytope is called an \emph{alternating polytope} if all of its induced subpolytopes are cyclic; Sturmfels shows that alternating polytopes correspond naturally to totally positive matrices \cite{sturmfels_pos}, which are of significant interest in both pure mathematics and applications \cite{ando,l_tp_intro,post-grass}.

The two higher Stasheff--Tamari orders are quite different in nature and each has its own advantages. The first order is more combinatorial and is defined by means of its covering relations, which are given by ``increasing bistellar flips''. The second order is more geometric and was originally defined by comparing the heights of sections induced by triangulations. However, it was shown in \cite{thomas,njw-hst} how one may define the second order combinatorially. The second order allows direct comparison between triangulations, whereas comparing triangulations in the first order requires one to find a sequence of increasing bistellar flips. On the other hand, the local structure of the second poset is not clear, because the covering relations are not known. It is also easier to compute the entire first poset than to compute the entire second poset. Computing either poset requires computing all the triangulations of a given cyclic polytope. The most efficient algorithm for doing this is to start at the minimal triangulation and iteratively compute increasing bistellar flips \cite{jos-kast}, which is tantamount to computing the first order. To construct the second order then requires additional computations on top of this.

It is clear that whenever the first higher Stasheff--Tamari order holds between a pair of triangulations, then the second order must hold too, as was noted in \cite{er}. This is because if a triangulation~$\mathcal{T}'$ is an increasing bistellar flip of a triangulation~$\mathcal{T}$, then the section of $\mathcal{T}'$ certainly lies above the section of $\mathcal{T}$. But it is not clear whether the first order should hold whenever the second one does. Indeed, an analogous result for the higher Bruhat orders has been known to be false since 1993 \cite{ziegler-bruhat}. \emph{A priori} it might be possible for there to exist a pair of pathological triangulations $\mathcal{T}$ and $\mathcal{T}'$ where the section of $\mathcal{T}$ lay below the section of $\mathcal{T}'$, and yet $\mathcal{T}'$ could not be reached by a sequence of increasing bistellar flips from $\mathcal{T}$.

In this paper, we prove the Edelman--Reiner conjecture that the first higher Stasheff--Tamari order ($\leqslant_{1}$) is equal to the second higher Stasheff--Tamari order~($\leqslant_{2}$).

\begin{theorem}[Theorem~\ref{thm:main_odd} and Theorem~\ref{thm:main_even}]
Let $\mathcal{T}$ and $\mathcal{T}'$ be triangulations of the cyclic polytope $C(m, n)$. Then $\mathcal{T} \leqslant_{1} \mathcal{T}'$ if and only if $\mathcal{T} \leqslant_{2} \mathcal{T}'$.
\end{theorem}

This reveals the remarkable---and somewhat surprising---fact that in order for a sequence of increasing bistellar flips to exist from a triangulation~$\mathcal{T}$ to a triangulation~$\mathcal{T}'$, it suffices for the section of $\mathcal{T}'$ to lie above the section of $\mathcal{T}$. This allows triangulations to be compared directly in the first order, and thus substantially increases the ease of working with this poset. An application of this result is that the orders considered in \cite{njw-hst} on tilting modules and their maximal chains, and cluster-tilting objects and ($d$-)maximal green sequences coincide for the higher Auslander algebras of type~$A$. This is an intriguing fact which raises the prospect that this might be true more generally in higher Auslander--Reiten theory. However, there are many obstacles to such a proof, principally, the lack of mutability present in higher Auslander--Reiten theory \cite{ot}.

Our proof of the conjecture is inductive and combinatorial, drawing upon the results of \cite{njw-hst}. The main difficulty in proving the conjecture is that as the dimension of the cyclic polytope grows, increasing bistellar flips become scarce. The key insight of the proof is that one can find increasing bistellar flips inductively by contracting triangulations, because when one reverses the contraction the small polytopes in which the flips occur remain small enough to find a new flip. The difficult step in the proof is then showing that the increasing bistellar flip one has found respects the second order, which allows one to build a chain of flips between the two sections. The details of this step differ between even and odd dimensions, and, accordingly, we deal with the proofs separately for the two different parities. Understanding how subpolytopes of triangulations behave under expansion requires extending the theory of contracting and expanding triangulations from \cite[Lemma~4.7(i)]{rs-baues} to arbitrary vertices, which is of independent interest. Indeed, this theory is extremely useful in proving new descriptions of triangulations of cyclic polytopes \cite{thomas,ot,njw-hst}.

This paper is structured as follows. In Section~\ref{sect:background} we give background on the higher Stasheff--Tamari orders. In Section~\ref{sect:main} we prove the Edelman--Reiner conjecture. At the beginning of this section we give an outline of our proof that the higher Stasheff--Tamari orders are equal. After proving some preliminary lemmas, we split the proof of the conjecture into two cases, depending upon whether the cyclic polytope is odd-dimensional or even-dimensional. 
In Section~\ref{sect:expansion} we generalise the theory from \cite[Lemma~4.7(i)]{rs-baues} of contracting and expanding triangulations of cyclic polytopes to arbitrary vertices. This allows us to prove technical results which are needed for the proof in Section~\ref{sect:main}.

\section{Background}\label{sect:background}

We first declare some notation. We use $[m]$ to denote the set $\{1, 2, \dots, m\}$. Similarly, we write $[m,n]$ for $\set{i \in \mathbb{N} \st m \leqslant i \leqslant n}$ and call subsets of this form \emph{intervals}. By $\binom{[m]}{k}$ we mean the set of subsets of $[m]$ of size $k$. When we display the elements of a subset of $[m]$, we shall always display the elements in order. Hence, if we write $S = \tup{a, b, c, \dots, x, y, z}$, we always mean that $a < b < c < \dots < x < y < z$. Furthermore, if $A \in \binom{[m]}{k + 1}$, then, unless indicated otherwise, we shall find it convenient to denote the elements of $A$ by $A = \tup{a_{0}, a_{1}, \dots, a_{k}}$. The same applies to other letters of the alphabet: the upper-case letter denotes the subset; the lower-case letter is used for the elements, which are ordered according to their index starting from 0. In an effort to make the notation lighter, we often omit braces around sets, writing $A \cup x$ for $A \cup \{x\}$ and $A \setminus x$ for $A \setminus \{x\}$.

\subsection{Cyclic polytopes and their triangulations}

Our framework for cyclic polytopes and their triangulations is based on \cite{rambau} and maintains a sharp distinction between the combinatorial and the geometric.

\subsubsection{Convex polytopes}\label{sect:back:cyc:conv}

A subset $\realset \subset \mathbb{R}^{n}$ is \emph{convex} if for any $x,x' \in \realset$, the line segment connecting $x$ and $x'$ is contained in $\realset$. The \emph{convex hull} $\mathrm{conv}(\realset)$ of $\realset$ is the smallest convex set containing $\realset$ or, equivalently, the intersection of all convex sets containing $\realset$.

Let $\vertices \subseteq \mathbb{Z}_{>0}$ be a finite set and $|-|\colon \vertices \to \mathbb{R}^{n}$ be an injective function, which we call the \emph{geometric realisation}. We extend the notation to subsets of $\vertices$ by setting $|A| = \conv\set{|a| \st a \in A}$. We let $\polytope = |\vertices|$ and suppose that the affine span of $\polytope$ is~$\mathbb{R}^{n}$. A subset $\polytope \subseteq \mathbb{R}^{n}$ of this form is called a (\emph{geometric}) \emph{convex polytope}. Given $\vertices' \subseteq V$, we say that $|\vertices'|$ is a \emph{subpolytope} of $|\vertices|$.

A \emph{face} of a $\polytope$ is a subset on which some linear functional is maximised. That is, $\face \subseteq \polytope$ is a face of $\polytope$ if there is a vector $\ve{a} \in \mathbb{R}^{n}$ such that \[\face = \set{\ve{x} \in \polytope \st \ip{\ve{a}}{\ve{x}} \geqslant \ip{\ve{a}}{\ve{y}},\, \forall \ve{y} \in \polytope},\] where `$\ip{-}{-}$' denotes the standard inner product. A (\emph{geometric}) \emph{facet} of $\polytope$ is a face of codimension one. A (\emph{combinatorial}) \emph{facet} of $\polytope$ is a subset $F \subseteq \vertices$ such that $|F|$ is a geometric facet of $\polytope$. The set of faces of a polytope $\polytope$, along with the empty set, forms a lattice under inclusion, which is known as the \emph{face lattice}.

Let $\vertex \in \vertices$ be such that $|\vertex|$ is the face of $\polytope$ given by maximising a functional $\ip{\ve{a}}{-}$. Further, let $\epsilon > 0$ be sufficiently small that, for all $w \in \vertices \setminus \vertex$, we have that $\ip{\ve{a}}{|w|} < \ip{\ve{a}}{|\vertex|} - \epsilon$. The \emph{vertex figure} of $\polytope$ at $\vertex$ is then the intersection \[\polytope\backslash \vertex := \polytope \cap \set{\ve{x} \in \mathbb{R}^{n} \st \ip{\ve{a}}{\ve{x}} = \ip{\ve{a}}{|\vertex|} - \epsilon},\] that is, the intersection of $\polytope$ with the hyperplane $\ip{\ve{a}}{\ve{x}} = \ip{\ve{a}}{|\vertex|} - \epsilon$.

A \emph{circuit} of a polytope $\polytope$ realised geometrically via $\realisation$ is a pair, $(Z_{+}, Z_{-})$, of disjoint subsets of $\vertices$ which are inclusion-minimal with the property that $|Z_{+}| \cap |Z_{-}| \neq \emptyset$. In this case, $|Z_{+}|$ and $|Z_{-}|$ intersect in a unique point.

The facets and circuits of the polytope $\polytope$ realised geometrically via $\realisation$ comprise the combinatorial data that we are interested in. If we let $\facets{\polytope}$ and $\circuits{\polytope}$ be respectively the set of combinatorial facets of $\polytope$ and the set of combinatorial circuits of $\polytope$, then we say that the triple $\combpoly = (\vertices, \facets{\polytope}, \circuits{\polytope})$ is a \emph{combinatorial polytope}.

\begin{remark}
Note that there might exist $\vertex \in \vertices$ such that $|\vertex|$ is not a face of $\polytope$, since we may have $|\vertex| \in |\vertices\setminus\vertex|$. Allowing such elements of $\vertices$ is necessary for considering one-dimensional cyclic polytopes. Hence, strictly, the data we consider comprise a point configuration rather than a polytope.
\end{remark}

\subsubsection{Cyclic polytopes}\label{sect:back:cyc:cyc}

Cyclic polytopes are the higher-dimensional analogues of convex polygons. General introductions to this class of polytopes can be found in \cite[Lecture 0]{ziegler} and \cite[Section 4.7]{gruenbaum}. Gr\"unbaum writes that the construction of cyclic polytopes in current use is due to Gale \cite{gale} and Klee \cite{klee}, and that they were introduced and studied in the 1950s by Gale \cite{gale-abstract} and Motzkin \cite{motzkin}. The earlier work of Carath\'eodory \cite{car1907,car1911} is related, but the convex bodies studied in these papers are not cyclic polytopes: they are the continuous analogues of even-dimensional cyclic polytopes.

\begin{definition}\label{def:cyc_poly}
The \emph{cyclic polytope} $\creal{\vertices}{n}$ is the polytope with geometric realisation
\begin{align*}
|-|_{n}\colon \vertices &\to \mathbb{R}^{n}\\
\vertex &\mapsto |\vertex|_{n} =  p_{n}(t_{\vertex}) := (t_{\vertex}, t_{\vertex}^{2}, \dots, t_{\vertex}^{n}),
\end{align*}
where $\{t_{\vertex_{0}}, t_{\vertex_{1}}, \dots, t_{\vertex_{k}}\} \subset \mathbb{R}$ and $k + 1 = \# \vertices$. (Recall our convention that $\vertices = \{\vertex_{0}, \vertex_{1}, \dots, \vertex_{k}\}$ and that by writing $\{t_{\vertex_{0}}, t_{\vertex_{1}}, \dots, t_{\vertex_{k}}\}$, we indicate that $t_{\vertex_{0}} < t_{\vertex_{1}} < \dots < t_{\vertex_{k}}$.)

When the dimension of the geometric realisation is clear from the context, we will drop the subscript and write $|-|$ instead of $|-|_{n}$. In the case where $\vertices = [m]$, we write $\creal{m}{n} := \creal{[m]}{n}$. The precise values $t_{i}$ do not affect the combinatorial properties of the cyclic polytope, so for simplicity we set $t_{i} = i$. The curve defined by $p_{n}(t):=(t, t^{2}, \dots , t^{n}) \subset \mathbb{R}^{n}$ is called the \emph{moment curve}.
\end{definition}

\begin{figure}
\caption{Cyclic polytopes \cite[Figure 2]{er}}
\[
\begin{tikzpicture}
\coordinate(11) at (-4,0);
\coordinate(12) at (-2.4,0);
\coordinate(13) at (-0.8,0);
\coordinate(14) at (0.8,0);
\coordinate(15) at (2.4,0);
\coordinate(16) at (4,0);

\coordinate(21) at (-4,4.5);
\coordinate(22) at (-2.4,2.9);
\coordinate(23) at (-0.8,2.1);
\coordinate(24) at (0.8,2.1);
\coordinate(25) at (2.4,2.9);
\coordinate(26) at (4,4.5);

\coordinate(31) at (-4,9);
\coordinate(32) at (-2.4,7.4);
\coordinate(33) at (-0.8,6.6);
\coordinate(34) at (0.8,6.6);
\coordinate(35) at (2.4,7.4);
\coordinate(36) at (4,9);

\draw[->] (0, 6) -- (0,5.1);
\draw[->] (0, 1.5) -- (0,0.6);

\node(1) at (-6,0){$\creal{6}{3}$};
\node(2) at (-6,4.5){$\creal{6}{2}$};
\node(3) at (-6,9){$\creal{6}{1}$};

\draw[cyc] (11) -- (12) -- (13) -- (14) -- (15) -- (16);

\draw[cyc] (21) -- (22) -- (23) -- (24) -- (25) -- (26) -- (21);

\draw[dotted] (31) -- (33);
\draw[dotted] (31) -- (34);
\draw[dotted] (31) -- (35);
\draw[fill=cyc, fill opacity = 0.3] (31) -- (32) -- (33) -- (34) -- (35) -- (36) -- (31);
\draw (36) -- (34);
\draw (36) -- (33);
\draw (36) -- (32);

\node at (11){$\bullet$};
\node at (12){$\bullet$};
\node at (13){$\bullet$};
\node at (14){$\bullet$};
\node at (15){$\bullet$};
\node at (16){$\bullet$};
\node at (21){$\bullet$};
\node at (22){$\bullet$};
\node at (23){$\bullet$};
\node at (24){$\bullet$};
\node at (25){$\bullet$};
\node at (26){$\bullet$};
\node at (31){$\bullet$};
\node at (32){$\bullet$};
\node at (33){$\bullet$};
\node at (34){$\bullet$};
\node at (35){$\bullet$};
\node at (36){$\bullet$};
\end{tikzpicture}
\]
\end{figure}
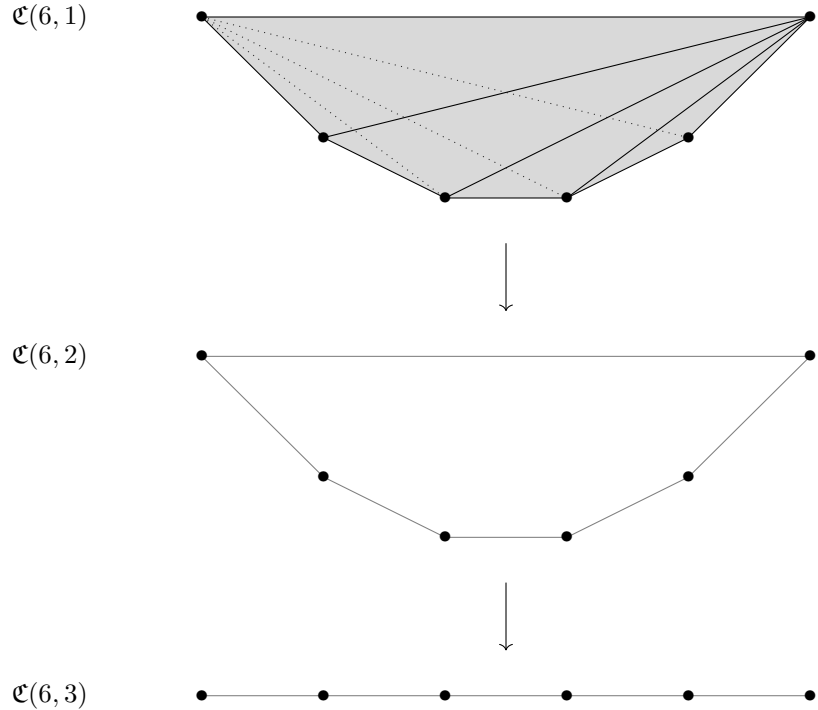

The facets of $\creal{m}{n}$ come in two different types.

\begin{definition}[\cite{er}]
A facet $\face$ of $\creal{m}{n}$ is an \emph{upper facet} if for any $\ve{a} \in \mathbb{R}^{n}$ such that $\ip{\ve{a}}{-}$ is maximised on $\face$, we have that $\ve{a}_{n} > 0$, where $\ve{a}_{n}$ is the $n$-th coordinate of $\ve{a}$. Dually, a facet $\face$ of $\creal{m, n}$ is a \emph{lower facet} if $\ve{a}_{n}$ is negative for any $\ve{a}$ such that $\face$ maximises $\ip{\ve{a}}{-}$.
\end{definition}

Equivalently, a facet $\face$ of $\creal{m}{n}$ is an upper (lower) facet if any normal vector to $\face$ which points out of the polytope has a positive (negative) $n$-th coordinate. Or, more informally, $\face$ is an upper (lower) facet if it can be seen from a very large positive (negative) $n$-th coordinate. Upper and lower facets of cyclic polytopes can be characterised combinatorially using the following notions.

\begin{definition}
Given a subset $F \subset \vertices$, we say that an element $v \in \vertices \setminus F$ is an \emph{even gap} in $F$ if $\#\set{ x \in F \st x > v }$ is even. Otherwise, it is an \emph{odd gap}. A subset $F \subset \vertices$ is \emph{even} if every $v \in \vertices \setminus F$ is an even gap. A subset $F \subset \vertices$ is \emph{odd} if every $v \in [m] \setminus F$ is an odd gap.
\end{definition}

\begin{theorem}[{Gale's Evenness Criterion, \cite[Theorem~3]{gale},\cite[Lemma~2.3]{er}}]
Given a $n$-subset $F \subset \vertices$, we have that $|F|$ is an upper facet of $\creal{\vertices}{n}$ if and only if $F$ is an odd subset, and that $|F|$ is a lower facet of $\creal{\vertices}{n}$ if and only if $F$ is an even subset.
\end{theorem}

We write
\begin{align*}
\lfacets{\vertices}{n} &:= \set{ F \subseteq \vertices \st |F|_{n} \text{ is a lower facet of }|\vertices|_{n}}, \text{ and}\\
\ufacets{\vertices}{n} &:= \set{ F \subseteq \vertices \st |F|_{n} \text{ is an upper facet of }|\vertices|_{n}}.
\end{align*}

One may likewise characterise the circuits of cyclic polytopes combinatorially.

\begin{definition}[{\cite[Definition~2.2]{ot},\cite{njw-hst}}]
If $A, B \subseteq \vertices$ are $(d + 1)$-subsets, then we say that $A$ \emph{intertwines} $B$, and write $A \wr B$, if \[a_{0} < b_{0} < a_{1} < b_{1} < \dots < a_{d} < b_{d}.\]

If either $A \wr B$ or $B \wr A$, then we say that $A$ and $B$ are \emph{intertwining}. That is, we use `are intertwining' to refer to the symmetric closure of the relation `intertwines'. A collection of $(d + 1)$-subsets is called \emph{non-intertwining} if no pair of the elements are intertwining.

If $A$ is a $d$-subset and $B$ is a $(d + 1)$-subset, then we also say that $A$ \emph{intertwines} $B$, and write $A \wr B$, if \[b_{0} < a_{0} < b_{1} < \dots < a_{d - 1} < b_{d}.\]
\end{definition}

\begin{theorem}[\cite{breen}]
The circuits of $\creal{\vertices}{n}$ are the pairs $(A, B)$ and $(B, A)$ such that $A$ is a $(\lfloor \frac{n}{2} \rfloor+1)$-subset, $B$ is a $(\lceil \frac{n}{2} \rceil +1)$-subset, and $A$ intertwines~$B$.
\end{theorem}

This characterisation of the circuits is well-known due to the description of the oriented matroid given by a cyclic polytope \cite{b-lv,sturmfels,cd}.

These combinatorial characterisations of upper and lower facets and circuits show that these notions are independent of the particular geometric realisation of $\creal{\vertices}{n}$. Hence, we make the following definition.

\begin{definition}\label{def:comb_cyc_poly}
We will write $\cfacets{\vertices}{n}$ for the combinatorial facets of $\creal{\vertices}{n}$, $\ccircuits{\vertices}{n}$ for the combinatorial circuits of $\creal{\vertices}{n}$, and $C(\vertices, n)$ for the \emph{combinatorial cyclic polytope} consisting of the triple $(\vertices, \cfacets{\vertices}{n}, \ccircuits{\vertices}{n})$.
\end{definition}

Facets of $C(\vertices, n)$ will also be designated as either upper facets or lower facets, as dictated by Gale's Evenness Criterion. As for geometric cyclic polytopes, we write $C(m, n) := C([m], n)$. If $\#\vertices = m$, then we say that $C(\vertices, n)$ and $C(m, n)$ are \emph{combinatorially equivalent}---they only differ by the labels of the vertices. More generally, two (geometric or combinatorial) polytopes are combinatorially equivalent if they have isomorphic face lattices.

\begin{remark}
In the literature, the term `cyclic polytope' is sometimes used more widely to refer to a geometric polytope which is combinatorially equivalent to a cyclic polytope $\creal{m}{n}$. Here we use the term `cyclic polytope' in the narrower sense in which we have defined it above: the convex hull of a set of points on the moment curve.

Another term that appears in the literature is `alternating polytope', which refers to a polytope $\polytope$ which is combinatorially equivalent to a cyclic polytope $\creal{m}{n}$ in such a way that this combinatorial equivalence restricts to a combinatorial equivalence between each of the corresponding subpolytopes of $\polytope$ and $\creal{m}{n}$. It is clear from the definitions that cyclic polytopes in our sense are also alternating polytopes in this sense. Indeed, when we consider subpolytopes of our cyclic polytopes, we will use the fact that these are cyclic polytopes under the vertex labelling induced from the larger polytope. Shemer proves that every polytope combinatorially equivalent to an even-dimensional cyclic polytope is an alternating polytope \cite{shemer_np}, but this is known not to be true in general---see, for instance, \cite{sturmfels_pos}. See also \cite{bk_subpoly}.
\end{remark}

\subsubsection{Triangulations}\label{sect:back:cyc:tri}

We now explain our framework for triangulations. We maintain our set-up from Section~\ref{sect:back:cyc:conv}, where $\vertices \subseteq \mathbb{Z}_{>0}$ is a finite subset, with $|-|\colon \vertices \to \mathbb{R}^{n}$ a geometric realisation giving a geometric polytope $\polytope = |\vertices|$, with corresponding combinatorial polytope $\combpoly = (\vertices, \facets{\polytope}, \circuits{\polytope})$.

A \emph{combinatorial $n$-simplex} in $\vertices$ is a $(n + 1)$-subset $\S \subseteq \vertices$. The \emph{$k$-faces} of $\S$ are the subsets of $\S$ of size $k + 1$. An \emph{abstract simplicial complex} is a set $\mathcal{A}$ of combinatorial simplices in $[m]$ such that if $\S,\S' \in \mathcal{A}$, then $\S \not\subseteq \S'$. The $k$-simplices of $\mathcal{A}$ are the $k$-faces of elements of $\mathcal{A}$.  An abstract simplicial complex $\mathcal{A}'$ is an \emph{abstract simplicial subcomplex} of $\mathcal{A}$ if every simplex of $\mathcal{A}'$ is a face of a simplex of~$\mathcal{A}$. Hence, we consider abstract simplicial complexes in terms of maximal simplices. 

Given a $(n + 1)$-simplex $\S \subseteq \vertices$, if $\set{|s| \st s \in \S}$ is an affinely independent set, then $|\S|$ is a \emph{geometric $n$-simplex}. A collection $\mathcal{G}$ of geometric simplices is a \emph{geometric simplicial complex} if \[|\S \cap \R| = |\S| \cap |\R|\] for all $|\S|, |\R| \in \mathcal{G}$, and if there exist no $|\S|, |\R| \in \mathcal{G}$ such that $|\S|$ is a face of $|\R|$. \emph{Geometric simplicial subcomplexes} are defined analogously to abstract simplicial subcomplexes.

\begin{definition}\label{def:triang}
A (\emph{geometric}) \emph{triangulation} of the geometric polytope $\polytope$ is a geometric simplicial complex $\geomtri$ such that $\polytope = \bigcup_{\geomsimp \in \geomtri} \geomsimp$.

A (\emph{combinatorial}) \emph{triangulation} of the combinatorial polytope $\combpoly$ is an abstract simplicial complex $\mathcal{T}$ such that
\begin{itemize}
\item for all $\S \in \mathcal{T}$ and all facets $\F$ of $\S$, we either have that $\F$ is contained in a facet of $\combpoly$, or there exists $\R \in \mathcal{T}\setminus \{\S\}$ such that $\F \subset \R$,
\item there is no circuit $(Z_{+}, Z_{-})$ of $\combpoly$ such that $Z_{+} \subseteq \S$ and $Z_{-} \subseteq \R$ for $\S, \R \in \mathcal{T}$.
\end{itemize}
We use $|\mathcal{T}|$ to refer to the geometric simplicial complex corresponding to $\mathcal{T}$.
\end{definition}

\begin{proposition}[{\cite[Proposition~2.2]{rambau}}]
Given an abstract simplicial complex $\mathcal{T}$, we have that $\mathcal{T}$ is a combinatorial triangulation of $P$ if and only if $|\mathcal{T}|$ is a geometric triangulation of $\polytope$.
\end{proposition}

In this paper we are usually concerned with combinatorial triangulations, but sometimes we shall need to consider geometric triangulations.

One can use the descriptions of the facets and circuits of $C(m, n)$ to determine whether or not a collection $\mathcal{T}$ of $n$-simplices gives a triangulation of $C(m, n)$. We denote the set of triangulations of the cyclic polytope $C(m, n)$ by $\mathsf{S}(m, n)$. There are two triangulations of $C(m, n)$ which are of particular note. Namely, the lower facets $\lfacets{[m]}{n + 1}$ of $C(m, n + 1)$ give a triangulation of $C(m, n)$, which is known as the \emph{lower triangulation}. Similarly, $\ufacets{[m]}{n + 1}$ gives a triangulation of $C(m, n)$, which is known as the \emph{upper triangulation}.

Given a triangulation~$\mathcal{T}$ of $C(m, n)$ and $\U \subseteq [m]$, we say that $C(\U, n)$ is a \emph{subpolytope} of $\mathcal{T}$ if the facets $\cfacets{\U}{n}$ of $C(\U, n)$ are a simplicial subcomplex of $\mathcal{T}$. Equivalently, we have that $C(\U, n)$ is a subpolytope of $\mathcal{T}$ if and only if $\set{\S \in \mathcal{T} \st \S \subseteq \U}$ is a triangulation of $C(\U, n)$. We refer to this triangulation as the \emph{induced triangulation} of $C(\U, n)$.

\subsection{The higher Stasheff--Tamari orders}

We now come to the definitions of the two higher Stasheff--Tamari orders. General introductions to these orders can be found in \cite{rr} and \cite[Section 6.1]{lrs}.

We make some observations to motivate the definition of the first higher Stasheff--Tamari order. Given $\U \in \binom{[m]}{n+2}$, we have that $C(\U, n)$ has only two triangulations, namely: the lower triangulation and the upper triangulation. For example, if $n = 2$, then $\creal{\U}{n}$ is a quadrilateral, with the two possible triangulations given by a choice of diagonal.

\begin{definition}[{\cite{er}}]
Suppose that $C(\U, n)$ is a subpolytope of $\mathcal{T}$ where  $\U \in \binom{[m]}{n + 2}$ and the induced triangulation of $C(\U, n)$ is the lower triangulation.

Let $\mathcal{T}' = (\mathcal{T}\setminus\lfacets{\U}{n + 1}) \cup \ufacets{\U}{n + 1}$; that is, $\mathcal{T}'$ results from replacing the induced triangulation of $C(\U, n)$ with the upper triangulation. Then $\mathcal{T}'$ is also a triangulation of $C(m, n)$ and we say that $\mathcal{T}'$ is an \emph{increasing bistellar flip} of $\mathcal{T}$.

Dually, we say that $\mathcal{T}$ is a \emph{decreasing bistellar flip} of $\mathcal{T}'$. We often simply say \emph{increasing flip} or \emph{decreasing flip}.
\end{definition}

Bistellar flips are also known as ``Pachner moves'' after \cite{pachner}.

\begin{definition}[{\cite{kv-poly,er}}]
The \emph{first higher Stasheff--Tamari order} is the partial order on triangulations of $C(m, n)$ with covering relations such that $\mathcal{T} \lessdot_{1} \mathcal{T}'$ if and only if $\mathcal{T}'$ is an increasing bistellar flip of $\mathcal{T}$. We write $\mathsf{S}_{1}(m,n)$ for the poset on $\mathsf{S}(m,n)$ this gives and $\leqslant_{1}$ for the partial order itself.
\end{definition}

We now define the second higher Stasheff--Tamari order. Unlike the first higher Stasheff--Tamari order, this order is defined using the geometric realisation of the cyclic polytope, although the order itself is independent of the geometric realisation. In Section~\ref{sect:back:comb_char} we shall explain the entirely combinatorial characterisation of the second higher Stasheff--Tamari order from \cite{njw-hst}.

Every triangulation $|\mathcal{T}|$ of $\creal{m}{n}$ determines a unique piecewise-linear section \[\sm_{|\mathcal{T}|}\colon \creal{m}{n}\rightarrow \creal{m}{n+1}\] of $\creal{m}{n +1}$ by sending each $n$-simplex $|\S|_{n}$ of $|\mathcal{T}|$ to $|\S|_{n+1}$ in $\creal{m}{n+1}$, in the natural way, recalling the notation $|S|_{n}$ and $|S|_{n + 1}$ from Definition~\ref{def:cyc_poly}.

\begin{definition}[{\cite{er}}]
The \emph{second higher Stasheff--Tamari order} on $\mathsf{S}(m,n)$ is defined by \[ \mathcal{T} \leqslant_{2} \mathcal{T}' \iff \sm_{|\mathcal{T}|}(x)_{n+1} \leqslant \sm_{|\mathcal{T}'|}(x)_{n+1} \quad \forall x \in \creal{m}{n},\] where $\sm_{|\mathcal{T}|}(x)_{n + 1}$ denotes the $(n + 1)$-th coordinate of the point $\sm_{|\mathcal{T}|}(x)$. We write $\mathsf{S}_{2}(m, n)$ for the poset on $\mathsf{S}(m,n)$ this gives.
\end{definition}

Given triangulations $\mathcal{T}\in \mathsf{S}(m, n)$ and $\mathcal{T}' \in \mathsf{S}(m,n+1)$, we say that $\mathcal{T}$ is a \emph{section} of $\mathcal{T}'$ if $\mathcal{T}$ is contained in $\mathcal{T}'$ as a simplicial subcomplex.

\subsection{Combinatorial characterisation}\label{sect:back:comb_char}

We now explain the results of \cite{njw-hst}, where combinatorial interpretations were given of both higher Stasheff--Tamari orders in the same framework. This makes comparison between the orders much easier. It is these interpretations of the higher Stasheff--Tamari orders that we use in this paper to prove that the orders are equivalent.

The combinatorial interpretations of the orders from \cite{njw-hst} require more sophisticated combinatorial descriptions of triangulations of cyclic polytopes than that of Section~\ref{sect:back:cyc:tri}.

\begin{definition}
We call a simplex $A \subseteq [m]$ \emph{an internal simplex of $C(m, n)$} if $A$ does not lie within any facet of $C(m, n)$.
\end{definition}

It follows from \cite{dey} that a triangulation of $C(m, n)$ is determined by its internal $\floor{n/2}$-simplices, as explained for even dimensions in \cite[Lemma~2.15]{ot} and for odd dimensions in \cite[Lemma~4.4]{njw-hst}. We can thus combinatorially describe triangulations of cyclic polytopes by characterising when a set of $d$-simplices in $[m]$ is the set of internal $d$-simplices of a triangulation of $C(m, 2d)$ or $C(m, 2d + 1)$.

In order to do this, we first need to know when $A$ is an internal $\floor{n/2}$-simplex of $C(m, n)$.

\begin{proposition}[{\cite[Lemma~2.1(3)]{ot} and \cite[Lemma~4.2]{njw-hst}}]
The internal $\floor{n/2}$-simplices of $C(m, n)$ are described as follows.
\begin{enumerate}
\item In even dimensions, $A$ is an internal $d$-simplex of $C(m,2d)$ if and only if
\begin{equation}
A \in \nonconsec{m}{d} := \adjset{B \in \binom{[m]}{d + 1} \st b_{i} \leqslant b_{i + 1} - 2 \: \forall i \in [d], \text{ and } b_{d} \leqslant b_{0} + m - 2}.\label{eq:even_int}
\end{equation}
\item In odd dimensions, $A$ is an internal $d$-simplex of $C(m, 2d + 1)$ if and only if
\begin{equation}
A \in \chainset{m}{d}:=\set{B \in \nonconsec{m}{d} \st b_{0} \neq 1, \, b_{d} \neq m}.\label{eq:odd_int}
\end{equation}
\end{enumerate}
\end{proposition}

These two facts both follow from applying Gale's Evenness Criterion. Given a triangulation~$\mathcal{T}$ of $C(m, 2d)$ or $C(m, 2d + 1)$, we write $\simp(\mathcal{T})$ for its set of internal $d$-simplices.

Even-dimensional triangulations may be described combinatorially by considering their sets of internal $d$-simplices.

\begin{theorem}[{\cite[Theorem~2.3 and Theorem 2.4]{ot}}]\label{thm:tri_char_even}
Given $\mathbf{X} \subseteq \nonconsec{m}{d}$, we have that $\mathbf{X} = \simp(\mathcal{T})$ for some triangulation~$\mathcal{T}$ of $C(m, 2d)$ if and only if $\# \mathbf{X} = \binom{m - d - 2}{d}$ and $\mathbf{X}$ is non-intertwining .
\end{theorem}


In odd dimensions we require the following properties to characterise triangulations combinatorially.

\begin{definition}[{\cite[Definition~4.8]{njw-hst}}]\label{def:support}
Let $\mathbf{X} \subseteq \chainset{m}{d}$ and $A \in \mathbf{X}$. A $(d - 1)$-simplex $\SU$ is called a \emph{support} for $A$ if $\SU \wr A$ and, for any internal $d$-simplex $B \in \chainset{m}{d}$ such that $B \subseteq A \cup \SU$, we have that $B \in \mathbf{X}$. We say that $\mathbf{X}$ is \emph{supporting} if every $A \in \mathbf{X}$ has a support.
\end{definition}

\begin{definition}[{\cite[Definition~4.10]{njw-hst}}]\label{def:bridge}
We say that $\mathbf{X} \subseteq \chainset{m}{d}$ is \emph{bridging} if whenever
\begin{align*}
\tup{\sQ_{0}, \sQ_{1}, \dots, \sQ_{i-1}, a_{i}, a_{i + 1}, \dots, a_{j}, \sQ_{j+1}, \sQ_{j+2}, \dots, \sQ_{d}} &\in \mathbf{X}, \text{ and}\\
\tup{\sQ_{0}, \sQ_{1}, \dots, \sQ_{i-1}, b_{i}, b_{i + 1}, \dots, b_{j}, \sQ_{j+1}, \sQ_{j + 2}, \dots, \sQ_{d}} &\in \mathbf{X},
\end{align*}
where possibly $i = 0$ or $j = d$, or both, such that $\tup{a_{i}, a_{i + 1}, \dots, a_{j}} \wr \tup{b_{i}, b_{i + 1}, \dots, b_{j}}$, we have that \[\tup{\sQ_{0}, \sQ_{1}, \dots, \sQ_{i-1}, a_{i}, a_{i + 1}, \dots, a_{k-1}, b_{k}, b_{k + 1}, \dots, b_{j}, \sQ_{j + 1}, \sQ_{j + 2}, \dots, \sQ_{d}} \in \mathbf{X}\] for all $i \leqslant k \leqslant j + 1$.
\end{definition}

The following result then characterises odd-dimensional triangulations combinatorially.

\begin{theorem}[{\cite[Theorem~4.1]{njw-hst}}]\label{thm:tri_char_odd}
Given $\mathbf{X} \subseteq \chainset{m}{d}$, we have that $\mathbf{X} = \simp(\mathcal{T})$ for some triangulation~$\mathcal{T}$ of $C(m, 2d + 1)$ if and only if $\mathbf{X}$ is supporting and bridging.
\end{theorem}

Building on this, one can characterise the higher Stasheff--Tamari orders combinatorially in terms of the sets $\simp(\mathcal{T})$. 

\begin{theorem}[{\cite[Theorem~3.3 and Theorem 3.6]{njw-hst}}]\label{thm:hst_char_even}
Given $\mathcal{T}, \mathcal{T}' \in \mathsf{S}(m,2d)$, we have that
\begin{itemize}
\item $\mathcal{T} \lessdot_{1} \mathcal{T}'$ if and only if $\simp(\mathcal{T}) = \mathcal{R} \cup \{ A\}$ and $\simp(\mathcal{T}')= \mathcal{R} \cup \{ B\}$, where $A \wr B$;
\item $\mathcal{T} \leqslant_{2} \mathcal{T}'$ if and only if for every $A \in \simp(\mathcal{T})$, there is no $B \in \simp(\mathcal{T}')$ such that $B \wr A$.
\end{itemize}
\end{theorem}

\begin{theorem}[{\cite[Theorem~5.5 and Corollary 5.14]{njw-hst}}]\label{thm:hst_char_odd}
Given $\mathcal{T}, \mathcal{T}' \in \mathsf{S}(m, 2d + 1)$, we have that
\begin{itemize}
\item $\mathcal{T} \lessdot_{1} \mathcal{T}'$ if and only if $\simp(\mathcal{T}') = \simp(\mathcal{T})\setminus\{A\}$ for some $A \in \simp(\mathcal{T})$;
\item $\mathcal{T} \leqslant_{2} \mathcal{T}'$ if and only if $\simp(\mathcal{T}) \supseteq \simp(\mathcal{T}')$.
\end{itemize}
\end{theorem}

\begin{remark}
It is useful to think of increasing bistellar flips in these terms. That is, in even dimensions, given triangulations $\mathcal{T}, \mathcal{T}'$ with $\simp(\mathcal{T}) = \mathcal{R} \cup \{ A\}$  and $\simp(\mathcal{T}')= \mathcal{R} \cup \{B\}$  and $A \wr B$, the increasing bistellar flip replaces the $d$-simplex $A$ with the $d$-simplex $B$ inside $C(A \cup B, 2d)$. This is because $A$ is the intersection of $\lfacets{A \cup B}{2d + 1}$ and $B$ is the intersection of $\ufacets{A \cup B}{2d + 1}$. Here we say that $A$ is a \emph{mutable} $d$-simplex in $\mathcal{T}$. For odd dimensions, suppose that we have $\mathcal{T}, \mathcal{T}'$ with $\simp(\mathcal{T}')=\simp(\mathcal{T})\setminus\{A\}$ for some $A \in \simp(\mathcal{T})$. Then there exists a $(d + 1)$-simplex $B$ with $A \wr B$ such that the increasing bistellar flip occurs inside the subpolytope $C(A \cup B, 2d + 1)$. The increasing bistellar flip then replaces the $d$-simplex $A$ with the $(d + 1)$-simplex $B$ inside this subpolytope. This is likewise because $A$ is the intersection of $\lfacets{A \cup B}{2d + 2}$ and $B$ is the intersection of $\ufacets{A \cup B}{2d + 2}$. We also say here that $A$ is a \emph{mutable} $d$-simplex in $\mathcal{T}$. Under these conceptions, we think of increasing bistellar flips as replacing one half of a circuit with another.
\end{remark}

\subsection{Operations on triangulations}\label{sect:back:op}

Our proof of the equivalence of the orders is inductive and uses the following operations on triangulations, which were introduced in \cite{rambau} based on the corresponding operations on oriented matroids from \cite{b-lv}. However, note that our notation and terminology differs from \cite{b-lv,rambau}, and is instead based on \cite{ot,rs-baues}.

\subsubsection{Operations at the first or last vertex}\label{sect:op:first}

\begin{definition}
If $\S \subseteq [m]$ is a $k$-simplex, we define the \emph{contraction} $\S[m - 1 \leftarrow m]$ of $\S$ by
\begin{align*}
\S[m - 1 \leftarrow m] :=
\begin{cases}
\hfil \S \quad &\text{if }m \notin \S,\\
(\S \setminus m) \cup m - 1 \quad &\text{otherwise.}
\end{cases}
\end{align*}
Note that $\S[m - 1 \leftarrow m]$ is a $(k - 1)$-simplex if $\S \supseteq \{m - 1, m\}$. Given a triangulation~$\mathcal{T}$ of $C(m,n)$, we define the \emph{contraction} $\mathcal{T}[m - 1 \leftarrow m]$ to be the triangulation of $C(m - 1, n)$ given by 
\[\mathcal{T}[m - 1 \leftarrow m] := \adjset{\S[m - 1 \leftarrow m] \in {\textstyle\binom{[m - 1]}{n + 1}} \st \S \in \mathcal{T} }.\]
\end{definition}

This is indeed a triangulation of $C(m - 1, n)$ by \cite[Theorem~4.2(iii)]{rambau}. This corresponds to the triangulation obtained from $|\mathcal{T}|$ by moving vertex $|m|$ along the moment curve until it coincides with vertex $|m-1|$, as illustrated in Figure~\ref{fig:contraction}.

\begin{figure}
\caption{The contraction operation $[4 \leftarrow 5]$.}\label{fig:contraction}
\[
\begin{tikzpicture}

\begin{scope}[shift={(-4,0)}]


\coordinate(1) at (-2,2);
\coordinate(2) at (-1,0.5);
\coordinate(3) at (0,0);
\coordinate(4) at (1,0.5);
\coordinate(5) at (2,2);


\draw[cyc] (1) -- (2) -- (3) -- (4) -- (5) -- (1);
\draw[blue] (1) -- (3);
\draw[blue] (3) -- (5);


\node(v1) at (1) {$\bullet$};
\node(v2) at (2) {$\bullet$};
\node(v3) at (3) {$\bullet$};
\node(v4) at (4) {$\bullet$};
\node(v5) at (5) {$\bullet$};


\draw[->,red,thick] (v5) to [bend left] (v4);


\node at (1) [left = 1mm of 1]{1};
\node at (2) [below left = 1mm of 2]{2};
\node at (3) [below = 1mm of 3]{3};
\node at (4) [below right = 1mm of 4]{4};
\node at (5) [right = 1mm of 5]{5};

\end{scope}


\draw[->, OliveGreen, ultra thick] (-1.2,1) -- (1.2,1);

\begin{scope}[shift={(4,0)}]


\coordinate(1) at (-2,2);
\coordinate(2) at (-1,0.5);
\coordinate(3) at (0,0);
\coordinate(4) at (1,0.5);


\draw[cyc] (1) -- (2) -- (3) -- (4) -- (1);
\draw[blue] (1) -- (3);


\node(v1) at (1) {$\bullet$};
\node(v2) at (2) {$\bullet$};
\node(v3) at (3) {$\bullet$};
\node(v4) at (4) {$\bullet$};


\node at (1) [left = 1mm of 1]{1};
\node at (2) [below left = 1mm of 2]{2};
\node at (3) [below = 1mm of 3]{3};
\node at (4) [below right = 1mm of 4]{4};

\end{scope}

\end{tikzpicture}
\]
\end{figure}
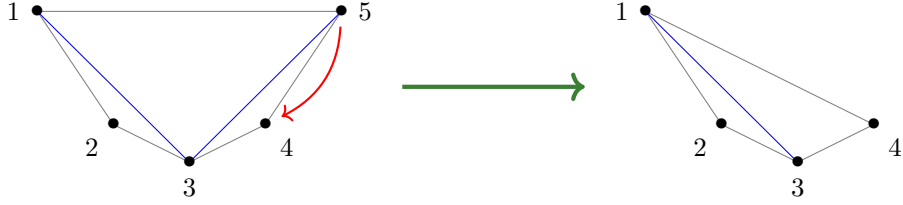

\begin{definition}
Given a triangulation~$\mathcal{T}$ of $C(m,n)$, we define the \emph{deletion} $\mathcal{T}\backslash m$ to be the triangulation of $C(m - 1, n - 1)$ given by \[\mathcal{T}\backslash m := \set{\S\setminus m \st S \in \mathcal{T},\, m \in S}.\]
\end{definition}

This is indeed a triangulation of $C(m - 1, n - 1)$ by \cite[Theorem~4.2(ii)]{rambau}. This the triangulation induced by $|\mathcal{T}|$ on the vertex figure of $\creal{m}{n}$ at $|m|$, as illustrated in Figure~\ref{fig:deletion}.

\begin{figure}
\caption{The deletion operation $- \backslash 5$.}\label{fig:deletion}
\[
\begin{tikzpicture}

\begin{scope}[shift={(-4,0)}]


\coordinate(1) at (-2,2);
\coordinate(2) at (-1,0.5);
\coordinate(3) at (0,0);
\coordinate(4) at (1,0.5);
\coordinate(5) at (2,2);


\draw[cyc] (1) -- (2) -- (3) -- (4) -- (5) -- (1);
\draw[blue] (1) -- (3);
\draw[blue] (3) -- (5);


\node(v1) at (1) {$\bullet$};
\node(v2) at (2) {$\bullet$};
\node(v3) at (3) {$\bullet$};
\node(v4) at (4) {$\bullet$};
\node(v5) at (5) {$\bullet$};


\draw[red,thick] (0,2.5) -- (2,0.5);


\node at (1) [left = 1mm of 1]{1};
\node at (2) [below left = 1mm of 2]{2};
\node at (3) [below = 1mm of 3]{3};
\node at (4) [below right = 1mm of 4]{4};
\node at (5) [right = 1mm of 5]{5};

\end{scope}


\draw[->, OliveGreen, ultra thick] (-1.2,1) -- (1.2,1);

\begin{scope}[shift={(4,0)}]


\coordinate(1) at (-2,1);
\coordinate(2) at (-1,1);
\coordinate(3) at (0,1);
\coordinate(4) at (1,1);


\draw[red] (1) -- (4);


\node(v1) at (1) {$\bullet$};
\node(v3) at (3) {\color{blue} $\bullet$};
\node(v4) at (4) {$\bullet$};


\node at (1) [below = 1mm of 1]{1};
\node at (3) [below = 1mm of 3]{3};
\node at (4) [below = 1mm of 4]{4};

\end{scope}

\end{tikzpicture}
\]
\end{figure}

These operations behave well with respect to the higher Stasheff--Tamari orders.

\begin{theorem}[{\cite[Proposition~5.14]{rambau}, \cite[Theorem~4.1]{thomas}}]
\hfill
\begin{enumerate}
\item The operation $[m - 1 \leftarrow m]$ is order-preserving with respect to both the first and the second higher Stasheff--Tamari orders.
\item The operation $-\backslash m$ is order-reversing with respect to both the first and the second higher Stasheff--Tamari orders.
\end{enumerate}
\end{theorem}

We shall need to use the following combinatorial characterisations of contraction.

\begin{proposition}[{\cite[Lemma~2.23]{ot}, \cite[Lemma~4.12]{njw-hst}}]\label{prop:cont_char_m}
\hfill
\begin{enumerate}
\item For a triangulation $\mathcal{T} \in \mathsf{S}(m, 2d)$, we have \[
\simp(\mathcal{T}[m - 1 \leftarrow m]) = \set{A \in \nonconsec{m - 1}{d} \st A = B[m - 1 \leftarrow m] \text{ for } B \in \simp(\mathcal{T})}.\]\label{op:cont_char_m:even}
\item For a triangulation $\mathcal{T} \in \mathsf{S}(M, 2d + 1)$, we have \[\simp(\mathcal{T}[m - 1 \leftarrow m])=\set{A \in \simp(\mathcal{T}) \st a_{d} \neq m-1}.\]\label{op:cont_char_m:odd}
\end{enumerate}
\end{proposition}

There are analogous operations $-[1 \rightarrow 2]$ and $-\backslash 1$, which are order-preserving for both orders.


\subsubsection{Operations at middle vertices}

In this paper, we also consider contractions of triangulations at other pairs of vertices besides $[1 \rightarrow 2]$ and $[m - 1 \leftarrow m]$. For this purpose we let $[m - 1]_{v+}:= \{1, 2, \dots, v - 1, x, y, v + 1, v + 2, \dots, m - 1\}$. We also extend this notation in a natural way to $\U \subseteq [m - 1]$, so that $\U_{v+} = (\U \setminus v) \cup \{x, y\}$ if $v \in \U$, and $\U = \U$ otherwise.

\begin{definition}
Given a $k$-simplex $\S \subseteq [m - 1]_{v+}$, we define
\begin{align*}
\S\xvy :=
\begin{cases}
\hfil \S \quad &\text{if } \{x,y\} \cap \S = \emptyset,\\
(\S \setminus \{x,y\}) \cup v \quad &\text{otherwise.}
\end{cases}
\end{align*}
Given a triangulation~$\mathcal{T}$ of $C([m - 1]_{v+}, n)$, we then define the \emph{contraction} $\mathcal{T}\xvy$ to be the triangulation of $C(m - 1, n)$ given by \[\mathcal{T}\xvy := \adjset{S\xvy \in {\textstyle \binom{[m - 1]}{n + 1}} \st S \in \mathcal{T}}.\]
\end{definition}

This is indeed a triangulation of $C(m - 1, n)$ by \cite[Theorem~3.3]{rs-baues}. Geometrically, it is obtained from $|\mathcal{T}|$ by moving $|x|$ and $|y|$ along the moment curve towards each other until they coincide with each other at a new vertex, which we label $|v|$. We choose to relabel $[m]$ as $[m - 1]_{v+}$ here so that there does not appear to be a missing vertex after contraction. It is also useful to distinguish between the two vertices before contraction and the vertex after contraction.

In order to understand how the contractions $\xvy$ behave, we will consider the deletion operation at other vertices too. Indeed, we make the following definition.

\begin{definition}\label{def:comb_vert_fig}
Given $C(m, n)$ and $v \in [m]$, define the \emph{combinatorial vertex figure at $v$} to be the combinatorial polytope $C(m, n)\backslash v = ([m]\setminus v, \vfacets{[m]\setminus v}{n}, \vcircuits{[m]\setminus v}{n})$, where \[\vfacets{[m]\setminus v}{n} = \set{F \subseteq [m]\setminus v \st F \cup v \in \cfacets{[m]}{n}},\] and \[\vcircuits{[m]\setminus v}{n} = \set{(Z_{-}, Z_{+}) \st (Z_{-} \cup v, Z_{+}) \in \ccircuits{[m]}{n} \text{ or } (Z_{-}, Z_{+} \cup v) \in \ccircuits{[m]}{n}}.\]
\end{definition}

\begin{definition}
Given a triangulation~$\mathcal{T}$ of $C(m, n)$, we define the \emph{deletion} $\mathcal{T}\backslash v$ to be the triangulation \[\mathcal{T}\backslash v := \set{S \setminus v \st S \in \mathcal{T}, v \in S}.\]
\end{definition}

It follows straightforwardly from the definition of $C(m, n)\backslash v$ and the definition of a combinatorial triangulation that $\mathcal{T}\backslash v$ is a triangulation of $C(m, n)\backslash v$. Note also that $|\mathcal{T}\backslash v|$ may be realised geometrically as the triangulation induced by $|\mathcal{T}|$ on the vertex figure of $\creal{m}{n}$ at $v$.

Finally, we shall also consider deletions of multiple vertices. Given a triangulation~$\mathcal{T}$ of $C(m, n)$ and $V \subseteq [m]$, we define the collection of simplices \[\mathcal{T}\backslash V := \set{S \setminus V \st S \in \mathcal{T},\, V \subseteq S}.\] In the examples we consider here, $V$ will always be a pair of consecutive vertices, such as $\{1, 2\}$ or $\{m - 1, m\}$ in $[m]$, or $\{x, y\}$ in $[m]_{v+}$.

\section{Equality of the two orders}\label{sect:main}

In this section we prove that the higher Stasheff--Tamari orders are equal. What we need to establish is that for $\mathcal{T}, \mathcal{T}' \in \mathsf{S}(m,n)$ with $\mathcal{T} <_{2} \mathcal{T}'$, then we can find a triangulation~$\mathcal{T}''$ either such that $\mathcal{T} \lessdot_{1} \mathcal{T}'' \leqslant_{2} \mathcal{T}'$, or such that $\mathcal{T} \leqslant_{2} \mathcal{T}'' \lessdot_{1} \mathcal{T}'$. See Lemma~\ref{lem:different_formulations} for more detail on this point. Following \cite{njw-hst}, we treat the odd-dimensional cases separately from the even-dimensional cases. The details of the proof are different for these two cases, but the broad outlines are similar. We explain these outlines now.

The proof is by induction on the number of vertices of the cyclic polytope $C(m, n)$, noting that the orders are known to be equal when $m \leqslant n + 3$ \cite{rr}. We start with triangulations $\mathcal{T}, \mathcal{T}'$ of $C(m,n)$ such that $\mathcal{T} <_{2} \mathcal{T}'$. We perform contractions to obtain triangulations $\mathcal{T}[m - 1 \leftarrow m]$ and ${\mathcal{T}'}[m - 1 \leftarrow m]$ of $C(m - 1, n)$. In the case that $\mathcal{T}[m - 1 \leftarrow m] \neq {\mathcal{T}'}[m - 1 \leftarrow m]$, we apply the induction hypothesis to these triangulations. This provides an increasing flip $\mathcal{U}$ of $\mathcal{T}[m - 1 \leftarrow m]$ such that $\mathcal{U} \leqslant_{2} \mathcal{T}'[m - 1 \leftarrow m]$, and hence provides a subpolytope of $\mathcal{T}[m - 1 \leftarrow m]$ combinatorially equivalent to $C(n + 2, n)$. We consider the pre-image of this subpolytope in $\mathcal{T}$. If the pre-image of this subpolytope is combinatorially equivalent to $C(n + 2, n)$, then we choose the increasing flip inside this subpolytope to obtain our triangulation~$\mathcal{T}'$. As we show in Lemma~\ref{lem:subpolytope_exp}, the only other option is that the preimage of the $C(n + 2, n)$ subpolytope is a subpolytope combinatorially equivalent to $C(n + 3, n)$. This polytope is still relatively small and the triangulations of it are well-understood, as recorded in Lemma~\ref{lem:pent_triangs}. We can find an increasing bistellar flip $\mathcal{T}''$ of $\mathcal{T}$ which occurs within the induced triangulation of this $C(n + 3, n)$ subpolytope. We then show that if we do not have $\mathcal{T}'' \leqslant_{2} \mathcal{T}'$, then there is a contradiction of the existence of the increasing bistellar flip we chose using the induction hypothesis. Deriving this contradiction requires a series of lemmas, and the details differ between even and odd dimensions. 

If $\mathcal{T}[m - 1 \leftarrow m] = \mathcal{T}'[m - 1 \leftarrow m]$, then we instead consider the contractions $\mathcal{T}[1 \rightarrow 2]$ and ${\mathcal{T}'}[1 \rightarrow 2]$. If $\mathcal{T}[1 \rightarrow 2] \neq {\mathcal{T}'}[1 \rightarrow 2]$, then we can apply symmetries of the cyclic polytope to convert to the case where $\mathcal{T}[m - 1 \leftarrow m] \neq \mathcal{T}'[m - 1 \leftarrow m]$.

If we have that both $\mathcal{T}[1 \rightarrow 2] = \mathcal{T}'[1 \rightarrow 2]$ and $\mathcal{T}[m - 1 \leftarrow m] = \mathcal{T}'[m - 1 \leftarrow m]$, then one can apply the results of \cite{njw-hst} to show that, since we have $\mathcal{T} <_{2} \mathcal{T}'$, there must be a $v \in [2,m-2]$ such that if we relabel the vertices of $C(m, n)$ such that $\mathcal{T}, \mathcal{T}'$ are triangulations of $C([m-1]_{v_{+}}, n)$, then we have that $\mathcal{T}\xvy \neq \mathcal{T}'\xvy$. Then one can proceed similarly to before.


\subsection{Preliminary lemmas}

We begin by proving some preliminary lemmas and recording some known results which we shall need. The proofs of three key lemmas (Lemma~\ref{lem:subpolytope_exp}, Lemma~\ref{lem:subpolytope_exp_oth_vert} and Lemma~\ref{lem:v_cont_op}) will be postponed to Section~\ref{sect:expansion}, which concerns the theory behind the contraction $\xvy$.

The following lemma records the possible triangulations of $C(n+3,n)$ and their properties. These triangulations are already well understood; for instance, see \cite[Proof of Proposition 9.1]{thomas-bst}. This lemma can be verified using the results described in Section~\ref{sect:back:comb_char}. Recall the sets $\nonconsec{m}{d}$ and $\chainset{m}{d}$ of internal $d$-simplices for cyclic polytopes in dimension $2d$ and $2d + 1$ from~\eqref{eq:even_int} and~\eqref{eq:odd_int}.

\begin{lemma}\label{lem:pent_triangs}
The triangulations of $C(n+3,n)$ may be described as follows.
\begin{enumerate}
\item If $n=2d$, then
	\begin{enumerate}
	\item $C(2d+3,2d)$ has $2d+3$ triangulations $\mathcal{T}_{1}, \mathcal{T}_{2}, \dots, \mathcal{T}_{2d+3}$;
	\item the triangulation~$\mathcal{T}_{i}$ is the \emph{fan} triangulation at the vertex $i$, that is \[\simp(\mathcal{T}_{i}) = \set{A \in \nonconsec{2d+3}{d} \st i \in A};\]
	\item the poset $\mathsf{S}_{1}(2d+3,2d)=\mathsf{S}_{2}(2d+3,2d)$ is as shown in Figure~\ref{fig:pent_triangs}.
	\item the bistellar flips of the triangulations are as follows:
		\begin{itemize}
		\item $\mathcal{T}_{1}$ possesses two increasing bistellar flips: one which replaces $\tup{1, 3, \dots, 2d + 1}$ with $\tup{2,4,\dots, 2d+2}$ and one which replaces $\tup{1, 4,\dots,2d+2}$ with $\tup{3,5,\dots,2d+3}$;
		\item for $i$ even, $\mathcal{T}_{i}$ admits an increasing flip replacing $\tup{1,3,\dots,i-3,i,i+2, \dots, 2d+2}$ with $\tup{2,4,\dots,i-2,i+1,i+3,\dots,2d+3}$;
		\item for $i$ odd with $i \notin \{1, 2d + 3\}$, $\mathcal{T}_{i}$ admits an increasing flip which replaces $\tup{1,3,\dots, i, i+3, \dots, 2d+2}$ with $\tup{2,4,\dots, i-1, i+2,i+4, \dots, 2d+3}$.
		\end{itemize}
	\end{enumerate}
\item If $n=2d+1$, then
	\begin{enumerate}
	\item $C(2d+4,2d+1)$ has $2d+4$ triangulations $\mathcal{T}_{1}, \mathcal{T}_{2}, \dots, \mathcal{T}_{2d+4}$;
	\item the triangulations $\mathcal{T}_{i}$ have the following sets of internal $d$-simplices:
	\begin{align*}
	\simp(\mathcal{T}_{1}) &= \chainset{2d+4}{d}, \\
	\simp(\mathcal{T}_{2d+4}) &= \emptyset,
	\end{align*}
	for $i \neq 2d + 4$ even 
	\begin{align*}
	\simp(\mathcal{T}_{i}) = \bigset{\tup{2,4,\dots, 2d+2}, &\tup{2, 4, \dots, 2d, 2d+3}, \dots, \\
	 &\tup{2, 4, \dots, i, i+3, i+5, \dots,2d+3}},
	\end{align*}
	 and for $i \neq 1$ odd \begin{align*}
	 \simp(\mathcal{T}_{i}) = \bigset{\tup{2,4, \dots, i-3,& i,i+2, \dots, 2d+3}, \dots,\\&\tup{2, 5,7, \dots, 2d+3},\tup{3,5,\dots, 2d+3}};
	 \end{align*}
	\item the poset $\mathsf{S}_{1}(2d+4,2d+1)=\mathsf{S}_{2}(2d+4,2d+1)$ is as shown in Figure~\ref{fig:pent_triangs}.
	\item the bistellar flips of the triangulations are as follows:
		\begin{enumerate}
		\item $\mathcal{T}_{1}$ admits two increasing bistellar flips: one from removing $\tup{2,4,\dots, 2d+2}$ and one from removing $\tup{3,5, \dots, 2d+3}$;
		\item for $i \neq 2d + 4$ even, $\mathcal{T}_{i}$ admits an increasing bistellar flip from removing $\tup{2,4, \dots, i, i+3, i+5, \dots,2d+3}$;
		\item for $i \neq 1$ odd, $\mathcal{T}_{i}$ admits an increasing bistellar flip from removing $\tup{2,4, \dots, i-3, i,i+2, \dots, 2d+3}$.
		\end{enumerate}
	\end{enumerate}
\end{enumerate}
\end{lemma}

\begin{figure}
\caption{The higher Stasheff--Tamari posets on $\mathsf{S}(n + 3, n)$}\label{fig:pent_triangs}
\[
	\begin{tikzpicture}
	
	\begin{scope}[xscale=2,shift={(-1.75,0)}]
	
	
	\node(b) at (0,0) {$\mathcal{T}_{1}$};
	\node(l1) at (-1,1) {$\mathcal{T}_{2d + 2}$};
	\node(r1) at (1,1) {$\mathcal{T}_{3}$};
	\node(l2) at (-1,2) {};
	\node(r2) at (1,2) {};
	\node(l3) at (-1,3) {};
	\node(r3) at (1,3) {};
	\node(l4) at (-1,4) {$\mathcal{T}_{2}$};
	\node(r4) at (1,4) {$\mathcal{T}_{2d + 1}$};
	\node(t) at (0,5) {$\mathcal{T}_{2d + 3}$};
	
	
	\draw[->] (b) -- (l1);
	\draw[->] (b) -- (r1);
	\draw[->] (l1) -- (l2);
	\draw[->] (r1) -- (r2);
	\draw[dot diameter=1pt, dot spacing=5pt, dots] (l2) -- (l3);
	\draw[dot diameter=1pt, dot spacing=5pt, dots] (r2) -- (r3);
	\draw[->] (l3) -- (l4);
	\draw[->] (r3) -- (r4);
	\draw[->] (r4) -- (t);
	\draw[->] (l4) -- (t);
	
	
	\node at (0,-1) {$\mathsf{S}_{1}(2d+3,2d)=\mathsf{S}_{2}(2d+3,2d)$};
	
	\end{scope}

	\begin{scope}[xscale=2,shift={(1.75,0)}]
	
	
	\node(b) at (0,0) {$\mathcal{T}_{1}$};
	\node(l1) at (-1,1) {$\mathcal{T}_{2}$};
	\node(r1) at (1,1) {$\mathcal{T}_{2d + 3}$};
	\node(l2) at (-1,2) {};
	\node(r2) at (1,2) {};
	\node(l3) at (-1,3) {};
	\node(r3) at (1,3) {};
	\node(l4) at (-1,4) {$\mathcal{T}_{2d + 2}$};
	\node(r4) at (1,4) {$\mathcal{T}_{3}$};
	\node(t) at (0,5) {$\mathcal{T}_{2d + 4}$};
	
	
	\draw[->] (b) -- (l1);
	\draw[->] (b) -- (r1);
	\draw[->] (l1) -- (l2);
	\draw[->] (r1) -- (r2);
	\draw[dot diameter=1pt, dot spacing=5pt, dots] (l2) -- (l3);
	\draw[dot diameter=1pt, dot spacing=5pt, dots] (r2) -- (r3);
	\draw[->] (l3) -- (l4);
	\draw[->] (r3) -- (r4);
	\draw[->] (r4) -- (t);
	\draw[->] (l4) -- (t);
	
	
	\node at (0,-1) {$\mathsf{S}_{1}(2d+4,2d+1)=\mathsf{S}_{2}(2d+4,2d+1)$};
	
	\end{scope}
	
	\end{tikzpicture}	
\]
\end{figure}
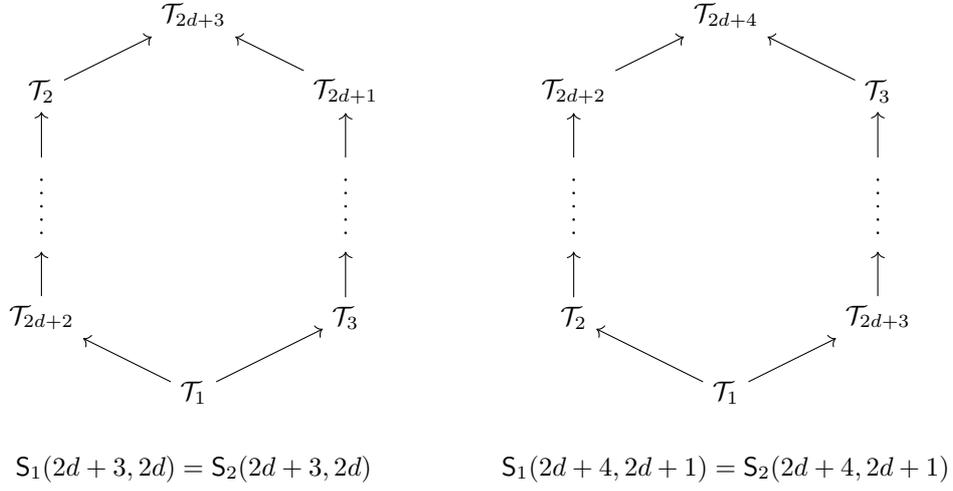

\begin{example}\label{ex:pent_triangs}
We give examples of the triangulations described in Lemma~\ref{lem:pent_triangs}. We denote each triangulation~$\mathcal{T}_{i}$ by its set of internal $d$-simplices $\simp(\mathcal{T})$. The poset $\mathsf{S}_{1}(7,4)=\mathsf{S}_{2}(7,4)$ is shown in Figure~\ref{fig:c74_triangs}. The poset $\mathsf{S}_{1}(8,5)=\mathsf{S}_{2}(8,5)$ is shown in Figure~\ref{fig:c85_triangs}.

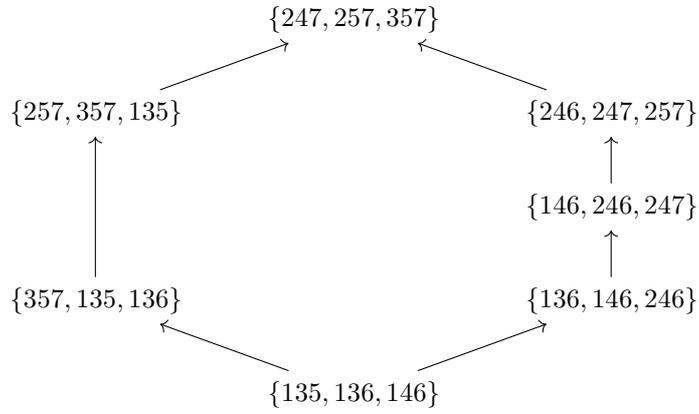
\begin{figure}
\caption{Triangulations of $C(7,4)$.}\label{fig:c74_triangs}
\[
\begin{tikzcd}
& \{247,257,357\} & \\
\{257,357,135\} \ar[ur] && \{246,247,257\} \ar[ul] \\
 && \{146,246,247\} \ar[u] \\
\{357,135,136\} \ar[uu] && \{136,146,246\} \ar[u] \\
& \{135,136,146\} \ar[ur] \ar[ul] &
\end{tikzcd}
\]
\end{figure}

\begin{figure}
\caption{Triangulations of $C(8,5)$.}\label{fig:c85_triangs}
\[
\begin{tikzcd}
& \emptyset & \\
\{246\} \ar[ur] && \{357\} \ar[ul] \\
\{246,247\} \ar[u] && \{257,357\} \ar[u] \\
\{246,247,257\} \ar[u] && \{247,257,357\} \ar[u] \\
& \{246,247,257,357\} \ar[ur] \ar[ul] &
\end{tikzcd}
\]
\end{figure}
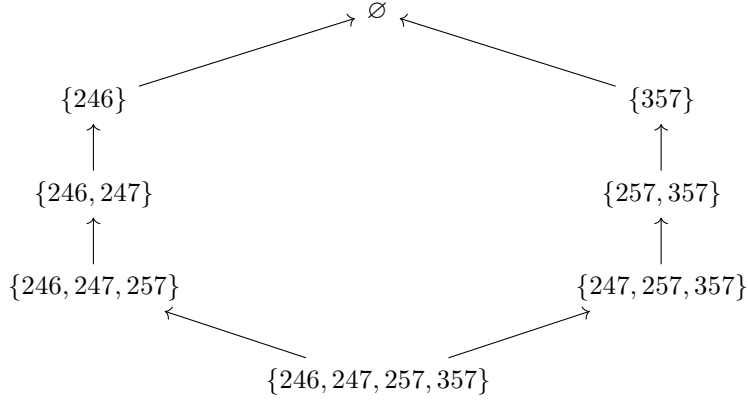

\end{example}

The following lemma is straightforward, but serves to clarify what needs to be proven in order to show that the orders are equivalent.

\begin{lemma}\label{lem:different_formulations}
The following are equivalent.\nopagebreak
\begin{enumerate}
\item For any pair of triangulations $\mathcal{T}, \mathcal{T}' \in \mathsf{S}(m,n)$, we have that \[\mathcal{T} \leqslant_{1} \mathcal{T}' \iff \mathcal{T} \leqslant_{2} \mathcal{T}'.\]\label{op:equiv}
\item For any pair of triangulations $\mathcal{T}, \mathcal{T}' \in \mathsf{S}(m,n)$ such that $\mathcal{T} <_{2} \mathcal{T}'$, there exists a triangulation~$\mathcal{T}'' \in \mathsf{S}(m,n)$ such that \[\mathcal{T} \lessdot_{1} \mathcal{T}'' \leqslant_{2} \mathcal{T}'.\]\label{op:inc}
\item For any pair of triangulations $\mathcal{T}, \mathcal{T}' \in \mathsf{S}(m,n)$ such that $\mathcal{T} <_{2} \mathcal{T}'$, there exists a triangulation~$\mathcal{T}'' \in \mathsf{S}(m,n)$ such that \[\mathcal{T} \leqslant_{2} \mathcal{T}'' \lessdot_{1} \mathcal{T}'.\]\label{op:dec}
\end{enumerate}
\end{lemma}
\begin{proof}
First note that it is already known from \cite[Proposition~2.5]{er} that if $\mathcal{T} \leqslant_{1} \mathcal{T}'$, then $\mathcal{T} \leqslant_{2} \mathcal{T}'$. To show that (\ref{op:equiv}) implies (\ref{op:inc}) and (\ref{op:dec}), suppose that we have $\mathcal{T},\mathcal{T}' \in \mathsf{S}(m,n)$ such that $\mathcal{T} <_{2} \mathcal{T}'$. Then, from (\ref{op:equiv}), it follows that $\mathcal{T} <_{1} \mathcal{T}'$, so that we have \[\mathcal{T}=\mathcal{T}_{0} \lessdot_{1} \mathcal{T}_{1} \lessdot_{1} \dots \lessdot_{1} \mathcal{T}_{r}=\mathcal{T}'.\] Hence we have $\mathcal{T} \lessdot_{1} \mathcal{T}_{1} \leqslant_{1} \mathcal{T}'$, and so $\mathcal{T} \lessdot_{1} \mathcal{T}_{1} \leqslant_{2} \mathcal{T}'$ and (\ref{op:inc}) holds. Similarly $\mathcal{T} \leqslant_{2} \mathcal{T}_{r - 1} \lessdot_{1} \mathcal{T}'$, and so (\ref{op:dec}) holds as well.

We now show that (\ref{op:inc}) implies (\ref{op:equiv}). We can assume that $\mathcal{T} <_{2} \mathcal{T}'$, since if $\mathcal{T} = \mathcal{T}'$, then it is trivial that $\mathcal{T} \leqslant_{1} \mathcal{T}'$. Then, by applying (\ref{op:inc}), we obtain that there is a triangulation~$\mathcal{T}_{1} \in \mathsf{S}(m,n)$ such that \[\mathcal{T} \lessdot_{1} \mathcal{T}_{1} \leqslant_{2} \mathcal{T}'.\] By applying (\ref{op:inc}) repeatedly, we obtain a chain \[\mathcal{T}=\mathcal{T}_{0} \lessdot_{1} \mathcal{T}_{1} \lessdot_{1} \dots \lessdot_{1} \mathcal{T}_{r}=\mathcal{T}'.\] This then establishes that $\mathcal{T} <_{1} \mathcal{T}'$, as desired. The proof that (\ref{op:dec}) implies (\ref{op:equiv}) is similar.
\end{proof}

The proof of the following three lemmas will be postponed to Section~\ref{sect:expansion}, since Lemma~\ref{lem:subpolytope_exp_oth_vert} requires substantial groundwork. Lemma~\ref{lem:subpolytope_exp} follows from Lemma~\ref{lem:subpolytope_exp_oth_vert} by choosing $v = m - 1$, but, since it does not require the groundwork that Lemma~\ref{lem:subpolytope_exp_oth_vert} does, it will be useful to prove it separately as a preamble to the proof of Lemma~\ref{lem:subpolytope_exp_oth_vert}. Lemma~\ref{lem:v_cont_op} is proven in Section~\ref{sect:expansion} since it also concerns contractions $\xvy$.

\begin{lemma}\label{lem:subpolytope_exp}
Let $\mathcal{T}$ be a triangulation of $C(m - 1,n)$. Suppose that $\mathcal{T}$ contains a cyclic subpolytope $C(\U,n)$. Let $\preim{\mathcal{T}}$ be a triangulation of $C(m,n)$ such that $\preim{\mathcal{T}}[m - 1 \leftarrow m]=\mathcal{T}$. Then either
\begin{enumerate}
\item $C(\U, n)$ is a subpolytope of $\preim{\mathcal{T}}$ and $m - 1 \notin \U$,
\item $C((\U\setminus m - 1)\cup m, n)$ is a subpolytope of $\preim{\mathcal{T}}$, where $m - 1 \in \U$, or
\item $C(\U \cup m, n)$ is a subpolytope of $\preim{\mathcal{T}}$, where $m - 1 \in \U$.
\end{enumerate}
\end{lemma}

\begin{lemma}\label{lem:subpolytope_exp_oth_vert}
Let $\mathcal{T}$ be a triangulation of $C(m - 1,n)$. Suppose that $\mathcal{T}$ contains a cyclic subpolytope $C(\U, n)$. Let $\preim{\mathcal{T}}$ be a triangulation of $C([m - 1]_{v+},n)$ such that $\preim{\mathcal{T}}\xvy = \mathcal{T}$. Then either
\begin{enumerate}
\item $C(\U, n)$ is a subpolytope of $\preim{\mathcal{T}}$ and $v \notin \U$,
\item $C((\U \setminus v)\cup x, n)$ is a subpolytope of $\preim{\mathcal{T}}$, where $v \in \U$,
\item $C((\U \setminus v)\cup y, n)$ is a subpolytope of $\preim{\mathcal{T}}$, where $v \in \U$, or
\item $C(\U_{v+}, n)$ is a subpolytope of $\preim{\mathcal{T}}$, where $v \in \U$.
\end{enumerate}
\end{lemma}

\begin{lemma}\label{lem:v_cont_op}
The contraction $\xvy$ is order-preserving with respect to the second higher Stasheff--Tamari order. That is, if $\preim{\mathcal{T}}$ and $\preim{\mathcal{T}}'$ are triangulations of $C([m - 1]_{v+}, n)$, with $\preim{\mathcal{T}} \leqslant_{2} \preim{\mathcal{T}}'$, then $\mathcal{T} \leqslant_{2} \mathcal{T}'$, where $\mathcal{T} = \preim{\mathcal{T}}\xvy$ and $\mathcal{T}' = \preim{\mathcal{T}}'\xvy$.
\end{lemma}

\subsection{Odd dimensions}\label{sect:odd}

We now prove the equivalence of the orders for odd dimensions. We begin by showing some preliminary lemmas which are specific to odd dimensions. Recall the notion of a support from Definition~\ref{def:support}.

\begin{lemma}\label{lem:unique_support}
Let $\mathcal{T} \in \mathsf{S}(m, 2d + 1)$ be a triangulation with a mutable $d$-simplex $A \in \simp(\mathcal{T})$ which is replaced by the $(d + 1)$-simplex $B$ in the increasing flip. Then $B' = \tup{b_{1}, b_{2}, \dots, b_{d}}$ is the unique support of $A$.
\end{lemma}
\begin{proof}
First note that $B'$ is a support of $A$. This follows from the fact that every $d$-simplex contained in $A \cup B'$, excluding $A$, contains consecutive entries in $A \cup B$, and is therefore a $d$-simplex of $\mathcal{T}$, by virtue of lying on the boundary of $C(A \cup B, 2d + 1)$, which is a subpolytope of $\mathcal{T}$.

We now suppose that $A$ possesses a support $\SU = \tup{\su_{1}, \su_{2}, \dots, \su_{d}}$. We show that if $\su_{i} \neq b_{i}$ for any $i$, then $\mathcal{T}$ contains a $d$-simplex which forms a circuit with a $(d + 1)$-simplex in the boundary of $C(A \cup B, 2d + 1)$. This is a contradiction, since $C(A \cup B, 2d + 1)$ is a subpolytope of $\mathcal{T}$. The internal $(d + 1)$-simplices of $C(A \cup B, 2d + 1)$ consist of $B$ along with the $(d + 1)$-simplices which have $A$ as a face, by Gale's Evenness Criterion. All other $(d + 1)$-simplices in $C(A \cup B, 2d + 1)$ lie on the boundary.

Suppose that $\su_{i} < b_{i}$ for some $i$. Then, since $\SU$ is a support of $A$, we have that $\simp(\mathcal{T})$ contains $\tup{a_{0}, a_{1}, \dots, a_{i - 2}, \su_{i}, a_{i}, a_{i + 1}, \dots, a_{d}}$, which intertwines $\tup{b_{0}, b_{1}, \dots, b_{i - 2}, a_{i - 1}, b_{i}, b_{i + 1}, \dots, b_{d + 1}}$, which is a boundary $(d + 1)$-simplex of $C(A \cup B, 2d + 1)$. Suppose instead that $\su_{i} > b_{i}$ for some $i$. Then, since $\SU$ is a support of $A$, we have that $\simp(\mathcal{T})$ contains $\tup{a_{0}, a_{1}, \dots, a_{i - 1}, \su_{i}, a_{i + 1},  a_{i + 2}, \dots, a_{d}}$ which intertwines $\tup{b_{0}, b_{1}, \dots, b_{i}, a_{i}, b_{i + 2}, b_{i + 3}, \dots, b_{d + 1}}$, which is a boundary $(d + 1)$-simplex of $C(A \cup B, 2d + 1)$. Therefore, we must have $\SU = B'$, and so $B'$ is the unique support of~$A$.
\end{proof}

The following lemma helps us to understand what supports look like in triangulations of $C(2d + 4, 2d + 1)$. This is useful when we expand from $C(2d + 3, 2d + 1)$ subpolytopes to $C(2d + 4, 2d + 1)$ subpolytopes.

\begin{lemma}\label{lem:encompassing_support}
Let $\mathcal{T}$ be a triangulation of $C(2d + 4, 2d + 1)$ which is neither the upper triangulation nor the lower triangulation. Let $A$ be the unique mutable $d$-simplex of $\mathcal{T}$ with $\SU$ the support of $A$. Then every internal $d$-simplex $A'$ of $\mathcal{T}$ has $A' \subseteq A \cup \SU$.
\end{lemma}
\begin{proof}
Note first that $\mathcal{T}$ has a unique mutable $d$-simplex by Lemma~\ref{lem:pent_triangs}. One can then proceed by direct verification. Suppose that we have the triangulation~$\mathcal{T}_{i}$ of $C(2d + 4, 2d + 1)$, where $i$ is even and $i \neq 2d +4$. Hence, as in Lemma~\ref{lem:pent_triangs}, we have $\simp(\mathcal{T}_{i})$ is \[\bigset{\tup{2, 4, \dots, 2d + 2}, \tup{2, 4 \dots, 2d, 2d + 3)}, \dots, \tup{2, 4, \dots, i, i + 3, i + 5, \dots, 2d + 3}}.\] The mutable $(d + 1)$-simplex here is $A = \tup{2, 4, \dots, i, i + 3, i + 5, \dots, 2d + 3}$ by Lemma~\ref{lem:pent_triangs}. One can verify that this has support $\SU = \tup{3, 5, \dots, i - 1, i + 2, i + 4, \dots, 2d + 2}$. Indeed, the internal $d$-simplices contained in $\tup{2, 4, \dots, i, i + 3, i + 5, \dots, 2d + 3} \cup \tup{3, 5, \dots, i - 1, i + 2, i + 4, \dots, 2d + 2}$ are precisely $\simp(\mathcal{T})$. This establishes the claim when $i$ is even, and the case where $i$ is odd is the mirror image of this. 
\end{proof}

The next lemma is the inductive step of the proof of the equivalence of the orders for odd dimensions. Giving it as a separate lemma simplifies the presentation of the proof.

\begin{lemma}\label{lem:odd_indstep_key}
Let $\mathcal{T}, \mathcal{T}' \in \mathsf{S}([m - 1]_{v+}, 2d + 1)$ be triangulations such that $\mathcal{T} <_{2} \mathcal{T}'$ and $\mathcal{T}\xvy <_{2} \mathcal{T}'\xvy$. Suppose that $\mathcal{T}\xvy$ possesses an increasing flip $\mathcal{U}$ such that $\mathcal{T}\xvy \lessdot_{1} \mathcal{U} \leqslant_{2} \mathcal{T}'\xvy$. Then $\mathcal{T}$ possesses an increasing flip $\mathcal{T}''$ such that $\mathcal{T} \lessdot_{1} \mathcal{T}'' \leqslant_{2} \mathcal{T}'$.
\end{lemma}
\begin{proof}
Let the increasing flip from $\mathcal{T}\xvy$ to $\mathcal{U}$ consist of replacing the $d$-simplex $A$ with the $(d+1)$-simplex $B$ inside the cyclic subpolytope $C(A \cup B, 2d+1)$. We must then have that $A \notin \simp(\mathcal{T}'\xvy)$ since $\mathcal{U} \leqslant_{2} \mathcal{T}'\xvy$ implies that $\simp(\mathcal{T}\xvy)\setminus\{A\}=\simp(\mathcal{U}) \supseteq \simp(\mathcal{T}{'}\xvy)$, using Theorem~\ref{thm:hst_char_odd}. 

By Lemma~\ref{lem:subpolytope_exp_oth_vert}, we have that either
\begin{caseenum}
\item $C(A \cup B, 2d + 1)$ is a subpolytope of $\mathcal{T}$, where $v \notin A \cup B$,\label{op:subpoly}
\item $C((A \cup B \cup x)\setminus v, 2d + 1)$ is a subpolytope of $\mathcal{T}$, where $v \in A \cup B$,\label{op:subpoly_x}
\item $C((A \cup B \cup y)\setminus v, 2d + 1)$ is a subpolytope of $\mathcal{T}$, where $v \in A \cup B$, or\label{op:subpoly_y}
\item $C((A \cup B)_{v+}, 2d + 1)$ is a subpolytope of $\mathcal{T}$, where $v \in A \cup B$.\label{op:subpoly_x_y}
\end{caseenum}
We deal with each of these cases in turn.
\begin{caseenum}
\item Suppose that $C(A \cup B, 2d + 1)$ is a subpolytope of $\mathcal{T}$. Then the induced triangulation of this subpolytope must contain $A$, since $A \in \simp(\mathcal{T}\xvy)$ and the contraction does not affect the subpolytope $C(A \cup B, 2d + 1)$. Since $A$ is contained in the subpolytope $C(A \cup B, 2d + 1)$ of $\mathcal{T}$, we have that $\mathcal{T}$ admits an increasing flip $\mathcal{T}''$ where $\simp(\mathcal{T}'') = \simp(\mathcal{T}) \setminus A$. Furthermore, $A \notin \simp(\mathcal{T}')$, since $A \notin \simp(\mathcal{T}'\xvy)$. Thus, we have $\mathcal{T} \lessdot_{1} \mathcal{T}'' \leqslant_{2} \mathcal{T}'$ by Theorem~\ref{thm:hst_char_odd}.

\item Suppose now that $C((A \cup B \cup x)\setminus v, 2d + 1)$ is a subpolytope of $\mathcal{T}$, where $v \in A \cup B$. If $v \in A$, then let $\preim{A} = (A \cup x)\setminus v$. Otherwise, let $\preim{A} = A$. Then the induced triangulation of the subpolytope $C((A \cup B \cup x)\setminus v, 2d + 1)$ must contain $\preim{A}$, since $A \in \simp(\mathcal{T}\xvy)$, so there must be a simplex which contracts to $A$. Moreover, $\preim{A}$ is contained in the subpolytope $C((A \cup B \cup x)\setminus v, 2d + 1)$ of $\mathcal{T}$, so that $\mathcal{T}$ admits an increasing flip $\mathcal{T}''$ where $\simp(\mathcal{T}'') = \simp(\mathcal{T})\setminus \preim{A}$. We also have that $\preim{A} \notin \simp(\mathcal{T}')$, because $A \notin \simp(\mathcal{T}'\xvy)$. Hence $\mathcal{T} \lessdot_{1} \mathcal{T}'' \leqslant_{2} \mathcal{T}'$ by Theorem~\ref{thm:hst_char_odd}, as desired.

\item The case where $C((A \cup B \cup y)\setminus v, 2d + 1)$ is a subpolytope of $\mathcal{T}$ with $v \in A \cup B$ behaves similarly to the previous case.

\item Consider now the case where $C((A \cup B \cup \{x, y\})\setminus v, 2d + 1)$ is a subpolytope of $\mathcal{T}$. Then the triangulation of this subpolytope induced by $\mathcal{T}$ must contain a $d$-simplex $\preim{A}$ such that $\preim{A}\xvy = A$. This $d$-simplex $\preim{A}$ must be internal in $C((A \cup B \cup \{x, y\})\setminus v, 2d + 1)$, since $A$ is internal in $C(A \cup B, 2d + 1)$. Hence the induced triangulation of $C((A \cup B \cup \{x, y\})\setminus v, 2d + 1)$ cannot be the upper triangulation. Moreover, we cannot have $\preim{A} \in \simp(\mathcal{T}')$, since this implies that $A \in \simp(\mathcal{T}'\xvy)$.

Suppose first that the induced triangulation is the lower triangulation~$\mathcal{T}_{l}$ of $C((A \cup B \cup \{x, y\})\setminus v, 2d + 1)$. Then, by Lemma~\ref{lem:pent_triangs}, the lower triangulation~$\mathcal{T}_{l}$ contains two mutable $d$-simplices, which we call $\J$ and $\K$. Suppose that we have both $\J, \K \in \simp(\mathcal{T}')$. Then, recalling the bridging property from Definition~\ref{def:bridge} and Theorem~\ref{thm:tri_char_odd}, $\mathcal{T}'$ must contain every internal $d$-simplex in $\mathcal{T}_{l}$. But this means that $\preim{A} \in \simp(\mathcal{T}')$. Thus, at least one of $\J$ and $\K$ is not a $d$-simplex of $\mathcal{T}'$. Hence, let $\mathcal{T}''$ be the increasing flip of $\mathcal{T}$ defined by removing whichever of $\J$ and $\K$ is not a $d$-simplex of $\mathcal{T}'$. We therefore have $\mathcal{T} \lessdot_{1} \mathcal{T}'' \leqslant_{2} \mathcal{T}'$ by Theorem~\ref{thm:hst_char_odd}, as desired.

Now suppose that the induced triangulation of $C((A \cup B, \cup \{x, y\})\setminus v, 2d + 1)$ is neither the lower triangulation nor the upper triangulation. Then, by Lemma~\ref{lem:pent_triangs}, the induced triangulation has a unique mutable $d$-simplex $\Ja$. By Lemma~\ref{lem:unique_support}, $\Ja$ has a unique support $\SU$ in $\simp(\mathcal{T})$. We then have that $\preim{A} \subseteq \Ja \cup \SU$, by Lemma~\ref{lem:encompassing_support}.

Suppose that $\Ja \in \simp(\mathcal{T}')$. Let $\SU'$ be the support of $\Ja$ in $\simp(\mathcal{T}')$. Then, since $\simp(\mathcal{T}) \supseteq \simp(\mathcal{T}')$ by Theorem~\ref{thm:hst_char_odd}, we have that $\SU'$ is a support of $\Ja$ in $\simp(\mathcal{T})$. This implies that $\SU' = \SU$, by Lemma~\ref{lem:unique_support}. In turn, this implies that $\preim{A} \in \simp(\mathcal{T}')$, which is a contradiction. Therefore, if $\mathcal{T}''$ is the triangulation of $C(m, 2d + 1)$ such that $\simp(\mathcal{T}'') = \simp(\mathcal{T})\setminus\{\Ja\}$, then we must have that $\mathcal{T} \lessdot_{1} \mathcal{T}'' \leqslant_{2} \mathcal{T}'$, as desired.
\end{caseenum}
\end{proof}

We now have enough preliminaries in place to prove the equivalence of the orders in odd dimensions.

\begin{theorem}\label{thm:main_odd}
Let $\mathcal{T}, \mathcal{T}'$ be triangulations of $C(m,2d + 1)$. Then $\mathcal{T} \leqslant_{1} \mathcal{T}'$ if and only if $\mathcal{T} \leqslant_{2} \mathcal{T}'$.
\end{theorem}
\begin{proof}
We prove the result by induction on the number of vertices of the cyclic polytope. We may use the cases where $m - (2d + 1) \leqslant 3$ as base cases, since the result is already known for these cases, as noted in \cite{rr}. Indeed, for $m = 2d + 2$, the cyclic polytope $C(m, 2d + 1)$ is a $(2d + 1)$-simplex, so the result is trivial. The result is also clear for $m = 2d + 3$, since here the cyclic polytope $C(m, 2d + 1)$ only has two triangulations. Finally, one may check the case where $m = 2d + 4$ by applying the results of \cite{njw-hst} to Lemma~\ref{lem:pent_triangs}. Hence, from now on, we assume that $m > 2d + 4$.

As in Lemma~\ref{lem:different_formulations}, we seek a triangulation~$\mathcal{T}''$ such that $\mathcal{T} \lessdot_{1} \mathcal{T}'' \leqslant_{2} \mathcal{T}'$. The existence of such a triangulation will establish our claim. We now split into three cases.

\begin{proofenum}
\item Suppose that $\mathcal{T}[m - 1 \leftarrow m] \neq {\mathcal{T}'}[m - 1 \leftarrow m]$. Then, by the induction hypothesis, $\mathcal{T}[m - 1 \leftarrow m]$ admits an increasing flip $\mathcal{U}$ such that $\mathcal{T}[m - 1 \leftarrow m] \lessdot_{1} \mathcal{U} \leqslant_{2} {\mathcal{T}'}[m - 1 \leftarrow m]$. By applying Lemma~\ref{lem:odd_indstep_key} with $m - 1$ and $m$ relabelled as $x$ and $y$, we obtain that $\mathcal{T}$ possesses an increasing flip $\mathcal{T}''$ such that $\mathcal{T} \lessdot_{1} \mathcal{T}'' \leqslant_{2} \mathcal{T}'$.

\item We now suppose that $\mathcal{T}[1 \rightarrow 2] \neq \mathcal{T}'[1 \rightarrow 2]$, so that $\mathcal{T}[1 \rightarrow 2] <_{2} \mathcal{T}'[1 \rightarrow 2]$. By \cite[Proposition~2.11]{er}, the permutation \[\alpha = \begin{pmatrix}
1 & 2 & \dots & m - 1 & m\\
m & m - 1 & \dots & 2 & 1
\end{pmatrix}
\] on the vertices of the cyclic polytope induces an order-preserving bijection $\alpha$ on both $\mathsf{S}_{1}(m, 2d + 1)$ and $\mathsf{S}_{2}(m, 2d + 1)$. We then have that $\alpha(\mathcal{T}[1 \to 2]) = \alpha(\mathcal{T})[m - 1 \leftarrow m] <_{2} \alpha(\mathcal{T}')[m - 1 \leftarrow m] = \alpha(\mathcal{T}'[1 \to 2])$. Hence, by applying the previous case, we obtain a triangulation~$\mathcal{T}''$ such that $\alpha(\mathcal{T}') \lessdot_{1} \mathcal{T}'' \leqslant_{2} \alpha(\mathcal{T})$. By applying $\alpha$ again, we obtain that $\mathcal{T} \lessdot_{1} \alpha(\mathcal{T}'') \leqslant_{2} \mathcal{T}'$, which resolves this case.

\item We may now suppose that we are in neither of the previous cases, so that $\mathcal{T}[1 \rightarrow 2] =\mathcal{T}'[1 \rightarrow 2]$ and $\mathcal{T}[m - 1 \leftarrow m]={\mathcal{T}'}[m - 1 \leftarrow m]$. Hence, by Proposition~\ref{prop:cont_char_m}\eqref{op:cont_char_m:odd} and its analogue for the contraction $[1 \to 2]$, we must have that
\begin{align*}
\set{A \in \simp(\mathcal{T}) \st 2 \notin A} &= \set{A \in \simp(\mathcal{T}') \st 2 \notin A}, \\
\set{A \in \simp(\mathcal{T}) \st m - 1 \notin A} &= \set{A \in \simp(\mathcal{T}') \st m - 1 \notin A}, 
\end{align*}
and so $\simp(\mathcal{T})$ and $\simp(\mathcal{T}')$ only differ in simplices containing both $2$ and $m-1$. Let $A$ be a simplex such that $A \in \simp(\mathcal{T})\setminus\simp(\mathcal{T}')$. Then $\{2,m-1\} \subseteq A$. There must be some $i \in [d]$ such that $a_{i}-a_{i-1}>2$, otherwise $m-1=2d+2$, and we are supposing that this is not the case because $m = 2d + 3$ is a base case.

Relabel $[m]$ as $[m - 1]_{v+}$, and relabel $A$ correspondingly, such that $a_{i - 1} < x < y < a_{i}$ and then perform the contraction $\xvy$. Lemma~\ref{lem:v_cont_op} tells us that $\mathcal{T}\xvy \leqslant_{2} \mathcal{T}'\xvy$. Moreover, $\mathcal{T}\xvy <_{2} \mathcal{T}'\xvy$, since $A \in \simp(\mathcal{T}\xvy)\setminus\simp(\mathcal{T}'\xvy)$. By the induction hypothesis, there is a increasing flip $\mathcal{U}$ of $\mathcal{T}\xvy$ such that $\mathcal{T}\xvy \lessdot_{1} \mathcal{U} \leqslant_{2} \mathcal{T}{'}\xvy$. We then apply Lemma~\ref{lem:odd_indstep_key} to obtain that there exists an increasing flip $\mathcal{T}''$ of $\mathcal{T}$ such that $\mathcal{T} \lessdot_{1} \mathcal{T}'' \leqslant_{2} \mathcal{T}'$.
\end{proofenum}
\end{proof}

\subsection{Even dimensions}\label{sect:even}

We now prove the equivalence of the orders for even dimensions, beginning by proving preliminary lemmas specific to this parity.

\begin{lemma}\label{lem:obstruction}
Let $\mathcal{T}, \mathcal{T}' \in \mathsf{S}(m,2d)$ such that $\mathcal{T} <_{2} \mathcal{T}'$. Suppose that $\mathcal{T}$ admits an increasing flip $\mathcal{T}''$ which is the result of replacing the $d$-simplex $A$ by the $d$-simplex~$B$. Then, we have that $\mathcal{T}'' \not\leqslant_{2} \mathcal{T}'$ if and only if $\tup{\obs{a}_{0},a_{1}, \dots, a_{d}} \in \simp(\mathcal{T}')$ for some $\obs{a}_{0}$ such that $a_{0} \leqslant \obs{a}_{0} < b_{0}$. 
\end{lemma}
\begin{proof}
If $\mathcal{T}'' \not\leqslant_{2} \mathcal{T}'$, then there must exist some $\J \in \simp(\mathcal{T}')$ such that $\J \wr B$, by Theorem~\ref{thm:hst_char_even}, since $\mathcal{T} <_{2} \mathcal{T}'$. Because we flip from $A$ to $B$ in $\mathcal{T}$, every internal $d$-simplex in $A \cup B$, excluding $B$, must be in $\simp(\mathcal{T})$ by \cite[Proposition~4.6]{ot}. Indeed, they all lie in the facets of the subpolytope $C(A \cup B, 2d)$. Hence, if there is a $d$-simplex $\K \subset A \cup B$ such that $\J \wr \K$, then this contradicts $\mathcal{T} \leqslant_{2} \mathcal{T}'$. If we have $\sJ_{i}<a_{i}$ for some $i$, then $\J \wr (B \setminus \{b_{i}\}) \cup \{a_{i}\}$. Similarly, if, for $i \neq 0$, we have $\sJ_{i}>a_{i}$, then $\J \wr (B \setminus \{b_{i-1}\}) \cup \{a_{i}\}$. Thus, we must have $\J=\tup{\obs{a}_{0},a_{1}, \dots, a_{d}}$, where $\obs{a}_{0} \geqslant a_{0}$. That $\obs{a}_{0}<b_{0}$ follows from the fact that $\J \wr B$.

Conversely, it is clear that if $\J=\tup{\obs{a}_{0},a_{1}, \dots, a_{d}} \in \simp(\mathcal{T}')$ such that $a_{0} \leqslant \obs{a}_{0} < b_{0}$, then $\J \wr B$, so that we have $\mathcal{T}'' \not\leqslant_{2} \mathcal{T}'$.
\end{proof}

We call such a simplex $\tup{\obs{a}_{0},a_{1}, \dots, a_{d}} \in \simp(\mathcal{T}')$ where $a_{0} \leqslant \obs{a}_{0} < b_{0}$ an \emph{obstruction} to the increasing flip of $\mathcal{T}$ which replaces $A$ with $B$. By Lemma~\ref{lem:different_formulations}, in order to prove the equivalence of the orders in even dimensions, we must find an increasing flip of $\mathcal{T}$ which is not obstructed by $\mathcal{T}'$. The following lemma allows us to describe the $2d$-simplex lying below the obstructing $d$-simplex.

\begin{lemma}\label{lem:even_simp_below}
Let $\mathcal{T}, \mathcal{T}' \in \mathsf{S}(m, 2d)$ such that $\mathcal{T} <_{2} \mathcal{T}'$. Suppose that $\mathcal{T}$ admits an increasing flip via replacing the $d$-simplex $A$ with the $d$-simplex $B$, and that $\tup{\obs{a}_{0}, a_{1}, \dots, a_{d}} \in \mathcal{T}'$, where $a_{0} \leqslant \obs{a}_{0} < b_{0}$. Then $\mathcal{T}'$ contains the $2d$-simplex $\tup{\obs{a}_{0}, b_{0}, a_{1}, b_{1}, \dots, b_{d - 1}, a_{d}}$.
\end{lemma}
\begin{proof}
By \cite[Proposition~2.13]{ot}, there exists a $2d$-simplex $S=\stup\obs{a}_{0}, \sQ_{0}, a_{1}, \sQ_{1},$ $\dots,\sQ_{d - 1}, a_{d}\etup$ in $\mathcal{T}'$. We will show that we must have $\sQ_{i} = b_{i}$ for all $i$ by ruling out the other cases.

Suppose that $\sQ_{i} < b_{i}$ for some $i$. Then $\tup{a_{1}, a_{2}, \dots, a_{i}, b_{i}, b_{i + 1}, \dots, b_{d}} \in \simp(\mathcal{T})$, since it is in the boundary of $C(A \cup B, 2d)$, and $\tup{\sQ_{0}, \sQ_{1}, \dots, \sQ_{i}, a_{i + 1}, a_{i + 2}, \dots, a_{d}} \in \simp(\mathcal{T}')$, since it is contained in $S$. But this contradicts $\mathcal{T} <_{2} \mathcal{T}'$ by Theorem~\ref{thm:hst_char_even}, since \[\tup{\sQ_{0}, \sQ_{1}, \dots, \sQ_{i}, a_{i + 1}, a_{i + 2}, \dots, a_{d}} \wr \tup{a_{1}, a_{2}, \dots, a_{i}, b_{i}, b_{i + 1}, \dots, b_{d}}.\]

Now suppose that $\sQ_{i} > b_{i}$ for some $i$. Then $\tup{b_{0}, b_{1}, \dots, b_{i}, a_{i + 1}, a_{i + 2}, \dots, a_{d}} \in \simp(\mathcal{T})$, since it is in the boundary of $C(A \cup B, 2d)$, and we also have that $\tup{\obs{a}_{0}, a_{1}, \dots, a_{i}, \sQ_{i},  \sQ_{i + 1}, \dots, \sQ_{d - 1}} \in \simp(\mathcal{T}')$, since it is contained in $S$. But this contradicts $\mathcal{T} <_{2} \mathcal{T}'$ by Theorem~\ref{thm:hst_char_even}, since \[\stup \obs{a}_{0}, a_{1}, \dots, a_{i}, \sQ_{i}, \sQ_{i + 1}, \dots, \sQ_{d - 1} \etup \wr \tup{b_{0}, b_{1}, \dots, b_{i}, a_{i + 1}, a_{i + 2}, \dots, a_{d}},\] noting that $\obs{a}_{0} < b_{0}$. Thus $\sQ_{i} = b_{i}$ for all~$i$, as desired.
\end{proof}

We will use the following lemma and its corollary to find the right pair of middle vertices to contract at in the final case of our proof of the equivalence of the orders in even dimensions.

\begin{lemma}\label{lem:even_contr_vertices}
Let $\mathcal{T}, \mathcal{T}' \in \mathsf{S}(m, 2d)$ be triangulations such that $\mathcal{T} <_{2} \mathcal{T}'$. Suppose further that both $\mathcal{T}\conto = \mathcal{T}'\conto$ and $\mathcal{T}\contm = \mathcal{T}'\contm$. Then, there exists $A \in \simp(\mathcal{T}) \setminus \simp(\mathcal{T}')$ and, for every such $A$, we have $a_{0} = 1$ and $a_{d} = m - 1$. Dually, there exists $B \in \simp(\mathcal{T}') \setminus \simp(\mathcal{T})$ and, for every such $B$, we have $b_{0} = 2, b_{d} = m$.
\end{lemma}
\begin{proof}
Since we know that $\mathcal{T} \neq \mathcal{T}'$, there must exist $A \in \simp(\mathcal{T})$ such that $A \notin \simp(\mathcal{T}')$. We then have that there is a $B \in \simp(\mathcal{T})$ such that $A$ and $B$ are intertwining. Since $\mathcal{T} <_{2} \mathcal{T}'$, we must in fact have that $A$ intertwines $B$, rather than the other way around, by Theorem~\ref{thm:hst_char_even}. One may also arrive at this situation by first choosing $B \in \simp(\mathcal{T}')$ such that $B \notin \simp(\mathcal{T})$. If $a_{0} > 1$, then $A \in \mathcal{T}[1 \to 2]$ and $B \in \mathcal{T}'[1 \to 2]$, which contradicts the fact that $\mathcal{T}[1 \to 2] = \mathcal{T}'[1 \to 2]$. Hence $a_{0} = 1$, and we similarly argue that $b_{d} = m$. We can continue with similar deductions. If $b_{0} > 2$, then $\mathcal{T}[1 \to 2] \ni \tup{2, a_{1}, a_{2}, \dots, a_{d}} \wr B \in \mathcal{T}'[1 \to 2]$. Therefore $b_{0} = 2$, and we can likewise reason that $a_{d} = m - 1$.
\end{proof}

\begin{corollary}\label{cor:even_contr_vertices}
Let $\mathcal{T}, \mathcal{T}' \in \mathsf{S}(m, 2d)$, with $m > 2d + 3$, be triangulations such that $\mathcal{T} <_{2} \mathcal{T}'$, and both $\mathcal{T}\conto = \mathcal{T}'\conto$ and $\mathcal{T}\contm = \mathcal{T}'\contm$. Then there exists $v \in [3, m - 2]$ such that if one relabels $[m]$ as $[m - 1]_{v+}$, then we have $\mathcal{T}\xvy <_{2} \mathcal{T}'\xvy$.
\end{corollary}
\begin{proof}
Since $\mathcal{T} <_{2} \mathcal{T}'$, there must exist $A \in \simp(\mathcal{T})$ and $B \in \simp(\mathcal{T})$ such that $A \wr B$ by Theorem~\ref{thm:hst_char_even}. By Lemma~\ref{lem:even_contr_vertices}, we have that $a_{0} = 1, b_{0} = 2, a_{d} = m - 1, b_{d} = m$. Because $m > 2d + 3$, we must have $[m] \setminus (A \cup B) \neq \emptyset$. We may therefore choose $\{v, v + 1\} \subset [m]$ such that $\#\{v, v + 1\} \cap (A \cup B) \leqslant 1$ and $\{v, v + 1\} \cap \{1, 2, m - 1, m\} = \emptyset$. We then relabel $[m]$ as $[m - 1]_{v+}$ so that $\{v, v + 1\}$ becomes $\{x, y\}$. We likewise relabel $\mathcal{T}, \mathcal{T}'$, $A$, and $B$. If we let $\cont{A}, \cont{B}$ be the respective images of $A$ and $B$ under the contraction $[x \to v \leftarrow y]$, then, by our choice of $v$, we obtain that $\cont{A} \wr \cont{B}$. By Lemma~\ref{lem:v_cont_op}, we obtain that $\mathcal{T}\xvy <_{2} \mathcal{T}'\xvy$.
\end{proof}

We now prove that the orders are equivalent in even dimensions. The structure of the proof is similar to odd dimensions, but we are not able to extract the inductive step of the proof as a separate lemma, since the details differ between the contractions $\contm$ and $\xvy$.

\begin{theorem}\label{thm:main_even}
Let $\mathcal{T}, \mathcal{T}'$ be triangulations of $C(m,2d)$. Then $\mathcal{T} \leqslant_{1} \mathcal{T}'$ if and only if $\mathcal{T} \leqslant_{2} \mathcal{T}'$.
\end{theorem}
\begin{proof}
As in the odd-dimensional case, we prove the result by induction on the number of vertices of the cyclic polytope. As noted in \cite{rr}, the result is already known for $m - 2d \leqslant 3$, so we use these as the base cases of our induction. One may also easily verify the result in these cases in the same way as explained in the proof of Theorem~\ref{thm:main_odd}.

Hence, we suppose for induction that we have triangulations $\mathcal{T}, \mathcal{T}' \in \mathsf{S}(m, 2d)$, where $m > 2d + 3$, such that $\mathcal{T} <_{2} \mathcal{T}'$. We split into three cases, seeking a triangulation~$\mathcal{T}''$ such that $\mathcal{T} \lessdot_{1} \mathcal{T}'' \leqslant_{2} \mathcal{T}'$.

\begin{proofenum}
\item Suppose that $\mathcal{T}[m - 1 \leftarrow m] \neq {\mathcal{T}'}[m - 1 \leftarrow m]$, so that $\mathcal{T}[m - 1 \leftarrow m] <_{2} {\mathcal{T}'}[m - 1 \leftarrow m]$. By the induction hypothesis, there exists a triangulation $\mathcal{U}$ such that $\mathcal{T}[m - 1 \leftarrow m] \lessdot_{1} \mathcal{U} \leqslant_{2} {\mathcal{T}'}[m - 1 \leftarrow m]$. Let this increasing flip of $\mathcal{T}[m - 1 \leftarrow m]$ be given by exchanging a $d$-simplex $A$ for a $d$-simplex $B$. Therefore, we have that $C(A \cup B, 2d)$ is a subpolytope of $\mathcal{T}[m - 1 \leftarrow m]$. By Lemma~\ref{lem:subpolytope_exp}, we have that either
\begin{caseenum}
\item $C(A \cup B, 2d)$ is a subpolytope of $\mathcal{T}$,
\item $C(A \cup B', 2d)$ is a subpolytope of $\mathcal{T}$, where $B' = \tup{b_{0}, b_{1}, \dots, b_{d - 1}, m}$ and $b_{d} = m - 1$, or
\item $C(A \cup B \cup m, 2d)$ is a subpolytope of $\mathcal{T}$, in which case $b_{d} = m - 1$.
\end{caseenum}

We deal with each of these cases in turn.

\begin{caseenum}
\item Suppose first that $C(A \cup B,2d)$ is a subpolytope of $\mathcal{T}$. This subpolytope therefore contains the $d$-simplex $A$, since it contains the $d$-simplex $A$ in $\mathcal{T}[m - 1 \leftarrow m]$. Thus $\mathcal{T}$ also admits an increasing flip by exchanging $A$ for $B$ to give a triangulation~$\mathcal{T}''$. If $\mathcal{T}'' \not\leqslant_{2} \mathcal{T}'$, then $\obs{A}=\tup{\obs{a}_{0},a_{1}, \dots, a_{d}} \in \simp(\mathcal{T}')$ where $a_{0} \leqslant \obs{a}_{0}<b_{0}$, by Lemma~\ref{lem:obstruction}. But then $\obs{A} \in \simp({\mathcal{T}'}[m - 1 \leftarrow m])$, since $m \notin \obs{A}$, which contradicts $\mathcal{U} \leqslant_{2} {\mathcal{T}'}[m - 1 \leftarrow m]$. Hence, in this case we have that $\mathcal{T} \lessdot_{1} \mathcal{T}'' \leqslant_{2} \mathcal{T}'$, as desired.

\item If $C(A\cup B',2d)$ is a subpolytope of $\mathcal{T}$, then we may exchange $A$ for $B'$. If there is an obstruction to this, then we get a contradiction in a similar way to the previous case.

\item Finally, suppose that $C(A \cup B \cup m,2d)$ is a subpolytope of $\mathcal{T}$, in which case $b_{d} = m - 1$. The triangulation of this subpolytope induced by $\mathcal{T}$ must contain the $d$-simplex $A$, since $m-1 \notin A$ but $A \in \simp(\mathcal{T}[m - 1 \leftarrow m])$. Since $C(A \cup B \cup m, 2d)$ is a cyclic polytope combinatorially equivalent to $C(2d+3,2d)$, all the triangulations of $C(A \cup B\cup m, 2d)$ are fan triangulations, by Lemma~\ref{lem:pent_triangs}. The possible triangulations of $C(A \cup B\cup m,2d)$ then consist of the fan triangulations determined by the elements $a_{i} \in A$, since we must have that $A$ is a $d$-simplex of the induced triangulation of the subpolytope. By Lemma~\ref{lem:pent_triangs}, the fan triangulation of $C(A \cup B \cup m,2d)$ at $a_{i}$ possesses an increasing flip at the $d$-simplex $\J = \tup{a_{0}, a_{1}, \dots, a_{i}, b_{i+1}, b_{i+2}, \dots, b_{d}}$, which is then exchanged for the $d$-simplex $\K  =\tup{b_{0}, b_{1}, \dots, b_{i-1}, a_{i+1}, a_{i+2}, \dots, a_{d}, m}$.

Let $\mathcal{T}''$ be the triangulation resulting from performing this increasing flip on $\mathcal{T}$. If $\mathcal{T}'' \not\leqslant_{2} \mathcal{T}'$, then $\obs{\J}=\tup{\obs{a}_{0},a_{1}, \dots, a_{i}, b_{i+1}, b_{i+2}, \dots, b_{d}} \in \simp(\mathcal{T}')$ where $a_{0} \leqslant \obs{a}_{0}<b_{0}$, by Lemma~\ref{lem:obstruction}. By Lemma~\ref{lem:even_simp_below}, we conclude that $\stup\obs{a}_{0}, b_{0}, a_{1}, b_{1}, \dots, a_{i}, a_{i + 1}, b_{i + 1}, \dots, a_{d}, b_{d}\etup$ is a $2d$-simplex of $\mathcal{T}'$. Consequently, $\obs{A} = \tup{\obs{a}_{0},a_{1}, \dots, a_{d}} \in \simp(\mathcal{T}')$. This implies that $\obs{A} \in \simp({\mathcal{T}'}[m - 1 \leftarrow m])$, which obstructs the flip from $\mathcal{T}[m - 1 \leftarrow m]$ to $\mathcal{U}$. But we assumed that $\mathcal{U} \leqslant_{2} \mathcal{T}'[m - 1 \leftarrow m]$ using the induction hypothesis. We therefore conclude that we cannot have $\obs{\J}  \in \simp(\mathcal{T}')$. This means that we have $\mathcal{T} \lessdot_{1} \mathcal{T}'' \leqslant_{2} \mathcal{T}'$, as desired.
\end{caseenum}

\item We now suppose that $\mathcal{T}[1 \to 2] \neq \mathcal{T}'[1 \to 2]$. By \cite[Proposition~2.11]{er}, the permutation \[\alpha = \begin{pmatrix}
1 & 2 & \dots & m - 1 & m\\
m & m - 1 & \dots & 2 & 1
\end{pmatrix}
\] on the vertices of the cyclic polytope induces an order-\emph{reversing} bijection $\alpha$ on both $\mathsf{S}_{1}(m, 2d)$ and $\mathsf{S}_{2}(m, 2d)$. We then have that $\alpha(\mathcal{T}[1 \to 2]) = \alpha(\mathcal{T})[m - 1 \leftarrow m] >_{2} \alpha(\mathcal{T}')[m - 1 \leftarrow m] = \alpha(\mathcal{T}'[1 \to 2])$. Hence, by applying the previous case, we obtain a triangulation~$\mathcal{T}''$ such that $\alpha(\mathcal{T}') \lessdot_{1} \mathcal{T}'' \leqslant_{2} \alpha(\mathcal{T})$. By applying $\alpha$ again, we obtain that $\mathcal{T} \leqslant_{2} \alpha(\mathcal{T}'') \lessdot_{1} \mathcal{T}'$, which resolves this case, noting Lemma~\ref{lem:different_formulations}.

\item We may now suppose that we have both $\mathcal{T}[m - 1 \leftarrow m] = \mathcal{T}'[m - 1 \leftarrow m]$ and $\mathcal{T}[1 \to 2] = \mathcal{T}'[1 \to 2]$. Since we are assuming that $m > 2d + 3$, we can apply Corollary~\ref{cor:even_contr_vertices} and relabel $[m]$ as $[m - 1]_{v+}$ in such a way that we have $\mathcal{T}\xvy <_{2} \mathcal{T}'\xvy$. By the induction hypothesis, we obtain that there exists an increasing flip $\mathcal{U}$ of $\mathcal{T}\xvy$ such that $\mathcal{T}\xvy \lessdot_{1} \mathcal{U} \leqslant_{2} \mathcal{T}'\xvy$. Suppose that this increasing bistellar flip replaces the $d$-simplex $A$ with the $d$-simplex $B$. Hence we have that $C(A \cup B, 2d)$ is a subpolytope of the triangulation~$\mathcal{T}\xvy$.

Note then that
\begin{align*}
(\mathcal{T}\xvy)[1 \rightarrow 2] &= (\mathcal{T}[1 \rightarrow 2])\xvy\\
&= (\mathcal{T}'[1 \rightarrow 2])\xvy\\
&= (\mathcal{T}'\xvy)[1 \rightarrow 2].
\end{align*}
This follows from \cite[Theorem~3.4]{rs-baues}, but can also be seen directly. We similarly reason that \[(\mathcal{T}\xvy)[m - 2 \leftarrow m - 1] = (\mathcal{T}'\xvy)[m - 2 \leftarrow m - 1].\] Using these observations, we can deduce the values of the first and last elements of $A$ and $B$. We know that $A \notin \simp(\mathcal{T}'\xvy)$, since $\mathcal{U}$ is obtained from $\mathcal{T}\xvy$ by replacing $A$ with $B$, and $\mathcal{U} \leqslant_{2} \mathcal{T}'\xvy$. By applying Lemma~\ref{lem:even_contr_vertices} to $A$, we obtain that $a_{0} = 1$ and $a_{d} = m - 2$. This implies that $b_{d} = m - 1$. Since $A \notin \simp(\mathcal{T}'\xvy)$, but $(\mathcal{T}\xvy)[1 \to 2] = (\mathcal{T}'\xvy)[1 \to 2]$, we must have $\tup{2, a_{1}, a_{2}, \dots, a_{d}} \in \simp(\mathcal{T}'\xvy)$. But $\tup{2, a_{1}, a_{2}, \dots, a_{d}}$ is an obstruction to the flip from $\mathcal{T}\xvy$ to $\mathcal{U}$ unless $b_{0} = 2$. We thus conclude that $b_{0} = 2$.

By Lemma~\ref{lem:subpolytope_exp_oth_vert}, we have that either
\begin{caseenum}
\item $C(A \cup B, 2d)$ is a subpolytope of $\mathcal{T}$ and $v \notin A \cup B$,\label{op:odd_subpoly}
\item $C((A \cup B \cup x)\setminus v, 2d)$ is a subpolytope of $\mathcal{T}$, where $v \in A \cup B$,\label{op:odd_subpoly_x}
\item $C((A \cup B \cup y)\setminus v, 2d)$ is a subpolytope of $\mathcal{T}$, where $v \in A \cup B$, or\label{op:odd_subpoly_y}
\item $C((A \cup B)_{v+}, 2d)$ is a subpolytope of $\mathcal{T}$, where $v \in A \cup B$.\label{op:odd_subpoly_x_y}
\end{caseenum}
We deal with each of these cases in turn.
\begin{caseenum}
\item We suppose that $C(A \cup B, 2d)$ is a subpolytope of $\mathcal{T}$ and $v \notin A \cup B$. The induced triangulation of this subpolytope must contain the $d$-simplex $A$, since it contains the $d$-simplex $A$ in $\mathcal{T}\xvy$. We hence perform an increasing flip on $\mathcal{T}$ by replacing $A$ with $B$ inside this subpolytope, obtaining a triangulation~$\mathcal{T}''$.

We claim that $\mathcal{T}'' \leqslant_{2} \mathcal{T}'$. If not, then, by Lemma~\ref{lem:obstruction}, there exists an obstruction $\obs{A} = \tup{\obs{a}_{0}, a_{1}, \dots, a_{d}} \in \simp(\mathcal{T}')$, where $a_{0} \leqslant \obs{a}_{0} < b_{0}$. But, since $a_{0} = 1, b_{0} = 2$, we must have that $\obs{a}_{0} = a_{0}$. This means that $A \in \simp(\mathcal{T}')$, which implies that $A \in \simp(\mathcal{T}'\xvy)$, which contradicts the fact that $\mathcal{U} \leqslant_{2} \mathcal{T}'\xvy$.

\item The case where $C((A \cup B \cup x) \setminus v, 2d)$ is a subpolytope of $\mathcal{T}$ and $v \in A \cup B$ is largely analogous to the previous case, although there are additional details. We let $\preim{A} \cup \preim{B} = (A \cup B \cup x) \setminus v$, where $\preim{A} \wr \preim{B}$. Hence, we have $\preim{A}\xvy = A$ and $\preim{B}\xvy = B$. We must have that the induced triangulation of the subpolytope $C(\preim{A} \cup \preim{B}, 2d)$ contains the $d$-simplex $\preim{A}$, since the induced triangulation of the subpolytope $C(A \cup B, 2d)$ of $\mathcal{T}\xvy$ contains $A$. We hence perform an increasing flip on $\mathcal{T}$ by replacing $\preim{A}$ with $\preim{B}$ inside this subpolytope, obtaining a triangulation~$\mathcal{T}''$.

We claim that $\mathcal{T}'' \leqslant_{2} \mathcal{T}$. If not, then, by Lemma~\ref{lem:obstruction}, there exists an obstruction $\obs{A} = \tup{\obs{a}_{0}, \preim{a}_{1}, \dots, \preim{a}_{d}}$, where $\preim{a}_{0} \leqslant \obs{a}_{0} < \preim{b}_{0}$. Since $1 < 2 < x < y$ in the ordering on $[m - 1]_{v+}$, by our choice of $x$ and $y$ from Corollary~\ref{cor:even_contr_vertices}, we must have that $\preim{a}_{0} = a_{0} = 1$ and $\preim{b}_{0} = b_{0} = 2$, and so $\obs{b}_{0} = b_{0} = 1$ and $\obs{A} = \preim{A}$. This means that $\preim{A} \in \simp(\mathcal{T}')$, and so $A \in \simp(\mathcal{T}'\xvy)$. This contradicts the fact that $\mathcal{U} \leqslant_{2} \mathcal{T}'\xvy$.

\item The case where $C((A \cup B \cup y) \setminus v, 2d + 1)$ is a subpolytope of $\mathcal{T}$ and $v \in A \cup B$ is analogous to the previous case.

\item We finally suppose that $C((A \cup B)_{v+}, 2d)$ is a subpolytope of $\mathcal{T}$, where $v \in A \cup B$. We label the vertices of this subpolytope by $\U = \{\h_{1}, \h_{2}, \dots, \h_{2d + 3}\}$. From what we deduced about the first and last values of $A$ and $B$, we know that $\h_{1} = 1, \h_{2} = 2, \h_{2d + 2} = m - 2, \h_{2d + 3} = m - 1$.

Using the notation of Lemma~\ref{lem:pent_triangs}, there are two cases to consider, depending upon whether the triangulation of $C(\U, 2d)$ is $\mathcal{T}_{\h_{2i}}$ for $i$ such that $2 \leqslant i \leqslant d + 1$, or $\mathcal{T}_{\h_{2i - 1}}$ for $i$ such that $1 \leqslant i \leqslant d + 1$. We know that the triangulation of this subpolytope contains an internal $d$-simplex $\preim{A}$ such that $\preim{A}\xvy = A$. Since $1 < 2 < x < y$ in the order on $[m - 1]_{v+}$, we must have that $\preim{a}_{0} = a_{0} = 1$. Moreover, we must have that $\preim{A} \subseteq \{\h_{1}, \h_{2}, \dots, \h_{2d + 2}\}$, since $a_{d} < b_{d}$. Hence, the triangulations $\mathcal{T}_{h_{2}}$ and $\mathcal{T}_{h_{2d + 3}}$ are excluded because they do not contain any internal simplices with $1$ as a vertex, while we know that $\preim{a}_{0} = 1$.

If the triangulation of $C(\U, 2d)$ is $\mathcal{T}_{\h_{2i}}$ for $i > 1$, then, by Lemma~\ref{lem:pent_triangs}, there exists an increasing flip given by replacing $\J = \stup\h_{1}, \h_{3}, \dots, \h_{2i - 3}, \h_{2i}, \h_{2i + 2}, \dots,$ $\h_{2d + 2}\etup$ with $\K = \tup{ \h_{2}, \h_{4}, \dots, \h_{2i - 2}, \h_{2i + 1}, \h_{2i + 3},\dots,\h_{2d + 3}}$. Since $\h_{1} = 1$ and $\h_{2} = 2$, if this flip is obstructed, it must be because $\J \in \simp(\mathcal{T}')$. If this is the case, then by Lemma~\ref{lem:even_simp_below}, we have that $\tup{\h_{1}, \h_{2}, \dots, \h_{2i - 3}, \h_{2i - 2}, \h_{2i}, \h_{2i + 1}, \dots, \h_{2d + 2}}$ is a $2d$-simplex of $\mathcal{T}'$. But we must have $\preim{A} \subseteq \stup \h_{1}, \h_{2}, \dots, \h_{2i - 3}, \h_{2i - 2}, \h_{2i}$, $\h_{2i + 1}$, $\dots$, $\h_{2d + 2}\etup$, since $\preim{A} \in \simp(\mathcal{T}_{\h_{2i}})$, and so $\h_{2i - 1} \notin \preim{A}$. This means that $\preim{A} \in \simp(\mathcal{T}')$, which implies that $A \in \simp(\mathcal{T}'\xvy)$. This contradicts the fact that $\mathcal{U} \leqslant_{2} \mathcal{T}'\xvy$, and so we conclude that the increasing flip given by replacing $\J$ with $\K$ cannot be obstructed. Hence if $\mathcal{T}''$ is the triangulation resulting from this flip, then we have $\mathcal{T} \lessdot_{1} \mathcal{T}'' \leqslant_{2} \mathcal{T}'$. The case where the triangulation of $C(\U, 2d)$ is $\mathcal{T}_{\h_{2i - 1}}$ for $i \geqslant 1$ behaves similarly. Thus, in all cases we are able to construct a triangulation~$\mathcal{T}''$ such that $\mathcal{T} \lessdot_{1} \mathcal{T}'' \leqslant_{2} \mathcal{T}'$, as desired.
\end{caseenum}
\end{proofenum}
\end{proof}

\section{Expanding triangulations}\label{sect:expansion}

We now wish to prove Lemma~\ref{lem:subpolytope_exp} and Lemma~\ref{lem:subpolytope_exp_oth_vert}. To do this, we need to understand the set of triangulations $\preim{\mathcal{T}}$ which transform into $\mathcal{T}$ under some contraction. In particular, we need to understand what the pre-image of a subpolytope of $\mathcal{T}$ looks like within $\preim{\mathcal{T}}$. We need to rule out the possibility that this pre-image is a strange collection of simplices which does not itself triangulate a subpolytope.

The pre-images of triangulations under the contraction $[m - 1 \leftarrow m]$ are well-understood due to \cite[Lemma~4.7(i)]{rs-baues}, which states that such triangulations $\preim{\mathcal{T}}$ are in bijection with sections of the vertex figure $\mathcal{T}\backslash (m - 1)$. By symmetry, one can apply the same theory for contractions $[1 \rightarrow 2]$. In this section we show how one can extend the theory to all contractions, that is, for any contraction $\xvy$ of $C([m - 1]_{v+}, n)$. This general case is more challenging, because the vertex figures of $C(m - 1, n)$ are less well-behaved at vertices which are not $1$ or $m - 1$. Considering this general case is necessary for proving Lemma~\ref{lem:subpolytope_exp_oth_vert}.

We also adopt the opposite perspective to contraction, where we think of the triangulation~$\mathcal{T}$ of expanding to $\preim{\mathcal{T}}$. Under this perspective, we are trying to understand the triangulations $\preim{\mathcal{T}}$ of $C(m, n)$ to which $\mathcal{T}$ can expand when one expands the cyclic polytope $C(m - 1, n)$ at a particular vertex. The question then becomes how subpolytopes of $\mathcal{T}$ behave under expansion.

\subsection{Expansion at the first or last vertex}

We begin by illustrating how the theory of \cite[Lemma~4.7(i)]{rs-baues} works, and then show how it can be used to prove Lemma~\ref{lem:subpolytope_exp}. The result \cite[Lemma~4.7(i)]{rs-baues} states that, given a triangulation~$\mathcal{T} \in \mathsf{S}(m - 1, n)$, triangulations $\preim{\mathcal{T}}$ of $C(m, n)$ such that $\preim{\mathcal{T}}[m - 1 \leftarrow m] = \mathcal{T}$ are in bijection with sections of the vertex figure $\mathcal{T}\backslash (m - 1)$. This bijection operates as follows.
\begin{align*}
\left\lbrace\parbox{2.9cm}{\begin{center}$\preim{\mathcal{T}} \in \mathsf{S}(m, n) \st$ $\preim{\mathcal{T}}[m - 1 \leftarrow m] = \mathcal{T}$\end{center}} \right\rbrace \quad &\longleftrightarrow \quad \left\lbrace\,\parbox{3.9cm}{$\text{Sections }\mathcal{W}\text{ of }\mathcal{T}\backslash (m - 1)$}\, \right\rbrace\\
\preim{\mathcal{T}} \qquad &\longmapsto \qquad \preim{\mathcal{T}}\backslash\{m - 1, m\} \\
\parbox{4cm}{\begin{center}$\mathcal{T}^{\circ} \, \cup \, \mathcal{W} \ast \{m - 1, m\}$ $ \cup\, \mathcal{T}\backslash (m - 1)^{+} \ast (m - 1)$ $\cup \,\mathcal{T}\backslash (m - 1)^{-} \ast m$ \end{center}} \quad &\longmapsfrom \qquad \mathcal{W}
\end{align*}
Here
\begin{itemize}
\item $\mathcal{T}^{\circ}$ is the set of $n$-simplices of $\mathcal{T}$ which contain neither $m - 1$ nor $m$ as a vertex.
\item $\mathcal{T}\backslash (m - 1)^{+}$ is the set of $(n - 1)$-simplices $S$ of $\mathcal{T}\backslash (m - 1)$ such that $|S|_{n - 1}$ is above $|\mathcal{W}|_{n - 1}$ with respect to the $(n - 1)$-th coordinate.
\item $\mathcal{T}\backslash (m - 1)^{-}$ is the set of $(n - 1)$-simplices $S$ of $\mathcal{T}\backslash (m - 1)$ such that $|S|_{n - 1}$ is below $|\mathcal{W}|_{n - 1}$ with respect to the $(n - 1)$-th coordinate.
\end{itemize}
Note that the sets $\mathcal{T}\backslash (m - 1)^{+}$ and $\mathcal{T}\backslash (m - 1)^{-}$ are defined with respect to a geometric realisation. Recall that $\mathcal{T}\backslash (m - 1)$ is a triangulation of $C(m - 2, n - 1)$, so our geometric realisation is $|-|_{n - 1}$. These sets, of course, do not depend upon the particular geometric realisation of $C(m - 2, n - 1)$ given by the choice of points on the moment curve.

We now demonstrate how this result works, using an example.

\begin{example}\label{ex:m_exp}
We consider the triangulation~$\mathcal{T}$ of $C(5, 3)$ with $3$-simplices $\{1234, 1245, 2345\}$. Here we abbreviate by writing the simplices as strings, so that $1234 = \tup{1, 2, 3, 4}$. The triangulations $\preim{\mathcal{T}}$ of $C(6,3)$ such that $\preim{\mathcal{T}}[5 \leftarrow 6] = \mathcal{T}$ are in bijection with the sections of the triangulation~$\mathcal{T}\backslash 5$. We have that $\mathcal{T}\backslash 5$ is a triangulation of $C(4, 2)$ which can be realised geometrically as the triangulation of the vertex figure of $C(5, 3)$ at $|5|$. The triangulations $|\mathcal{T}|$ and $|\mathcal{T}\backslash 5|$ are illustrated in Figure~\ref{fig:m_vertex_figure}.

The triangulation~$\mathcal{T} \backslash 5$ has three sections $\mathcal{W}_{1}, \mathcal{W}_{2}, \mathcal{W}_{3}$, which are illustrated in Figure~\ref{fig:mvf_sections}. By \cite[Lemma~4.7(i)]{rs-baues}, these sections correspond to triangulations $\preim{\mathcal{T}}_{1}, \preim{\mathcal{T}}_{2}, \preim{\mathcal{T}}_{3}$ of $C(6, 3)$ such that $\preim{\mathcal{T}}_{i}[5 \leftarrow 6] = \mathcal{T}$, and \[\preim{\mathcal{T}}_{i} = \mathcal{T}^{\circ} \cup (\mathcal{W}_{i} \ast \{5, 6\}) \cup (\mathcal{T}\backslash 5^{+} \ast 5) \cup (\mathcal{T} \backslash 5^{-} \ast 6).\] Hence one may compute that
\begin{align*}
\preim{\mathcal{T}}_{1} &= \{1234\} \cup \{1256, 2356, 3456\} \cup \{1245, 2345\} \cup \emptyset, \\
\preim{\mathcal{T}}_{2} &= \{1234\} \cup \{1256, 2456\} \cup \{1245\} \cup \{2346\}, \\
\preim{\mathcal{T}}_{3} &= \{1234\} \cup \{1456\} \cup \emptyset \cup \{1246, 2346\}.
\end{align*}
\end{example}

\begin{figure}
\caption{The triangulation $|\mathcal{T}|$ of $\creal{5}{3}$ and the triangulation $|\mathcal{T}\backslash 5|$}\label{fig:m_vertex_figure}
\[
\begin{tikzpicture}

\begin{scope}[shift={(-3,0)}]


\coordinate(1a) at (-2.77,2.25);
\coordinate(2a) at (-1.77,0.75);
\coordinate(3a) at (-0.77,0.25);
\coordinate(4a) at (0.23,0.75);

\coordinate(1b) at (-1.75,2.5);
\coordinate(2b) at (-0.75,1);
\coordinate(4b) at (1.15,1);
\coordinate(5b) at (2.25,2.5);

\coordinate(2c) at (-0.25,0.25);
\coordinate(3c) at (0.75,-0.25);
\coordinate(4c) at (1.75,0.25);
\coordinate(5c) at (2.75,1.75);



\path[name path = line 1] (1a) -- (4a);


\path[name path = line 2] (1b) -- (2b);
\path[name path = line 3] (2b) -- (4b);


\path [name intersections={of = line 1 and line 2}];
\coordinate (a)  at (intersection-1);

\path [name intersections={of = line 1 and line 3}];
\coordinate (b)  at (intersection-1);


\draw[fill=red!30, draw=none] (2a) -- (3a) -- (4a) -- (2a);
\draw[fill=green!30, draw=none] (1a) -- (2a) -- (2b) -- (a) -- (1a);
\draw[fill=green!30, draw=none] (2a) -- (2b) -- (b) -- (4a) -- (2a);
\draw[fill=blue!30, draw=none] (2b) -- (5b) -- (4b) -- (2b);
 
\draw[fill=cyc!30, draw=none] (1b) -- (2b) -- (5b) -- (1b);
\draw[fill=cyc!30, draw=none] (2c) -- (5c) -- (3c) -- (2c);
\draw[fill=cyc!30, draw=none] (3c) -- (4c) -- (5c) -- (3c);


\draw (1a) -- (a);
\draw[dotted] (a) -- (b);
\draw (b) -- (4a);


\draw(1a) -- (2a);
\draw (2a) -- (3a);
\draw (3a) -- (4a);
\draw (2a) -- (4a);
\draw[dotted] (1a) -- (3a); 


\draw (1b) -- (5b);
\draw (1b) -- (2b);
\draw (2b) -- (4b);
\draw (4b) -- (5b);
\draw (2b) -- (5b);
\draw[dotted] (1b) -- (4b);


\draw (2c) -- (5c);
\draw (2c) -- (3c);
\draw (3c) -- (4c);
\draw (4c) -- (5c);
\draw (3c) -- (5c);
\draw[dotted] (2c) -- (4c);


\node at (1a) {$\bullet$};
\node at (2a) {$\bullet$};
\node at (3a) {$\bullet$};
\node at (4a) {$\bullet$};

\node at (1b) {$\bullet$};
\node at (2b) {$\bullet$};
\node at (4b) {$\bullet$};
\node at (5b) {$\bullet$};

\node at (2c) {$\bullet$};
\node at (3c) {$\bullet$};
\node at (4c) {$\bullet$};
\node at (5c) {$\bullet$};

\end{scope}


\begin{scope}[shift={(3,0)}]


\coordinate(31) at (-2,2);
\coordinate(32) at (-1,0.5);
\coordinate(33) at (0,0);
\coordinate(34) at (1,0.5);
\coordinate(35) at (2,2);


\coordinate (v1) at (0,2);
\coordinate (v2) at (0.5,1.25);
\coordinate (v3) at (1,1);
\coordinate (v4) at (1.5,1.25);


\draw[fill=cyc!30,draw=none] (31) -- (32) -- (v2) -- (v1) -- (31);
\draw[fill=cyc!30,draw=none] (32) -- (v2) -- (v3) -- (33) -- (32);
\draw[fill=cyc!30,draw=none] (33) -- (v3) -- (v4) -- (34) -- (33);


\draw[green] (31) -- (32);
\draw[red] (32) -- (33);
\draw[red] (33) -- (34);
\draw[blue] (34) -- (v4);
\draw (v4) -- (v1) -- (31);
\draw (v3) -- (33);
\draw[blue] (v2) -- (32);
\draw (v1) -- (v2) -- (v3) -- (v4);
\draw[dotted] (31) -- (33);
\draw[dotted] (31) -- (34);


\draw[blue] (v2) -- (v4);


\node at (v1) {$\bullet$};
\node at (v2) {$\bullet$};
\node at (v3) {$\bullet$};
\node at (v4) {$\bullet$};


\node at (31) {$\bullet$};
\node at (32) {$\bullet$};
\node at (33) {$\bullet$};
\node at (34) {$\bullet$};


\node at (31) [left = 1mm of 31]{1};
\node at (32) [below left = 1mm of 32]{2};
\node at (33) [below = 1mm of 33]{3};
\node at (34) [below right = 1mm of 34]{4};

\end{scope}

\end{tikzpicture}
\]
\end{figure}

\begin{figure}
\caption{Sections of $|\mathcal{T}\backslash 5|$}\label{fig:mvf_sections}
\[
\begin{tikzpicture}

\begin{scope}[shift={(-4,0)}]


\coordinate(1) at (-2,2);
\coordinate(2) at (-1,0.5);
\coordinate(3) at (0,0);
\coordinate(4) at (1,0.5);


\draw (1) -- (2) -- (3) -- (4) -- (1);


\draw (2) -- (4);


\draw[ultra thick,secclr] (1) -- (2) -- (3) -- (4);


\node at (1) {$\bullet$};
\node at (2) {$\bullet$};
\node at (3) {$\bullet$};
\node at (4) {$\bullet$};


\node at (1) [left = 1mm of 1]{1};
\node at (2) [below left = 1mm of 2]{2};
\node at (3) [below = 1mm of 3]{3};
\node at (4) [below right = 1mm of 4]{4};


\node at (0,-1) {$\mathcal{W}_{1}$};

\end{scope}

\begin{scope}[shift={(0,0)}]


\coordinate(1) at (-2,2);
\coordinate(2) at (-1,0.5);
\coordinate(3) at (0,0);
\coordinate(4) at (1,0.5);


\draw (1) -- (2) -- (3) -- (4) -- (1);


\draw (2) -- (4);


\draw[ultra thick,secclr] (1) -- (2) -- (4);


\node at (1) {$\bullet$};
\node at (2) {$\bullet$};
\node at (3) {$\bullet$};
\node at (4) {$\bullet$};


\node at (1) [left = 1mm of 1]{1};
\node at (2) [below left = 1mm of 2]{2};
\node at (3) [below = 1mm of 3]{3};
\node at (4) [below right = 1mm of 4]{4};


\node at (0,-1) {$\mathcal{W}_{2}$};

\end{scope}

\begin{scope}[shift={(4,0)}]


\coordinate(1) at (-2,2);
\coordinate(2) at (-1,0.5);
\coordinate(3) at (0,0);
\coordinate(4) at (1,0.5);


\draw (1) -- (2) -- (3) -- (4) -- (1);


\draw (2) -- (4);


\draw[ultra thick,secclr] (1) -- (4);


\node at (1) {$\bullet$};
\node at (2) {$\bullet$};
\node at (3) {$\bullet$};
\node at (4) {$\bullet$};


\node at (1) [left = 1mm of 1]{1};
\node at (2) [below left = 1mm of 2]{2};
\node at (3) [below = 1mm of 3]{3};
\node at (4) [below right = 1mm of 4]{4};


\node at (0,-1) {$\mathcal{W}_{3}$};

\end{scope}

\end{tikzpicture}
\]
\end{figure}

We now show how one may apply the result of \cite[Lemma~4.7]{rs-baues} to prove Lemma~\ref{lem:subpolytope_exp}.

\begin{proof}[Proof of Lemma~\ref{lem:subpolytope_exp}]
As above, by \cite[Lemma~4.7(i)]{rs-baues}, triangulations $\preim{\mathcal{T}}$ of $C(m, n)$ such that $\preim{\mathcal{T}}[m - 1 \leftarrow m]=\mathcal{T}$ are in bijection with sections $\mathcal{W}$ of $\mathcal{T}\backslash m - 1$. Moreover, given a section $\mathcal{W}$ of $\mathcal{T}\backslash m - 1$, the corresponding triangulation $\preim{\mathcal{T}}$ has the set of $n$-simplices \[\mathcal{T}^{\circ} \cup (\mathcal{W} \ast \{m - 1, m\}) \cup (\mathcal{T}\backslash m - 1^{+}\ast (m - 1)) \cup (\mathcal{T} \backslash m - 1^{-} \ast m).\]

The set-up of Lemma~\ref{lem:subpolytope_exp} gives us that $\mathcal{T}$ contains a cyclic subpolytope $C(\U, n)$. Let $\mathcal{T}_{H}$ be the induced triangulation of the subpolytope $C(H, n)$ in $\mathcal{T}$. If $m - 1 \notin\U$, then we have $\mathcal{T}_{\U} \subseteq \mathcal{T}^{\circ}$, so that $C(\U, n)$ is a subpolytope of $\preim{\mathcal{T}}$. Hence, assume that $m - 1 \in \U$. This means that $C(\U \setminus m - 1, n)$ is a subpolytope of $\mathcal{T}\backslash m - 1$, with induced triangulation~$\mathcal{T}_{\U}\backslash m - 1$. There are then three options:
\begin{enumerate}
\item The subpolytope $\creal{\U \setminus m - 1}{n}$ lies below the section $|\mathcal{W}|$.\label{op:below}
\item The $\creal{\U \setminus m - 1}{n}$ lies above the section $|\mathcal{W}|$.\label{op:above}
\item The section $|\mathcal{W}|$ intersects $\creal{\U\setminus m - 1}{n}$.\label{op:intersect}
\end{enumerate}
In case (\ref{op:below}), we have that $\mathcal{T}_{\U}\backslash m - 1 \subseteq \mathcal{T}\backslash m - 1^{-}$, so that $C((\U\setminus m - 1) \cup m, n)$ is a subpolytope of $\preim{\mathcal{T}}$. In case (\ref{op:above}), we reason similarly that $C(\U, n)$ is a subpolytope of $\preim{\mathcal{T}}$.

For case (\ref{op:intersect}), we have that $\mathcal{W}$ induces a section $\mathcal{W}_{\U}$ of $\mathcal{T}_{\U}\backslash m - 1$. By \cite[Lemma~4.7(i)]{rs-baues}, we have that the section $\mathcal{W}_{\U}$ of $\mathcal{T}_{\U}$ gives us a triangulation $\preim{\mathcal{T}}_{\U}$ of $C(\U \cup m, n)$. Moreover, the triangulation $\preim{\mathcal{T}}_{\U}$ of $C(\U \cup m, n)$ has simplices \[\mathcal{T}_{\U}^{\circ} \cup (\mathcal{W}_{\U} \ast \{m - 1, m\}) \cup (\mathcal{T}_{\U}\backslash m - 1^{+}\ast (m - 1)) \cup (\mathcal{T}_{\U} \backslash m - 1^{-} \ast m).\] It is then clear that $\mathcal{T}_{\U}^{\circ} \subseteq \mathcal{T}^{\circ}$, $\mathcal{W}_{\U} \subseteq \mathcal{W}$, $\mathcal{T}_{\U}\backslash m - 1^{+} \subseteq \mathcal{T}\backslash m - 1^{+}$, and $\mathcal{T}_{\U}\backslash m - 1^{-} \subseteq \mathcal{T}\backslash m - 1^{-}$. Hence $\preim{\mathcal{T}}_{\U}$ is a subtriangulation of $\preim{\mathcal{T}}$, which gives us that $C(H \cup m, n)$ is a subpolytope of $\preim{\mathcal{T}}$.
\end{proof}

However, we also wish to prove the analogue of Lemma~\ref{lem:subpolytope_exp} for expansion at vertices other than $1$ and $m$, namely Lemma~\ref{lem:subpolytope_exp_oth_vert}. The difficulty is that \cite[Lemma~4.7(i)]{rs-baues} does not apply to these other vertices. Nevertheless, we consider the following example, which suggests that a version of \cite[Lemma~4.7(i)]{rs-baues} ought to hold at these vertices too. We spend most of the remainder of this section proving this more general version of \cite[Lemma~4.7(i)]{rs-baues}, which we then use to prove Lemma~\ref{lem:subpolytope_exp_oth_vert}.

\begin{example}\label{ex:v_exp}
We proceed in the opposite direction to Example~\ref{ex:m_exp}. That is, we consider the same triangulation~$\mathcal{T}$ of $C(5,3)$, but now directly compute the triangulations $\preim{\mathcal{T}}$ of $C([5]_{2+}, 3)$ such that $\preim{\mathcal{T}}[x \rightarrow 2 \leftarrow y] = \mathcal{T}$. We then analyse the triangulated vertex figure $\mathcal{T}\backslash 2$ to see if there is a correspondence.

By direct computation, there are four triangulations $\preim{\mathcal{T}}$ of $C([5]_{2+}, 3)$ such that $\preim{\mathcal{T}}[x \rightarrow 2 \leftarrow y] = \mathcal{T}$, namely
\begin{align*}
\preim{\mathcal{T}_{1}} &= \emptyset \cup \{1xy3, xy35\} \cup \{1x34, 1x45, x345\} \cup \emptyset, \\
\preim{\mathcal{T}_{2}} &= \emptyset \cup \{1xy3, xy34, xy45\} \cup \{1x34, 1x45\} \cup \{y345\}, \\
\preim{\mathcal{T}_{3}} &= \emptyset \cup \{1xy4, xy45\} \cup \{1x45\} \cup \{1y34, y345\}, \\
\preim{\mathcal{T}_{4}} &= \emptyset \cup \{1xy5\} \cup \emptyset \cup \{1y34, 1y45, y345\}.
\end{align*}
Here we have split up the simplices into sets according to whether they have the vertex $x, y$, or both. Now let \[\mathcal{W}_{i} = \set{A \st A \cup \{x, y\} \in \preim{\mathcal{T}}_{i}},\] so that
\begin{align*}
\mathcal{W}_{1} &= \{13, 35\},\\
\mathcal{W}_{2} &= \{13, 34, 45\},\\
\mathcal{W}_{3} &= \{14, 45\},\\
\mathcal{W}_{4} &= \{15\}.
\end{align*}
Consider these sets of simplices as subcomplexes of $\mathcal{T}\backslash 2$. We obtain the results shown in Figure~\ref{fig:ivf_sections} using geometric realisations. Note further that the simplices of the triangulation $\preim{\mathcal{T}}_{i}$ which have $x$ as a vertex correspond to the simplices of $\mathcal{T}\backslash 2$ which are above the section, and the simplices of the triangulation $\preim{\mathcal{T}}_{i}$ which possess $y$ as a vertex correspond to the simplices of $\mathcal{T}\backslash 2$ which are below the section.
\end{example}

\begin{figure}
\caption{The triangulation $|\mathcal{T}|$ of $\creal{5}{3}$ and the triangulation $|\mathcal{T}\backslash 2|$}\label{fig:i_vertex_figure}
\[
\begin{tikzpicture}

\begin{scope}[shift={(-3,0)}]


\coordinate(1a) at (-2.77,2.25);
\coordinate(2a) at (-1.77,0.75);
\coordinate(3a) at (-0.77,0.25);
\coordinate(4a) at (0.23,0.75);

\coordinate(1b) at (-1.75,2.5);
\coordinate(2b) at (-0.75,1);
\coordinate(4b) at (1.15,1);
\coordinate(5b) at (2.25,2.5);

\coordinate(2c) at (-0.25,0.25);
\coordinate(3c) at (0.75,-0.25);
\coordinate(4c) at (1.75,0.25);
\coordinate(5c) at (2.75,1.75);



\path[name path = line 1] (1a) -- (4a);


\path[name path = line 2] (1b) -- (2b);
\path[name path = line 3] (2b) -- (4b);


\path [name intersections={of = line 1 and line 2}];
\coordinate (a)  at (intersection-1);

\path [name intersections={of = line 1 and line 3}];
\coordinate (b)  at (intersection-1);


\draw[fill=red!30, draw=none] (2a) -- (3a) -- (4a) -- (2a);
\draw[fill=green!30, draw=none] (1a) -- (2a) -- (2b) -- (a) -- (1a);
\draw[fill=green!30, draw=none] (2a) -- (2b) -- (b) -- (4a) -- (2a);
\draw[fill=blue!30, draw=none] (2b) -- (5b) -- (4b) -- (2b);
 
\draw[fill=cyc!30, draw=none] (1b) -- (2b) -- (5b) -- (1b);
\draw[fill=cyc!30, draw=none] (2c) -- (5c) -- (3c) -- (2c);
\draw[fill=cyc!30, draw=none] (3c) -- (4c) -- (5c) -- (3c);


\draw (1a) -- (a);
\draw[dotted] (a) -- (b);
\draw (b) -- (4a);


\draw(1a) -- (2a);
\draw (2a) -- (3a);
\draw (3a) -- (4a);
\draw (2a) -- (4a);
\draw[dotted] (1a) -- (3a); 


\draw (1b) -- (5b);
\draw (1b) -- (2b);
\draw (2b) -- (4b);
\draw (4b) -- (5b);
\draw (2b) -- (5b);
\draw[dotted] (1b) -- (4b);


\draw (2c) -- (5c);
\draw (2c) -- (3c);
\draw (3c) -- (4c);
\draw (4c) -- (5c);
\draw (3c) -- (5c);
\draw[dotted] (2c) -- (4c);


\node at (1a) {$\bullet$};
\node at (2a) {$\bullet$};
\node at (3a) {$\bullet$};
\node at (4a) {$\bullet$};

\node at (1b) {$\bullet$};
\node at (2b) {$\bullet$};
\node at (4b) {$\bullet$};
\node at (5b) {$\bullet$};

\node at (2c) {$\bullet$};
\node at (3c) {$\bullet$};
\node at (4c) {$\bullet$};
\node at (5c) {$\bullet$};

\end{scope}


\begin{scope}[shift={(3,0)}]


\coordinate(31) at (-2,2);
\coordinate(32) at (-1,0.5);
\coordinate(33) at (0,0);
\coordinate(34) at (1,0.5);
\coordinate(35) at (2,2);


\coordinate (v1) at ($0.5*(31) + 0.5*(32)$);
\coordinate (v3) at ($0.75*(33) + 0.25*(32)$);
\coordinate (v4) at ($0.4*(34) + 0.6*(32)$);
\coordinate (v5) at ($0.5*(35) + 0.5*(32)$);


\path[name path = p13] (31) -- (33);
\path[name path = p14] (31) -- (34);
\path[name path = p24] (32) -- (34);
\path[name path = pv15] (v1) -- (v5);
\path[name path = pv35] (v3) -- (v5);


\path [name intersections={of = p13 and pv15}];
\coordinate (p13pv15)  at (intersection-1);

\path [name intersections={of = p13 and pv35}];
\coordinate (p13pv35)  at (intersection-1);

\path [name intersections={of = p14 and pv15}];
\coordinate (p14pv15)  at (intersection-1);

\path [name intersections={of = p14 and pv35}];
\coordinate (p14pv35)  at (intersection-1);

\path [name intersections={of = p24 and pv35}];
\coordinate (p24pv35)  at (intersection-1);


\draw[fill=cyc!30,draw=none] (31) -- (v1) -- (v5) -- (35) -- (31);
\draw[fill=cyc!30,draw=none] (v3) -- (v5) -- (35) -- (33) -- (v3);
\draw[fill=cyc!30,draw=none] (33) -- (35) -- (34) -- (33);


\draw[green] (31) -- (v1);
\draw (v1) -- (v3);
\draw[red] (v3) -- (33);
\draw[red] (33) -- (34);
\draw[blue] (34) -- (35);
\draw (35) -- (31);
\draw (v3) -- (v5) -- (v1);
\draw[blue] (v5) -- (35);
\draw (33) -- (35);
\draw[dotted] (31) -- (p13pv15);
\draw[dotted] (p13pv35) -- (33);
\draw[dotted] (31) -- (p14pv15);
\draw[dotted] (p14pv35) -- (34);


\draw[green] (v1) -- (v4);
\draw[red] (v3) -- (v4);
\draw[blue] (v5) -- (v4);


\node at (v1) {$\bullet$};
\node at (v3) {$\bullet$};
\node at (v5) {$\bullet$};


\node at (31) {$\bullet$};
\node at (33) {$\bullet$};
\node at (34) {$\bullet$};
\node at (35) {$\bullet$};


\node at (31) [left = 1mm of 31]{1};
\node at (33) [below = 1mm of 33]{3};
\node at (34) [below right = 1mm of 34]{4};
\node at (35) [right = 1mm of 35]{5};

\end{scope}

\end{tikzpicture}
\]
\end{figure}

\begin{figure}
\caption{Sections of $|\mathcal{T}\backslash 2|$}\label{fig:ivf_sections}
\[
\begin{tikzpicture}

\begin{scope}[scale=1.5]

\begin{scope}[shift={(-2,0)}]


\coordinate(31) at (-2,2);
\coordinate(32) at (-1,0.5);
\coordinate(33) at (0,0);
\coordinate(34) at (1,0.5);
\coordinate(35) at (2,2);


\coordinate (v1) at ($0.5*(31) + 0.5*(32)$);
\coordinate (v3) at ($0.75*(33) + 0.25*(32)$);
\coordinate (v4) at ($0.4*(34) + 0.6*(32)$);
\coordinate (v5) at ($0.5*(35) + 0.5*(32)$);


\draw (v3) -- (v5) -- (v1) -- (v3);


\draw (v1) -- (v4);
\draw (v3) -- (v4);
\draw (v5) -- (v4);


\node at (v1) [left = 1mm of v1]{$1$};
\node at (v3) [below = 1mm of v3]{$3$};
\node at (v4) [above = 1mm of v4]{$4$};
\node at (v5) [right = 1mm of v5]{$5$};


\draw[secclr,ultra thick] (v1) -- (v3) -- (v5);


\node at (v1) {$\bullet$};
\node at (v3) {$\bullet$};
\node at (v5) {$\bullet$};


\node at (-0.25,-0.6) {$\mathcal{W}_{1}$};

\end{scope}

\begin{scope}[shift={(2,0)}]


\coordinate(31) at (-2,2);
\coordinate(32) at (-1,0.5);
\coordinate(33) at (0,0);
\coordinate(34) at (1,0.5);
\coordinate(35) at (2,2);


\coordinate (v1) at ($0.5*(31) + 0.5*(32)$);
\coordinate (v3) at ($0.75*(33) + 0.25*(32)$);
\coordinate (v4) at ($0.4*(34) + 0.6*(32)$);
\coordinate (v5) at ($0.5*(35) + 0.5*(32)$);


\draw (v3) -- (v5) -- (v1) -- (v3);


\draw (v1) -- (v4);
\draw (v3) -- (v4);
\draw (v5) -- (v4);


\node at (v1) [left = 1mm of v1]{$1$};
\node at (v3) [below = 1mm of v3]{$3$};
\node at (v4) [above = 1mm of v4]{$4$};
\node at (v5) [right = 1mm of v5]{$5$};


\draw[secclr,ultra thick] (v1) -- (v3) -- (v4) -- (v5);


\node at (v1) {$\bullet$};
\node at (v3) {$\bullet$};
\node at (v5) {$\bullet$};


\node at (-0.25,-0.6) {$\mathcal{W}_{2}$};

\end{scope}

\begin{scope}[shift={(-2,-2.5)}]


\coordinate(31) at (-2,2);
\coordinate(32) at (-1,0.5);
\coordinate(33) at (0,0);
\coordinate(34) at (1,0.5);
\coordinate(35) at (2,2);


\coordinate (v1) at ($0.5*(31) + 0.5*(32)$);
\coordinate (v3) at ($0.75*(33) + 0.25*(32)$);
\coordinate (v4) at ($0.4*(34) + 0.6*(32)$);
\coordinate (v5) at ($0.5*(35) + 0.5*(32)$);


\draw (v3) -- (v5) -- (v1) -- (v3);


\draw (v1) -- (v4);
\draw (v3) -- (v4);
\draw (v5) -- (v4);


\node at (v1) [left = 1mm of v1]{$1$};
\node at (v3) [below = 1mm of v3]{$3$};
\node at (v4) [above = 1mm of v4]{$4$};
\node at (v5) [right = 1mm of v5]{$5$};


\draw[secclr,ultra thick] (v1) -- (v4) -- (v5);


\node at (v1) {$\bullet$};
\node at (v3) {$\bullet$};
\node at (v5) {$\bullet$};


\node at (-0.25,-0.6) {$\mathcal{W}_{3}$};

\end{scope}

\begin{scope}[shift={(2,-2.5)}]


\coordinate(31) at (-2,2);
\coordinate(32) at (-1,0.5);
\coordinate(33) at (0,0);
\coordinate(34) at (1,0.5);
\coordinate(35) at (2,2);


\coordinate (v1) at ($0.5*(31) + 0.5*(32)$);
\coordinate (v3) at ($0.75*(33) + 0.25*(32)$);
\coordinate (v4) at ($0.4*(34) + 0.6*(32)$);
\coordinate (v5) at ($0.5*(35) + 0.5*(32)$);


\draw (v3) -- (v5) -- (v1) -- (v3);


\draw (v1) -- (v4);
\draw (v3) -- (v4);
\draw (v5) -- (v4);


\node at (v1) [left = 1mm of v1]{$1$};
\node at (v3) [below = 1mm of v3]{$3$};
\node at (v4) [above = 1mm of v4]{$4$};
\node at (v5) [right = 1mm of v5]{$5$};


\draw[secclr,ultra thick] (v1) -- (v5);


\node at (v1) {$\bullet$};
\node at (v3) {$\bullet$};
\node at (v5) {$\bullet$};


\node at (-0.25,-0.6) {$\mathcal{W}_{4}$};

\end{scope}

\end{scope}

\end{tikzpicture}
\]
\end{figure}

This suggests that there ought to be a version of \cite[Lemma~4.7(i)]{rs-baues} for expansion at vertices $v$ such that $1 < v < m$. However, there are several outstanding issues.
\begin{enumerate}
\item The vertex figures $C(m, n)\backslash v$ are not generally cyclic polytopes for $1 < v < m$, as can be seen from Figure~\ref{fig:i_vertex_figure} and Figure~\ref{fig:ivf_sections}.\label{q:vf_not_cp}
\item It is not clear how to define the orientation on the vertex figure $C(m, n)\backslash v$, that is, how to decide what the upper and lower facets of $C(m, n) \backslash v$ are. The orientation of $\creal{5}{3}\backslash 2$ from Figure~\ref{fig:ivf_sections} may look very natural, with $|13|$ and $|35|$ as lower facets and $|15|$ as the sole upper facet. But it is not clear where this comes from, because $|123|$ is a lower facet of $\creal{5}{3}$, whereas $|125|$ and $|235|$ are upper facets.\label{q:vf_orient}
\item Likewise, it is not clear how to orient the simplices in the triangulation. From Figure~\ref{fig:ivf_sections}, it seems that $|345|$ has lower facet $|35|$ and upper facets $|34|,|45|$, whereas $|134|$ has lower facets $|13|,|34|$ and upper facet $|14|$. But is not immediately obvious what the basis for this is.\label{q:simp_orient}
\item Finally, it is not obvious how to define sections of a triangulation of the vertex figure $C(m, n)\backslash v$. Moreover, if one can define the right notion of a section, it is not clear whether such sections will be triangulations of lower-dimensional cyclic polytopes, given that the vertex figures themselves are not cyclic polytopes. \label{q:sect_triang}
\end{enumerate}
Over the course of this section we shall show how to resolve all these issues and prove the analogue of \cite[Lemma~4.7(i)]{rs-baues} at vertices which are not $1$ or $m$. We first show how one can orient the vertex figures $C(m, n)\backslash v$, that is, decide which the upper and lower facets of $C(m, n)\backslash v$ are. This explains the natural orientation we arrived at in Figure~\ref{fig:ivf_sections}, and solves issue (\ref{q:vf_orient}). Next we apply the same logic to the simplices of the triangulation, thereby answering (\ref{q:simp_orient}).

This gives us a partial order on the simplices of $\mathcal{T}\backslash v$. Using this, we derive the relevant notion of a section within the triangulation~$\mathcal{T}\backslash v$. We show that, in fact, our sections are triangulations of $C([m] \setminus v, n - 2)$, solving issue (\ref{q:sect_triang}). This culminates in our proving the following proposition, which is the analogue of \cite[Lemma~4.7(i)]{rs-baues}, showing that point (\ref{q:vf_not_cp}) is not a problem.

\begin{proposition}\label{prop:sections=expansions}
Let $\mathcal{T}$ be a triangulation of $C(m, n)$. There is a bijection between triangulations $\preim{\mathcal{T}}$ of $C([m]_{v+}, n)$ such that $\preim{\mathcal{T}}\xvy = \mathcal{T}$ and sections of $\mathcal{T} \backslash v$, given by
\begin{align*}
\adjset{\parbox{2.9cm}{\begin{center}$\preim{\mathcal{T}} \in \mathsf{S}([m]_{v+}, n) \st$ $\preim{\mathcal{T}}\xvy = \mathcal{T}$\end{center}}} &\longleftrightarrow \set{\parbox{2.9cm}{$\text{Sections }\mathcal{W}\text{ of }\mathcal{T}\backslash v$}}\\
\preim{\mathcal{T}} \quad &\longmapsto \quad \preim{\mathcal{T}}\backslash\{x, y\}\\
\parbox{3.7cm}{\begin{center}$\mathcal{T}^{\circ}  \cup (\mathcal{W} \ast \{x, y\})$ $\cup (\tvplusx) \cup (\tvminusy)$\end{center}} &\longmapsfrom \quad \mathcal{W}.
\end{align*}
\end{proposition}

\begin{remark}\label{rmk:odd_only}
It suffices to prove Proposition~\ref{prop:sections=expansions} for $n$ odd. For $n$ even it already follows from \cite[Lemma~4.7(i)]{rs-baues}, since in this case the cyclic permutation $i \mapsto i + (m - v)$ defines an automorphism of $C(m, n)$ which sends vertex $v$ to vertex $m$---see \cite{kw}. Hence one may apply \cite[Lemma~4.7(i)]{rs-baues} to the vertex $v$ as if it were vertex $m$, which gives Proposition~\ref{prop:sections=expansions} in this case. But for $n$ odd this permutation does not define an automorphism, and so more work needs to be done. Indeed, the fact that, for $n$ odd and $v \in [2, m - 1]$, $C(m, n)\backslash v$ is not a cyclic polytope precludes this permutation from giving an automorphism.

The methods of this section may still be applied to even-dimensional cyclic polytopes. One can check that this is equivalent to considering $C(m, 2d)$ subject to the given automorphism. However, restricting our attention to $n$ odd allows us to simplify some proofs.
\end{remark}

Proving Proposition~\ref{prop:sections=expansions} requires theory for working with triangulated vertex figures and their sections. Developing this theory is the task of Sections~\ref{sect:vf_facets}, \ref{sect:vf_simplices}, \ref{sect:vf_sections}, and \ref{sect:vf_exp}. The purpose of everything we prove in these sections is ultimately to derive Proposition~\ref{prop:sections=expansions}, and in turn to apply this proposition to prove Lemma~\ref{lem:subpolytope_exp_oth_vert}.

\begin{remark}
Proposition~\ref{prop:sections=expansions} and \cite[Lemma~4.7(i)]{rs-baues} are reminiscent of the single-element extension theorem of Las Vergnas for oriented matroids \cite{lv_ext} \cite[Section 7.1]{blswz} \cite[Theorem~4.1.(1)]{rgz_bd}. However, it does not seem that they follow from this result in any obvious way.
\end{remark}

\subsection{Facets of vertex figures}\label{sect:vf_facets}

Our first task is to find the correct orientation of the vertex figure $C(m, 2d + 1)\backslash v$, that is: a classification of its facets into lower facets and upper facets. It is important to note that, as in Example~\ref{ex:v_exp}, this will not generally match the orientation of $C(m, 2d + 1)$. That is to say, if $F$ is a lower facet of $C(m, 2d + 1) \backslash v$ according to our orientation, then $F \cup v$ will not generally be a lower facet of $C(m, 2d + 1)$.

Recall from Gale's Evenness Criterion that a facet $F$ of $C(m, 2d + 1)$ can be expressed uniquely as a union of disjoint pairs of consecutive numbers along with either $1$ or $m$. Hence, given $v \in [2, m - 1]$, and a facet $F$ of $C(m, n)$ such that $v \in F$, we can talk about the pair of consecutive entries that $v$ lies in, which must either be $\{v - 1, v\}$ or $\{v, v + 1\}$. We then define the upper and lower facets of $C(m, 2d + 1)\backslash v$ as follows.

\begin{definition}\label{def:poly_synth_facets}
Let $v \in [2, m - 1]$ and let $F$ be a facet of $C(m, 2d + 1)$ such that $v \in F$. Then $F \setminus v$ is a facet of $C(m, 2d + 1) \backslash v$.
\begin{itemize}
\item If the other element in the pair with $v$ in $F$ is $v + 1$, then we say that $F \setminus v$ is a \emph{lower facet of $C(m, 2d + 1)\backslash v$}.
\item If the other element in the pair with $v$ in $F$ is $v - 1$, then we say that $F \setminus v$ is an \emph{upper facet of $C(m, 2d + 1)\backslash v$}.
\end{itemize}
\end{definition}

The following lemma indicates why our orientation of the vertex figure $C(m, 2d + 1)\backslash v$ is the correct one when it comes to considering expansion. A lower facet $F$ of $C(m, 2d + 1)\backslash v$ should always lie below the section, and hence should always become a facet $F \cup y$ of $C([m]_{v+}, 2d + 1)$. 

\begin{lemma}\label{lem:facets_under_exp}
We have that $F$ is a lower facet of $C(m, 2d + 1)\backslash v$ if and only if $F \cup y$ is a facet of $C([m]_{v+}, 2d + 1)$. Dually, we have that $F$ is an upper facet of $C(m, 2d + 1)\backslash v$ if and only if $F \cup x$ is a facet of $C([m]_{v+}, 2d + 1)$.
\end{lemma}
\begin{proof}
If $F \cup v$ is an even subset of $[m]$ and $v$ is in a pair with $v + 1$, then $F \cup y$ will be an even subset of $[m]_{v+}$. Conversely, if $F \cup y$ is either an even or an odd subset of $[m]_{v+}$ where $F \subseteq [m]\setminus v$, then $y$ must be in a pair with $v + 1$, since $x \notin F$. Consequently $F \cup v$ is either an even or an odd subset of $[m]$ with $v$ in a pair with $v + 1$. The analogous claim for upper facets follows by a similar argument.
\end{proof}

One can also describe the upper and lower facets of $C(m, 2d + 1)\backslash v$ using the following, which can be seen as a generalisation of Gale's Evenness Criterion.

\begin{lemma}\label{lem:evenness_facet_desc}
Let $F \subseteq [m]\setminus v$. Then
\begin{itemize}
\item $F$ is a lower facet of $C(m, 2d + 1)\backslash v$ if and only if $\# \set{\aI \in F \st \bI < \aI < v}$ is even for all $\bI \in [v - 1]\setminus F$ and $\# \set{\aI \in F \st v < \aI < \bI}$ is odd for all $\bI \in [v + 1, m]\setminus F$; and
\item $F$ is an upper facet of $C(m, 2d + 1)\backslash v$ if and only if $\# \set{\aI \in F \st \bI < \aI < v}$ is odd for all $\bI \in [v - 1]\setminus F$ and $\# \set{\aI \in F \st v < \aI < \bI}$ is even for all $\bI \in [v + 1, m]\setminus F$.
\end{itemize}
\end{lemma}
\begin{proof}
We only show the first claim, since the second claim is similar. Suppose that $F$ is a lower facet of $C(m, 2d + 1)\backslash v$. Then $F \cup v$ is a facet of $C(m, 2d + 1)$ and $v$ occurs in a pair with $v + 1$. Let $\bI \in [v - 1] \setminus F$. There are then a whole number of pairs of consecutive numbers between $\bI$ and $v$, so $\# \set{\aI \in [m] \st \bI < \aI < v}$ is even. Let $\bI \in [v + 1, m] \setminus F$. Then the elements of $F$ between $\bI$ and $v$ consist of $v + 1$ and a set of pairs of consecutive numbers, so $\# \set{\aI \in [m] \st v < \aI < \bI}$ is odd.

Suppose now that $F$ is such that $\# \set{\aI \in [m] \st \bI < \aI < v}$ is even for all $\bI \in [v - 1]\setminus F$ and $\# \set{\aI \in [m] \st v < \aI < \bI}$ is odd for all $\bI \in [v + 1, m]\setminus F$. Then we must have $v + 1 \in F$, since otherwise we can choose $\bI = v + 1$, and $\# \set{\aI \in [m] \st v < \aI < v+1} = 0$. The remaining elements of $F$ must consist of disjoint pairs of consecutive numbers and possibly $1$ or $m$, otherwise we can find gaps in $F$ which contradict our assumption. Moreover, $F$ cannot contain both $1$ and $m$, since $F$ must have $2d + 1$ elements. This gives that $F \cup v$ is a facet of $C(m, 2d + 1)$ where $v$ occurs in a pair with $v + 1$. Hence $F$ is a lower facet of $C(m, 2d + 1)\backslash v$.
\end{proof}

The following lemma describes the significance of the intersections of upper and lower facets of the vertex figure $C(m, n)\backslash v$. It is analogous to the easily-verified fact that the facets of $C(m, n)$ correspond precisely to the $(n - 1)$-simplices which are intersections of a lower facet and an upper facet of $C(m, n + 1)$.

\begin{lemma}\label{lem:ul_facet_intersect}
If $G \in \binom{[m] \setminus v}{2d - 1}$ and $G = F \cap F'$, where $F$ is an upper facet of $C(m, 2d + 1)\backslash v$ and $F'$ is a lower facet of $C(m, 2d + 1) \backslash v$, then $G$ is a facet of $C([m] \setminus v, 2d - 1)$.
\end{lemma}
\begin{proof}
Let
\begin{align*}
\bI &= \max\set{\aI \in [m]\setminus F \st \aI < v},\\
\bI' &= \max\set{\aI \in [m]\setminus F' \st \aI < v},\\
\cI &= \min\set{\aI \in [m]\setminus F \st v < \aI},\\
\cI' &= \min\set{\aI \in [m]\setminus F' \st v < \aI}. 
\end{align*}
We cannot have $\bI = \bI'$, since $\# \set{\aI \in F \st \bI < \aI < v}$ is odd, whereas $\# \set{\aI \in F' \st \bI' < \aI < v}$ is even. Therefore, suppose that $\bI < \bI'$. This implies that $\bI' \in F$, so that $F = G \cup \bI'$. Hence, $\cI' \notin F$, so that $\cI < \cI'$ and $F' = G \cup \cI'$.

Then $\#\set{\aI \in G \st \bI' < \aI < v} = \#\set{\aI \in F' \st \bI' < \aI < v}$, which is even, and $\#\set{\aI \in G \st v < \aI < \cI} = \#\set{\aI \in F \st v < \aI < \cI}$, which is also even. Furthermore, $\#\set{\aI \in G \st \bI < \aI < \bI'}$ is even, since $\#\set{\aI \in F \st \bI < \aI < v}$ is odd and $\#\set{\aI \in F' \st \bI' < \aI < v}$ is even. Similarly, $\#\set{\aI \in G \st \cI < \aI < \cI'}$ is even.

Thus, $G \cap [1, \bI]$ consists of a set of disjoint pairs along with possibly $1$, since this is true of $F$ and $F'$; $G \cap [\bI, \bI']$ is an interval of even length; $G \cap [\bI', \cI]$ is an interval in $[m] \setminus v$ of even length; $G \cap [\cI, \cI']$ is an interval of even length; and $G \cap [\cI', m]$ consists of a disjoint union of pairs, along with possibly $m$. Consequently, $G$ satisfies Gale's Evenness Criterion, and so is a facet of $C([m]\setminus v, 2d + 1)$. The case where $\bI' < \bI$ is similar.
\end{proof}

\begin{corollary}\label{cor:ul_facet_intersect}
If $G = F \cap F'$, where $F$ is an upper facet of $C(m, 2d + 1)\backslash v$ and $F'$ is a lower facet of $C(m, 2d + 1) \backslash v$, then $G \cup \{x, y\}$ is a facet of $C([m]_{v+}, 2d + 1)$.
\end{corollary}
\begin{proof}
This follows from Gale's Eveness Criterion and Lemma~\ref{lem:ul_facet_intersect}. If $G$ is an even (respectively, odd) subset of $[m] \setminus v$, then $G \cup \{x,y\}$ is an even (respectively, odd) subset of $[m]_{v+}$. Adding a pair of consecutive entries cannot change the parities of any gaps.
\end{proof}

\subsection{Orienting the simplices}\label{sect:vf_simplices}

We now show how one can similarly orient the simplices of the triangulation~$\mathcal{T}\backslash v$, which allows us to introduce a partial order on these simplices.

We first explain the logic of our orientation of the simplices of $\mathcal{T}\backslash v$. Given a triangulation~$\mathcal{T}$ of $C(m - 1, 2d + 1)$, we wish to understand the different triangulations $\preim{\mathcal{T}}$ of $C([m - 1]_{v+}, 2d + 1)$ such that $\preim{\mathcal{T}}\xvy = \mathcal{T}$. We consider the triangulated vertex figure $\mathcal{T}\backslash v$. It is clear that $\mathcal{T}\backslash v$ contains $\preim{\mathcal{T}}\backslash \{x, y\}$ as a simplicial subcomplex. We would like to think of these simplicial subcomplexes as sections which divide $\mathcal{T}\backslash v$ into a part where $x \leftarrow v$ under expansion and a part where $v \rightarrow y$ under expansion.

However, it is not clear geometrically which part of $\mathcal{T}\backslash v$ lies above $\preim{\mathcal{T}}\backslash \{x, y\}$ and which part lies below. Hence, we look to characterise this combinatorially instead. Note that for every $2d$-simplex $\S$ of $\mathcal{T}\backslash v$, we must have that $\S \cup v$ is a simplex of~$\mathcal{T}$, so that $\S \cup x$ is a $(2d + 1)$-simplex of~$\preim{\mathcal{T}}$, or that $\S \cup y$ is a $(2d + 1)$-simplex of~$\preim{\mathcal{T}}$. But we cannot have both, since $\S \cup \{x, y\}$ can be decomposed into two halves of a circuit, one of which is contained in $\S \cup x$, and the other of which is contained in $\S \cup y$, since $x$ and $y$ are adjacent in $\S \cup \{x, y\}$. We therefore orient the simplices of the triangulated vertex figure $\mathcal{T}\backslash v$ as follows.

\begin{definition}\label{def:synthetic_facets}
Let $\S$ be a $2d$-simplex of $\mathcal{T}\backslash v$. Then $\S \cup \{x, y\}$ consists of $2d + 3$ distinct vertices, and so uniquely gives two halves of a circuit of $C([m]_{v+}, 2d + 1)$, which we denote $(\S_{-} \cup x, \S_{+} \cup y)$. Then we say that $\S\setminus \s$ is an \emph{lower facet} of $\S$ if $\s \in \S_{+}$, and an \emph{upper facet} of $\S$ if $\s \in \S_{-}$.
\end{definition}

One can also translate this definition of the upper and lower facets of simplices of $\mathcal{T}\backslash v$ into an evenness criterion, which can be deduced straightforwardly from the definition.

\begin{lemma}\label{lem:even_simp_desc}
Let $\S$ be a $2d$-simplex of $\mathcal{T}\backslash v$ and let $\s \in \S$. Then
\begin{enumerate}
\item if $\s < v$, then $\S \setminus \s$ is
	\begin{enumerate}
	\item a lower facet of $\S$ if $\# \set{\aI \in \S \st \s < \aI < v}$ is even, and
	\item an upper facet of $\S$ if $\# \set{\aI \in \S \st \s < \aI < v}$ is odd; 
	\end{enumerate}
\item if $v < \s$, then $\S \setminus \s$ is
	\begin{enumerate}
	\item a lower facet of $\S$ if $\# \set{\aI \in \S \st v < \aI < \s}$ is odd, and
	\item an upper facet of $\S$ if $\# \set{\aI \in \S \st v < \aI < \s}$ is even. 
	\end{enumerate}
\end{enumerate}
\end{lemma}

By comparing with Lemma~\ref{lem:evenness_facet_desc}, we see that our notion of the upper and lower facets of a simplex of $\mathcal{T}\backslash v$ matches our notion of the upper and lower facets of $C(m, 2d + 1)\backslash v$.

\begin{lemma}\label{lem:local_stackability}
Let $\mathcal{T}$ be a triangulation of $C(m, 2d + 1)$. Let $\S, \R$ be $2d$-simplices of $\mathcal{T}\backslash v$. Then $\S \cap \R$ cannot be both a lower facet of $\S$ and a lower facet of $\R$. Similarly, $\S \cap \R$ cannot be both an upper facet of $\S$ and an upper facet of $\R$.
\end{lemma}
\begin{proof}
We only show the first claim, since the second claim is similar. Suppose that $F$ is both a lower facet of $\S = F \cup \s$ and a lower facet of $\R = F \cup \sR$. Without loss of generality, assume that $\s < \sR$. If $\s < \sR < v$ or $v < \s < \sR$, then are are an even number of elements $f \in F$ such that $\s < f < \sR$, by Lemma~\ref{lem:even_simp_desc}. If $\s < v < \sR$, then there are an odd number of elements $f \in F$ such that $\s < f < \sR$, by Lemma~\ref{lem:even_simp_desc}.

We have $\# F \cup \{\s, \sR, v\} = 2d + 3$, and so there is a circuit $(Z, Z')$ of $C(m, 2d + 1)$ such that $Z \cup Z' = F \cup \{\s, \sR, v\}$. Suppose, without loss of generality, that $\s \in Z$. By the previous paragraph, we must then have $\sR \in Z'$. Hence the simplices $F \cup \{\s, v\}$ and $F \cup \{\sR, v\}$ each contain one half of a circuit, which contradicts their both being simplices of $\mathcal{T}$.
\end{proof}

In the manner of \cite[Definition~5.7]{rambau}, we may now define a relation on the set of $2d$-simplices of $\mathcal{T}\backslash v$. Given two $2d$-simplices $\S, \R$, we write that $\S \vcover \R$ if and only if $\S \cap \R$ is an upper facet of $\S$ and a lower facet of $\R$. We show that $\vorder$ is a partial order using the method of \cite[Corollary 5.8]{rambau}: we define a total order on the simplices of $\mathcal{T}\backslash v$ and show that $\vorder$ is a sub-order of it. This means that every triangulation of $C(m, 2d + 1)\backslash v$ which comes from a triangulation of $C(m, 2d + 1)$ is \emph{stackable}, in the sense of \cite[Definition~2.13]{rs-baues}.

To each $\S \in \binom{[m]\setminus v}{2d + 1}$, we assign a unique string by
\begin{align*}
\Gamma \colon \binom{[m]\setminus v}{2d + 1} &\to \{o, \ast, e\}^{m - 1} \\
\Gamma(\S) &:= (\gamma_{v + 1}(\S), \gamma_{v + 2}(\S), \dots, \gamma_{m}(\S), \gamma_{1}(\S), \gamma_{2}(\S), \dots, \gamma_{v - 1}(\S)),
\end{align*}
where
\[
\gamma_{j}(\S) = 
\begin{cases}
\ast \text{ if } j \in \S \\
\text{ if } j \notin \S
\begin{cases}
\text{ if } j < v
\begin{cases}
e \text{ if } \#\set{b \in \S \st j < b < v} \text{ is even,} \\
o \text{ if } \#\set{b \in \S \st j < b < v} \text{ is odd,}
\end{cases}
\\
\text{ if } v < j
\begin{cases}
e \text{ if } \#\set{b \in \S \st v < b < j} \text{ is even,} \\
o \text{ if } \#\set{b \in \S \st v < b < j} \text{ is odd.}
\end{cases}
\end{cases}
\end{cases}
\]
We then denote by $\preceq$ the lexicographic order on $\binom{[m]\setminus v}{2d + 1}$ induced by $\Gamma$ and the ordering of the letters $o \prec \ast \prec e$.

\begin{lemma}\label{lem:gap_analysis}
Let $\mathcal{T}$ be a triangulation of $C(m, 2d + 1)$ and consider the triangulated vertex figure $\mathcal{T} \backslash v$. Let $\S$ and $\R$ be $2d$-simplices of $\mathcal{T} \backslash v$ such that $\S \vcover \R$, with $\S \setminus \{\s\} = \R \setminus \{\sR\}$.
\begin{enumerate}
\item If we have $v < \s < \sR$ in the cyclic ordering, then $\gamma_{\sR}(\S) = e$ and $\gamma_{\s}(\R) = e$.\label{op:vab}
\item If we have $v < \sR < \s$ in the cyclic ordering, then $\gamma_{\sR}(\S) = o$ and $\gamma_{\s}(\R) = o$.\label{op:vba}
\item For $\bI \notin \S \cup \R$, we have that $\gamma_{\bI}(\S) \neq \gamma_{\bI}(\R)$ if and only if $\bI$ lies between $\s$ and $\sR$ in the ordering $v + 1, v + 2, \dots, n, 1, 2, \dots, v - 1$.\label{op:between_gaps}
\end{enumerate}
\end{lemma}
\begin{proof}
By Lemma~\ref{lem:even_simp_desc}, the fact that $\S \vcover \R$ implies that
\begin{align*}
\text{if } \s < v \text{, then } &\# \set{\aI \in \S \st \s < \aI < v} \text{ is odd, and} \\
\text{if } v < \s \text{, then } &\# \set{\aI \in \S \st v < \aI < \s} \text{ is even},
\end{align*}
and
\begin{align*}
\text{if } \sR < v \text{, then } &\# \set{\aI \in \R \st \sR < \aI < v} \text{ is even, and} \\
\text{if } v < \sR \text{, then } &\# \set{\aI \in \R \st v < \aI < \sR} \text{ is odd.}
\end{align*}

We consider the case where $v < \s < \sR$ is a cyclic ordering.
\begin{enumerate}[wide,  labelindent=\parindent]
\item Within this set of cases, we first suppose that $v < \s < \sR$. Then \[\# \set{\aI \in \R \st v < \aI < \s} = \# \set{\aI \in \S \st v < \aI < \s},\] which is even, and \[\# \set{\aI \in \S \st v < \aI < \sR} = \# \set{\aI \in \R \st v < \aI < \sR} + 1,\] which is even. Therefore $\gamma_{\s}(\R) = e$ and $\gamma_{\sR}(\S) = e$. Moreover, if $\bI \notin \S \cup \R$, then $\gamma_{\bI}(\S) \neq \gamma_{\bI}(\R)$ if and only if $\s < \bI < \sR$ in the cyclic ordering.
\item If $\sR < v < \s$, then \[\# \set{\aI \in \R \st v < \aI < \s} = \# \set{\aI \in \S \st v < \aI < \s},\] which is even, and \[\# \set{\aI \in \S \st \sR < \aI < v} = \# \set{\aI \in \R \st \sR < \aI < v},\] which is even. Therefore $\gamma_{\s}(\R) = e$ and $\gamma_{\sR}(\S) = e$. Moreover, if $\bI \notin \S \cup \R$, then $\gamma_{\bI}(\S) \neq \gamma_{\bI}(\R)$ if and only if $\s < \bI < \sR$ in the cyclic ordering.
\item If $\s < \sR < v$, then \[\#\set{\aI \in \R \st \s < \aI < v} = \# \set{\aI \in \S \st \s < \aI < v} - 1,\] which is even, and \[\# \set{\aI \in \S \st \sR < \aI < v} = \# \set{\aI \in \R \st \sR < \aI < v},\] which is even. Therefore $\gamma_{\s}(\R) = e$ and $\gamma_{\sR}(\S) = e$. Moreover, if $\bI \notin \S \cup \R$, then $\gamma_{\bI}(\S) \neq \gamma_{\bI}(\R)$ if and only if $\s < \bI < \sR$ in the cyclic ordering.
\end{enumerate}
The cases where $v < \sR < \s$ is a cyclic ordering are similar.
\end{proof}

\begin{corollary}\label{cor:partial_order}
The relation $\vorder$ is a partial order.
\end{corollary}
\begin{proof}
We show that $\S \vorder \R$ implies that $\S \preceq \R$. For this it suffices to show that $\S \vcover \R$ implies that $\S \preceq \R$. If $v < \s < \sR$ in the cyclic ordering, then, by Lemma~\ref{lem:gap_analysis}(\ref{op:between_gaps}), it suffices to consider $\gamma_{\s}(\S)$ and $\gamma_{\s}(\R)$ in order to compare $\Gamma(\S)$ and $\Gamma(\R)$ in the lexicographic order, since this is the first entry that differs. Then we have $\gamma_{\s}(\S) = \ast$ and $\gamma_{\s}(\R) = e$ by Lemma~\ref{lem:gap_analysis}(\ref{op:vab}) so that $\S \preceq \R$. Similarly, if $v < \sR < \s$ in the cyclic ordering, then we consider $\gamma_{\sR}(\S) = o$ and $\gamma_{\sR}(\R) = \ast$, so that $\S \preceq \R$ likewise. We conclude that $\S \vorder \R$ implies that $\S \preceq \R$. This entails that $\vorder$ is a partial order, since $\preceq$ is a total order.
\end{proof}

Recall that $\mathcal{L}$ is a \emph{lower set} for a partial order $\leqslant$ on a set $\mathcal{P}$ if $\mathcal{L}$ is a subset of $\mathcal{P}$ such that whenever $p \in \mathcal{L}$ and $p' \leqslant p$, we also have $p' \in \mathcal{L}$. The notion of an \emph{upper set} of a partial order is defined dually. These concepts, together with our partial order $\vorder$, allow us to characterise the set of simplices in $\mathcal{T}\backslash v$ where $x \leftarrow v$ under expansion and the set of simplices where $v \rightarrow y$ under expansion.

\begin{lemma}\label{lem:lu_set}
Let $\mathcal{T}$ be a triangulation of $C(m, 2d + 1)$ with $\preim{\mathcal{T}}$ a triangulation of $C([m]_{v+}, 2d + 1)$ such that $\preim{\mathcal{T}}\xvy = \mathcal{T}$. Let $\mathcal{L}$ be the set of $2d$-simplices $\S$ of $\mathcal{T} \backslash v$ such that $\S \cup y$ is a $(2d + 1)$-simplex of $\preim{\mathcal{T}}$ and let $\mathcal{U}$ be the set of $2d$-simplices $\R$ of $\mathcal{T} \backslash v$ such that $\R \cup x$ is a $(2d + 1)$-simplex of $\preim{\mathcal{T}}$. Then $\mathcal{L}$ is a lower set for $\vorder$, $\mathcal{U}$ is an upper set for $\vorder$, and $\mathcal{L} \cup \mathcal{U} = \mathcal{T} \backslash v$ with $\mathcal{L} \cap \mathcal{U} = \emptyset$.
\end{lemma}
\begin{proof}
It is clear that $\mathcal{L} \cup \mathcal{U}$ must comprise all of the $2d$-simplices of $\mathcal{T} \backslash v$. This is because if $\S$ is a $2d$-simplex of $\mathcal{T} \backslash v$, then $\S \cup v$ is a $(2d + 1)$-simplex of $\mathcal{T}$, and so either $\S \cup x$ or $\S \cup y$ is a $(2d + 1)$-simplex of $\preim{\mathcal{T}}$. Then, as we argued earlier, we cannot have $\mathcal{L} \cap \mathcal{U} \neq \emptyset$, since then both $\S \cup x$ and $\S \cup y$ are $(2d + 1)$-simplices of $\preim{\mathcal{T}}$. But this is prevented by the circuit $\relcirc{\S}$ of Definition~\ref{def:synthetic_facets}, as $\S \cup x \supseteq \S_{-} \cup x$ and $\S \cup y \supseteq \S_{+} \cup y$.

We now show that $\mathcal{L}$ is a lower set for $\vorder$. We suppose that $\R \in \mathcal{L}$ and $\S \in \mathcal{T}\backslash v$ are such that $\S \vcover \R$. Let $F = \S \cap \R$, which is an upper facet of $\S$ and a lower facet of $\R$. Suppose for contradiction that $\S \cup x$ is a $(2d + 1)$-simplex of $\preim{\mathcal{T}}$. Then $F \cup x$ is a $2d$-simplex of $\preim{\mathcal{T}}$. Since $F$ is a lower facet of $\R$, we have that $F \cup x \supseteq \R_{-} \cup x$, where $\relcirc{\R}$ is the circuit from Definition~\ref{def:synthetic_facets}. Since $\R \in \mathcal{L}$, we have that $\R \cup y \in \preim{\mathcal{T}}$. But then, $\R \cup y \supseteq \R_{+} \cup y$, so that both halves of $\relcirc{\R}$ are contained in simplices of $\preim{\mathcal{T}}$, which is a contradiction. Hence $\CL$ is a lower set, which also implies that $\CU$ is an upper set.
\end{proof}

\subsection{Sections of vertex figures}\label{sect:vf_sections}

We now show how our partial order on the $2d$-simplices of $\mathcal{T}\backslash v$ allows us to define the notion of a section of $\mathcal{T}\backslash v$. We then prove fundamental properties of sections of $\mathcal{T}\backslash v$ which will enable us to prove that they are in bijection with triangulations $\preim{\mathcal{T}}$ of $C([m]_{v+}, 2d + 1)$ such that $\preim{\mathcal{T}}\xvy = \mathcal{T}$.

\begin{definition}\label{def:synth_sect}
Given a lower set $\mathcal{L}$ of $(\mathcal{T}\backslash v, \vorder)$, we let $\mathcal{U} = (\mathcal{T}\backslash v)\setminus \mathcal{L}$ be the upper set which is its complement, and define the associated \emph{section} $\synth$ to be the abstract simplicial complex given by the set of $(2d - 1)$-simplices $\Wa$ of $\mathcal{T} \backslash v$ such that either
\begin{itemize}
\item $\Wa = A \cap B$ where $A \in \mathcal{L}$ and $B \in \mathcal{U}$, or
\item $\Wa$ is an upper facet of $C(m, 2d + 1) \backslash v$ and an upper facet of $A \in \mathcal{L}$, or
\item $\Wa$ is a lower facet of $C(m, 2d + 1) \backslash v$ and a lower facet of $B \in \mathcal{U}$.
\end{itemize}
\end{definition}

We show that sections of $\mathcal{T}\backslash v$ are triangulations of $C([m] \setminus v, 2d - 1)$, just as sections of $\mathcal{T}\backslash m$ are triangulations of $C(m - 1, n - 2)$ for triangulations $\mathcal{T}$ of $C(m, n)$. Recall our notation for facets of cyclic polytopes and facets of vertex figures of cyclic polytopes from Definition~\ref{def:comb_cyc_poly} and Definition~\ref{def:comb_vert_fig}, respectively.

\begin{lemma}\label{lem:sects=triangs}
For a triangulation~$\mathcal{T}$ of $C(m, 2d + 1)$, sections of $\mathcal{T}\backslash v$ are triangulations of $C([m]\setminus v, 2d - 1)$.
\end{lemma}
\begin{proof}
We prove the claim by induction on $\# \CL$. In the base case, we have that $\CL =\nolinebreak\emptyset$, so that $\synth = \vlfacets{[m]\setminus v}{2d + 1}$. Hence, we must show that $\vlfacets{[m]\setminus v}{2d + 1}$ is a triangulation of $C([m]\setminus v, 2d - 1)$. We must first show that there is no circuit $(A, B)$ of $C([m]\setminus v, 2d - 1)$ such that $A$ and $B$ are both faces of simplices in $\vlfacets{[m]\setminus v}{2d + 1}$. If this were the case, then either $(A \cup x, B \cup y)$ or $(A \cup y, B \cup x)$ would be a circuit of $C([m]_{v+}, 2d + 1)$. But this contradicts Lemma~\ref{lem:facets_under_exp}, which gives that $A \cup y$ and $B \cup y$ must be contained in lower facets of $C([m]_{v+}, 2d + 1)$, which cannot contain halves of circuits.

We now show that the facets of the $(2d - 1)$-simplices in $\vlfacets{[m]\setminus v}{2d + 1}$ are either shared with other $(2d - 1)$-simplices of $\vlfacets{[m]\setminus v}{2d + 1}$, or are facets of $C([m]\setminus v, 2d - 1)$. Let $S \in \vlfacets{[m]\setminus v}{2d + 1}$ and let $s \in S$, so that $S \setminus s$ is a facet of $S$. We have that $S \cup v$ is a facet of $C(m, 2d + 1)$, so we must have that $(S \setminus s) \cup v = (S \cup v) \cap (\R \cup v)$ for a facet $\R \cup v$ of $C(m, 2d + 1)$. Hence $S \setminus s = S \cap \R$ for a facet $\R \in \vfacets{[m] \setminus v}{2d + 1}$. If $\R \in \vlfacets{[m]\setminus v}{2d + 1}$, then we are done. Otherwise, $\R \in \vufacets{[m]\setminus v}{2d + 1}$, and so $S \setminus s = S \cap \R$ is a facet of $C([m]\setminus v, 2d - 1)$ by Lemma~\ref{lem:ul_facet_intersect}. This establishes the base case.

Now, to show the inductive step, we suppose that we have a section $\synth$ such that $\# \CL \neq \emptyset$. Choose a simplex $S \in \CL$ which is maximal in $\CL$ with respect to $\vorder$. Then $\CL' := \CL \setminus S$ is a lower set of $\vorder$ and, by the induction hypothesis, $\synthb$ is a triangulation of $C([m] \setminus v, 2d - 1)$. It follows from Definition~\ref{def:synth_sect} and the fact that $\S$ is maximal in $\CL$ that $\synth = (\synthb \setminus \vlfacets{S}{2d + 1}) \cup \vufacets{S}{2d + 1}$.

Let $\even{S} = \tup{s_{0}, s_{2}, \dots, s_{2d}}$ and $\odd{S} = \tup{s_{1}, s_{3}, \dots, s_{2d - 1}}$. Then either $(\even{S} \cup x, \odd{S} \cup y)$ is a circuit of $C([m]_{v+}, 2d + 1)$  or $(\odd{S} \cup x, \even{S} \cup y)$ is a circuit of $C([m]_{v+}, 2d + 1)$. Hence, either $\vlfacets{S}{2d + 1} = \set{S\setminus s \st s \in \odd{S}} = \ufacets{S}{2d - 1}$, or $\vlfacets{S}{2d + 1} = \set{S \setminus s \st s \in \even{S}} = \lfacets{S}{2d - 1}$. Then, respectively, either $\synth = (\synthb \setminus \ufacets{S}{2d - 1}) \cup \lfacets{S}{2d - 1})$ or $\synth = (\synthb \setminus \lfacets{S}{2d - 1}) \cup \ufacets{S}{2d - 1})$. In the former case, $\synth$ is an increasing bistellar flip of $\synthb$ as a triangulation of $C([m] \setminus v, 2d - 1)$; in the latter case, $\synth$ is a decreasing bistellar flip of $\synth$ as a triangulation of $C([m] \setminus v, 2d - 1)$. Since bistellar flips send triangulations of $C([m] \setminus v, 2d - 1)$ to triangulations of $C([m] \setminus v, 2d - 1)$, we have in either case that $\synth$ is a triangulation of $C([m]\setminus v, 2d - 1)$. The result then follows by induction.
\end{proof}

We obtain the following result, which will be useful in showing how sections of $\mathcal{T}\backslash v$ correspond to expanded triangulations.

\begin{corollary}\label{cor:synth_sect_no_circ}
Let $\mathcal{T}$ be a triangulation of $C(m, 2d + 1)$ with $\synth$ a section of $\mathcal{T} \backslash v$. Then there exists no circuit $(A \cup x, B \cup y)$ of $C([m]_{v+}, 2d + 1)$ such that $A$ and $B$ are both simplices in $\synth$.
\end{corollary}
\begin{proof}
If there were simplices $A$ and $B$ of $\synth$, such that $(A \cup x, B \cup y)$ was a circuit of $C([m]_{v+}, 2d + 1)$, then $(A, B)$ would be a circuit of $C([m] \setminus v, 2d - 1)$, which would contradict Lemma~\ref{lem:sects=triangs}.
\end{proof}

Referring back to Proposition~\ref{prop:sections=expansions}, this corollary allows us to show that the simplices $\synth \ast \{x, y\}$ do not contain any circuits. However, we also need to show that there can be no circuits between $\synth \ast \{x, y\}$ and $\tvminusy$, and $\synth \ast \{x, y\}$ and $\tvplusx$, for which we need the following definition and lemma, which uses Corollary~\ref{cor:synth_sect_no_circ} in its proof. Of course, we also need that there can be no circuits between $\tvminusy$ and $\tvplusx$, which we subsequently deduce from this.

\begin{definition}\label{def:above_below}
Let $\mathcal{T}$ be a triangulation of $C(m, 2d + 1)$ with $\synth$ a section of $\mathcal{T} \backslash v$ and $A$ a simplex of $\mathcal{T} \backslash v$. Then we say that $A$ is \emph{submerged} by $\synth$ if $A$ is contained in a simplex of $\CL$ or a simplex of $\synth$. Similarly, we say that $A$ is \emph{supermerged} by $\synth$ if $A$ is contained in a simplex of $\CU$ or a simplex of $\synth$.
\end{definition}

These are analogues for vertex figures of $C(m, 2d + 1)$ of the usual notions of submersion and supermersion from \cite{er,njw-hst} respectively. These usual notions are defined for $C(m, n)$ by comparing heights with respect to the $(n + 1)$-th coordinate. For $C(m, n)\backslash v$ it is not clear what direction one should use to compare heights, so we recreate the notions combinatorially using the partial order $\vorder$. We now prove the following lemma concerning submersion. There is clear intuition behind the result: it should be seen as analogous to one direction of \cite[Proposition~3.7]{njw-hst}.

\begin{lemma}\label{lem:submersion}
Let $\mathcal{T}$ be a triangulation of $C(m, 2d + 1)$ with $\synth$ a section of $\mathcal{T} \backslash v$. Then there exists no circuit $(A \cup x, B \cup y)$ of $C([m]_{v+}, 2d + 1)$ such that $A$ is a simplex of $\synth$ and $B$ is submerged by $\synth$.  
\end{lemma}
\begin{proof}
Suppose for contradiction that we are in the situation described and that there exists a circuit $(A \cup x, B \cup y)$ of $C([m]_{v+}, 2d + 1)$ such that $A$ is a simplex of $\synth$ and $B$ is submerged by $\synth$.

We show the result by induction on $\# \CL$. In the base case we have $\CL = \emptyset$, and so both $A$ and $B$ must be simplices of $\synth$. But this contradicts Corollary~\ref{cor:synth_sect_no_circ}. For the inductive step, we may assume that $\CL \neq \emptyset$, and so choose $\S \in \CL$ which is maximal, so that $\CL' := \CL \setminus \S$ is a lower set with associated section $\synthb$. By the induction hypothesis, the claim holds for $\synthb$, which is equal to $(\synth \setminus \vufacets{S}{2d + 1}) \cup \vlfacets{S}{2d + 1}$.

We have that $A$ is a simplex of $\synth$ and $B$ is submerged by $\synth$, but we cannot have this for $\synthb$ by the induction hypothesis. Hence, we must either have that $A$ is not a simplex of $\synthb$ or that $B$ is not submerged by $\synthb$. In the latter case, we must have that $A$ is contained in upper facets of $\S$ but no lower facets, and in the former case we must have that $B$ is contained in upper facets of $\S$ but no lower facets. Note that at most one of these cases can hold, since in the first case $A$ must contain the intersection of the upper facets of $\S$, and in the second case $B$ must contain the intersection of the upper facets of $S$, whereas $A$ and $B$ are disjoint. We consider each of these cases in turn.

Suppose first that $B$ is contained in upper facets of $\S$ but no lower facets. This means that $B$ is a simplex of of $\synth$. But $A$ is also a simplex of $\synth$, so that we have a circuit $(A \cup x, B \cup y)$ of $C([m]_{v+}, 2d + 1)$ where $A$ and $B$ are both simplices of $\synth$. This contradicts Corollary~\ref{cor:synth_sect_no_circ}.

Suppose now that $A$ is only contained in upper facets of $S$. We must have that either $A$ is the intersection of the upper facets of $S$, or that $S$ has $d + 1$ upper facets and $A$ is a $d$-simplex contained in all but one of these facets. Note that this latter case is not possible for triangulations of cyclic polytopes, where $2d$-simplices always have $d$ upper facets, but it is possible for triangulations of vertex figures of cyclic polytopes: see the simplex $|345|$ in Figure~\ref{fig:ivf_sections}. Simplices of the triangulated vertex figure are sometimes upside-down, as it were.

If $A$ is the intersection of the upper facets of $S$, then we have that $S = \J \cup A$ where $(\J \cup x, A \cup y)$ is a circuit of $C([m]_{v+}, 2d + 1)$. If $a, b, \sJ$ are the smallest elements of the respective sets which are greater than $v$ (or simply the smallest if no elements are greater than $v$), then we have that $\sJ < a < b$ is a cyclic ordering by considering the circuits $(\J \cup x, A \cup y)$ and $(A \cup x, B \cup y)$. We then obtain that $((A \setminus a) \cup \{\sJ, x\}, B \cup y)$ is a circuit of $C([m]_{v+}, 2d + 1)$. This contradicts the induction hypothesis, since $B$ is submerged by $\synthb$ and $(A \setminus a) \cup \sJ$ is a simplex of $\synthb$, because it lies in the lower facet $S \setminus a$ of $S$.

We now must consider the case where $S$ has $d + 1$ upper facets and $A$ is a $d$-simplex contained in all but one of these facets. Hence, let $S = \J \cup A = S_{-} \cup S_{+}$, where $(S_{-} \cup x, S_{+} \cup y)$ is a circuit and $\J \cap A = \emptyset$. By assumption, we have that $A \supseteq S_{+}$, and
\begin{align*}
&\# S_{+} = d, & &\# S_{-} = d + 1,\\
&\# A = d + 1, & &\# \J = d.
\end{align*}
This also implies that $\# B = d$, by considering the circuit $(A \cup x, B \cup y)$. We must have that at least one of $s_{0}^{-}$ and $s_{d}^{-}$ is not an element of $A$, since $\# A \cap S_{-} = 1$. Suppose that $s_{0}^{-} \notin A$; the other case behaves similarly. Here we have $s_{0}^{-} < s_{0}^{+} = a_{0} < b_{0}$. We then have that $((A \setminus a_{0}) \cup \{s_{0}^{-},x\}, B \cup y)$ is a circuit of $C([m]_{v+}, 2d + 1)$ with the lower facet $S \setminus a_{0}$ of $S$ containing $(A \setminus a_{0}) \cup s_{0}^{-}$. Thus $(A \setminus a_{0}) \cup s_{0}^{-}$ a simplex of $\synthb$, giving a contradiction, because $B$ is submerged by $\synthb$. This concludes the proof.
\end{proof}

We now apply Lemma~\ref{lem:submersion} to prove the following lemma. The intuition here is that if we have that $(A \cup x, B \cup y)$ is a circuit of $C([m]_{v+}, 2d + 1)$, then $B$ is ``above'' $A$ in the triangulated vertex figure $\mathcal{T} \backslash v$, and so there can be no section $\synth$ of $\mathcal{T} \backslash v$ where $B$ is submerged by $\synth$ and $A$ is supermerged by $\synth$.

\begin{lemma}\label{lem:synth_circ_orientation}
Let $\mathcal{T}$ be a triangulation of $C(m, 2d + 1)$ with $\synth$ a section of $\mathcal{T} \backslash v$. Then there exists no circuit $(A \cup x, B \cup y)$ of $C([m]_{v+}, 2d + 1)$ such that $A$ is supermerged by $\synth$ and $B$ is submerged by $\synth$.
\end{lemma}
\begin{proof}
Suppose for contradiction that we are in the situation described and that there exists a circuit $(A \cup x, B \cup y)$ of $C([m]_{v+}, 2d + 1)$ such that $A$ is supermerged by $\synth$ and $B$ is submerged by $\synth$. Suppose that $A$ is not a face of a simplex of $\synth$. Then there is a simplex $\S \in \mathcal{U} := (\mathcal{T}\setminus v) \setminus \CL$ such that the lower facets of $\S$ are all $(2d - 1)$-simplices of $\S$ and none of them contain $A$ as a face. We obtain that $\mathcal{L}' = \mathcal{L} \cup \S$ is also a lower set, with $A$ still supermerged by $\synthb$ and $B$ still submerged by $\synthb$. By repeating this process, we may assume that $A$ is a face of a simplex of $\synth$. But this contradicts Lemma~\ref{lem:submersion}.
\end{proof}

\subsection{Expansion at other vertices}\label{sect:vf_exp}

We can now derive the main result of this section, namely, that the different triangulations which may result from expansion at the vertex $v$ are in bijection with the sections of $\mathcal{T} \backslash v$. We prove our bijection in two halves, showing first that every expanded triangulation gives us a section.

\begin{lemma}\label{lem:triang_to_synth_sect}
Let $\mathcal{T}$ be a triangulation of $C(m, 2d + 1)$ with $\preim{\mathcal{T}}$ a triangulation of $C([m]_{v+}, 2d + 1)$ such that $\preim{\mathcal{T}}\xvy = \mathcal{T}$. Let $\mathcal{L}$ be the set of $2d$-simplices $\S$ of $\mathcal{T} \backslash v$ such that $\S \cup y$ is a $(2d + 1)$-simplex of $\preim{\mathcal{T}}$. Then $\synth = \preim{\mathcal{T}} \backslash \{x, y\}$.
\end{lemma}
\begin{proof}
To start, note that by Lemma~\ref{lem:lu_set}, the complement $\CU$ of $\CL$ in $\mathcal{T}\backslash v$ consists of the $2d$-simplices $\S$ such that $\S \cup x$ is a $(2d + 1)$-simplex of $\preim{\mathcal{T}}$.

We first prove that $\preim{\mathcal{T}}\backslash \{x, y\} \subseteq \synth$. Let $\Wa$ be a $(2d - 1)$-simplex of $\preim{\mathcal{T}} \backslash \{x, y\}$. Then $\Wa \cup \{x,y\}$ is a $(2d + 1)$-simplex of $\preim{\mathcal{T}}$. We have that $\Wa \cup x$ is either a facet of $C([m]_{v+}, 2d + 1)$ or a facet of $\R \cup x$ where $\R \in \mathcal{U}$. Likewise, either $\Wa \cup y$ is a facet of $C([m]_{v+}, 2d + 1)$ or a facet of $\S \cup y$ where $\S \in \mathcal{L}$.

Note that we cannot both have that $\Wa \cup x$ is a facet of $C([m]_{v+}, 2d + 1)$ and that $\Wa \cup y$ is a facet of $C([m]_{v+}, 2d + 1)$. To see this, suppose that $\Wa \cup x$ is an upper facet, so that it is an odd subset. This means that $x$ is an even gap in $\Wa \cup y$, since, by assumption, $y$ is an odd gap in $\Wa \cup x$. Hence, if $\Wa \cup y$ is a facet of $C([m]_{v+}, 2d + 1)$, then $x$ must be the only gap in $\Wa \cup y$, otherwise $\Wa \cup y$ must have both odd and even gaps. This means that $C([m]_{v+}, 2d + 1)$ is a $(2d + 1)$-simplex, and so $C(m, 2d + 1)$ is degenerate. The case where $\Wa \cup x$ is a lower facet is similar. 

Hence, we either have that
\begin{itemize}
\item $\Wa \cup x$ is facet of $\R \cup x$ where $\R \in \mathcal{U}$ and $\Wa \cup y$ is a facet of $\S \cup y$ where $\S \in \mathcal{L}$, or
\item $\Wa \cup x$ is a facet of $\R \cup x$ where $\R \in \mathcal{U}$ and $\Wa \cup y$ is a facet of $C([m]_{v+}, 2d + 1)$, or
\item $\Wa \cup x$ is a facet of $C([m]_{v+}, 2d + 1)$ and $\Wa \cup y$ is a facet of $\S \cup y$ where $\S \in \mathcal{L}$.
\end{itemize}
Hence, in all cases $\Wa \in \synth$, by applying Lemma~\ref{lem:facets_under_exp} and using Definition~\ref{def:synth_sect}.

We now prove that $\synth \subseteq \preim{\mathcal{T}} \backslash \{x, y\}$. Suppose that $\Wa$ is a $(2d - 1)$-simplex of $\synth$. We claim that $\Wa \cup \{x, y\}$ is a $(2d + 1)$-simplex of $\preim{\mathcal{T}}$. By the following reasoning, we have that both $\Wa \cup x$ and $\Wa \cup y$ are $2d$-simplices of $\preim{\mathcal{T}}$.
\begin{itemize}
\item If $\Wa = \R \cap \S$ where $\R \in \mathcal{L}$ and $\S \in \mathcal{U}$, then we have that $\R \cup y$ and $\S \cup x$ are $(2d + 1)$-simplices of $\preim{\mathcal{T}}$ by definition of $\mathcal{L}$ and $\mathcal{U}$.
\item If $\Wa$ is an upper facet of $C(m, 2d + 1) \backslash v$ and an upper facet of $\R$ for $\R \in \mathcal{L}$, then $\Wa \cup x$ is a $2d$-simplex of $\preim{\mathcal{T}}$ by Lemma~\ref{lem:facets_under_exp} and $\R \cup y$ is a $(2d + 1)$-simplex of $\preim{\mathcal{T}}$ by definition of $\mathcal{L}$.
\item If $\Wa$ is a lower facet of $C(m, 2d + 1) \backslash v$ and a lower facet of $\S$ for $\S \in \mathcal{U}$, then $\Wa \cup y$ is a $2d$-simplex of $\preim{\mathcal{T}}$ by Lemma~\ref{lem:facets_under_exp} and $\R \cup x$ is a $(2d + 1)$-simplex of $\preim{\mathcal{T}}$ by definition of $\mathcal{U}$.
\end{itemize}

We now show that $\Wa \cup \{x, y\}$ is a $(2d + 1)$-simplex of $\preim{\mathcal{T}}$ by applying \cite[Lemma~4.4]{njw-hst}, which states that it suffices to check that $d$- and $(d + 1)$-faces of $\Wa \cup \{x, y\}$ are in $\preim{\mathcal{T}}$. Let $A \subseteq \Wa \cup \{x, y\}$ be such that $\# A = d + 1$. If $x, y \in A$, then $A$ lies on the boundary of $C([m]_{v+}, 2d + 1)$ by Gale's Evenness Criterion, so $A$ is a $d$-simplex of $\preim{\mathcal{T}}$. If $x \notin A$, then $A$ is a $d$-face of $\Wa \cup y$, which we already know is a $(2d - 1)$-simplex of $\preim{\mathcal{T}}$. The case where $y \notin A$ may be treated similarly.

Now let $B \subseteq \Wa \cup \{x, y\}$ such that $\# B = d + 2$. Every $d$-face of $B$ is a $d$-simplex of $\preim{\mathcal{T}}$, by what we have just argued. If $x, y \in B$, then $B$ cannot be half of a circuit of $C([m]_{v+}, 2d + 1)$, since $x$ and $y$ are consecutive in $[m]_{v+}$. Applying \cite[Lemma~4.4]{njw-hst} then gives that $B$ is a $(d + 1)$-simplex of $\preim{\mathcal{T}}$. If, on the other hand, $x \notin B$ (alternatively, $y \notin B$), then $B$ is a $(d + 1)$-face of $\Wa \cup y$ (alternatively, $\Wa \cup x$), which we know is a $2d$-simplex of $\preim{\mathcal{T}}$. Therefore, by \cite[Lemma~4.4]{njw-hst}, $\Wa \cup \{x, y\}$ is a $(2d + 1)$-simplex of $\preim{\mathcal{T}}$, and so $\Wa$ is a $(2d - 1)$-simplex of $\preim{\mathcal{T}}\backslash \{x, y\}$.
\end{proof}

We now show the other half of the bijection, namely, that one can construct an expanded triangulation from every section.

\begin{lemma}\label{lem:synth_sect_to_triang}
Let $\mathcal{T}$ be a triangulation of $C(m, 2d + 1)$ with $\synth$ a section of $\mathcal{T} \backslash v$. Then there is a triangulation $\preim{\mathcal{T}}$ of $C([m]_{v+}, 2d + 1)$ such that $\preim{\mathcal{T}}\xvy = \mathcal{T}$ and $\preim{\mathcal{T}}\backslash\{x, y\} = \synth$.
\end{lemma}
\begin{proof}
Suppose that we are in the situation described and let $\CU$ be the complement of $\CL$ in $\mathcal{T} \backslash v$. We define $\preim{\mathcal{T}}$ to consist of the $(2d + 1)$-simplices \[\preim{\mathcal{T}} = \mathcal{T}^{\circ} \cup (\synth \ast \{x, y\}) \cup (\CU \ast x) \cup (\CL \ast y),\] where $\mathcal{T}^{\circ} = \set{\T \in \binom{[m]}{2d + 2} \st \T \in \mathcal{T}, v \notin \T}$. It is evident from the definition of $\preim{\mathcal{T}}$ that $\preim{\mathcal{T}}\xvy = \mathcal{T}$ and $\preim{\mathcal{T}}\backslash\{x, y\} = \synth$. We now show that $\preim{\mathcal{T}}$ is a triangulation of $C([m]_{v+}, 2d + 1)$ by explicitly verifying that it satisfies the definition of a combinatorial triangulation from Definition~\ref{def:triang}.

We first verify that, for any simplex $\T$ of $\preim{\mathcal{T}}$ and any facet $\F$ of $\T$, either $\F$ is a facet of $C([m]_{v+}, 2d + 1)$ or a facet of another $(2d + 1)$-simplex of $\preim{\mathcal{T}}$.
\begin{enumerate}[wide,  labelindent=\parindent]
\item Suppose first that $\T \in \mathcal{T}^{\circ}$. 

If $\F$ is a facet of $C(m, 2d + 1)$ in $\mathcal{T}$, $\F$ will be a facet of $C([m]_{v+}, 2d + 1)$ in $\preim{\mathcal{T}}$. Suppose instead that $\F$ is a facet of $\Tb$ for some simplex $\Tb$ in $\mathcal{T}$. Then, if $v \notin \Tb$, then $\Tb \in \mathcal{T}^{\circ} \subseteq \preim{\mathcal{T}}$. On the other hand, if $v \in \Tb$, then either $(\Tb\setminus v) \cup x$ or $(\Tb\setminus v) \cup y$ is a $(2d + 1)$-simplex of $\preim{\mathcal{T}}$, and $\F$ is a facet of either of these.
\item Suppose now that $\T \in \CU \ast x$.

If $x \notin \F$, then $\T = \F \cup x$. Then $\F \cup v$ is a $(2d + 1)$-simplex of $\mathcal{T}$, where $\F$ is either a facet of $C(m, 2d + 1)$, or a facet of a $(2d + 1)$-simplex $\Tb$, where $v \notin \Tb$. If $\F$ is a facet of $C(m, 2d + 1)$, then $\F$ is a facet of $C([m]_{v+}, 2d + 1)$. If $\F$ is a facet of a $(2d + 1)$-simplex $\Tb$ in $\mathcal{T}$, then $\Tb \in \mathcal{T}^{\circ}$, and $\F$ is a facet of $\Tb$ in $\preim{\mathcal{T}}$.

If $x \in \F$, then in $\mathcal{T}\backslash v$, $\F\setminus x$ is either a facet of $C(m, 2d + 1)\backslash v$, or a facet of some $2d$-simplex $\S$ distinct from $\T \setminus x$. We consider these two cases in turn.

If $\F\setminus x$ is a facet of $C(m, 2d + 1)\backslash v$, then it is either an upper facet or a lower facet. In the former case, by Lemma~\ref{lem:facets_under_exp}, $\F$ is a facet of $C([m]_{v+}, 2d + 1)$. In the latter case, since $\F \setminus x$ is a facet of $\T \setminus x$ and $\T \in \mathcal{U}$, we have that $\F \setminus x$ is in the section $\synth$. This then means that $\F$ is a facet of the $(2d + 1)$-simplex $(\F\setminus x) \cup \{x, y\}$ in $\preim{\mathcal{T}}$.
	
If $\F\setminus x = (\T \setminus x) \cap \S$ for some $2d$-simplex $\S$, then either $\S \in \mathcal{U}$, or $\S \in \mathcal{L}$. If $\S \in \mathcal{U}$, then $\S \cup x$ is a $(2d + 1)$-simplex of $\mathcal{T}$ distinct from $\T$ with $\F$ as a facet. If $\S \in \mathcal{L}$, then $\F \setminus x \in \synth$. In this case $(\F\setminus x) \cup \{x, y\}$ is a $(2d + 1)$-simplex of $\preim{\mathcal{T}}$ and it has $\F$ as a facet.

\item The case where $\T \in \CL \ast y$ is similar to the previous case.

\item Finally, suppose that $\T \in \synth \ast \{x, y\}$. 

If $x \notin \F$, then $\T = \F \cup x$. We have that $\F \setminus y$ is a $(2d - 1)$-simplex of $\mathcal{T}\backslash v$, and is therefore either both a facet of $C(m, 2d + 1) \backslash v$ and a facet of a $2d$-simplex $\S$ of $\mathcal{T} \backslash v$, or a shared facet of two $2d$-simplices $\R$ and $\S$ of $\mathcal{T} \backslash v$. Note also that $\F \setminus y \in \synth$, since $\T \in \synth \ast \{x, y\}$. 

If $\F \setminus y$ is both a facet of $C(m, 2d + 1) \backslash v$ and a facet of a $2d$-simplex $\S$, then either $\F \setminus y$ is a lower facet or it is an upper facet. If it is a lower facet, then $\F$ is a facet of $C([m]_{v+}, 2d + 1)$. If it is an upper facet, then we must have $\S \in \CL$, since $\F \setminus y \in \synth$. Consequently, $\S \cup y$ is a $(2d + 1)$-simplex of $\preim{\mathcal{T}}$, and $\F$ is a facet of~it.

If $\F \setminus y$ is a shared facet of two $2d$-simplices $\S$ and $\R$ in $\mathcal{T}$, then we may suppose without loss of generality that $\S \in \mathcal{L}$ and $\R \in \mathcal{U}$, since we know that $\F \setminus y \in \synth$. We then have that $\F$ is a shared facet of $\T$ and $\S \cup y$ in $\preim{\mathcal{T}}$.

The case when $y \notin \F$ is similar to the case where $x \notin \F$.

If $x, y \in \F$, then let $\Wa = \T \setminus \{x,y\}$, so that $\Wa \in \synth$. Then $\G = \F \setminus \{x, y\}$ is a facet of $\Wa$. Since, by Lemma~\ref{lem:sects=triangs}, $\synth$ is a triangulation of $C([m] \setminus v, 2d - 1)$, then there either exists $\Wb \in \synth$ such that $\Wa \cap \Wb = \G$, or that $\G$ is a facet of $C([m]\setminus v, 2d + 1)$. In the second case, we are done immediately by applying Corollary~\ref{cor:ul_facet_intersect}, which gives us that $\F = \G \cup \{x, y\}$ is a facet of $C([m]_{v+}, 2d + 1)$. In the first case, we have that $\Wb \cup \{x, y\} \in \synth \ast \{x, y\}$ and that $\F$ is a shared facet between $\T$ and $\Wb \cup \{x, y\}$.
\end{enumerate}

We must now show that there can be no pair of $(2d + 1)$-simplices $\S, \R$ in $\preim{\mathcal{T}}$ such that $\S \supseteq Z_{-}$ and $\R \supseteq Z_{+}$, where $(Z_{-}, Z_{+})$ is a circuit of $C([m]_{v+}, 2d + 1)$. Suppose for contradiction that there does exist such a pair of $(2d + 1)$-simplices $\S$ and $\R$.

We use the fact that $\preim{\mathcal{T}}\xvy = \mathcal{T}$. This implies that any such circuit $(Z_{-}, Z_{+})$ must degenerate under the contraction $\xvy$, since otherwise we would obtain a circuit in $\mathcal{T}$. This means that we have $x \in Z_{\pm}$ and $y \in Z_{\mp}$. Hence we only need to consider the cases where
\begin{enumerate}
\item $\S \in \CU \ast x$ and $\R \in \CL \ast y$;\label{op:int1}
\item $\S \in \CU \ast x$ and $\R \in \synth \ast \{x, y\}$;\label{op:int2}
\item $\S \in \synth \ast \{x, y\}$ and $\R \in \CL \ast y$; and\label{op:int3}
\item $\S, \R \in \synth \ast \{x, y\}$.\label{op:int4}
\end{enumerate}
But each case gives a contradiction to Lemma~\ref{lem:synth_circ_orientation}, or the more specific instances of Lemma~\ref{lem:submersion} and Corollary~\ref{cor:synth_sect_no_circ}. Hence, we obtain that $\preim{\mathcal{T}}$ is indeed a triangulation of $C([m]_{v+}, 2d + 1)$.
\end{proof}

\begin{proof}[Proof of Proposition~\ref{prop:sections=expansions}]
The proposition now follows by putting together the results Lemma~\ref{lem:triang_to_synth_sect}, Lemma~\ref{lem:synth_sect_to_triang}, and using Remark~\ref{rmk:odd_only}.
\end{proof}

Now, using Proposition~\ref{prop:sections=expansions}, we can apply the argument of Lemma~\ref{lem:subpolytope_exp} to the contraction $\xvy$.

\begin{proof}[Proof of Lemma~\ref{lem:subpolytope_exp_oth_vert}]
By Proposition~\ref{prop:sections=expansions}, triangulations $\preim{\mathcal{T}}$ of $C([m]_{v+},n)$ such that $\preim{\mathcal{T}}\xvy=\mathcal{T}$ are in bijection with sections $\synth$ of $\mathcal{T}\backslash v$. Moreover, given a section $\synth$ of $\mathcal{T}\backslash v$, the corresponding triangulation $\preim{\mathcal{T}}$ has the set of $n$-simplices \[\mathcal{T}^{\circ} \cup (\synth\ast \{x, y\}) \cup (\CL\ast x) \cup (\CU \ast y),\] where $\mathcal{T}^{\circ}$ denotes the $n$-simplices of $\mathcal{T}$ which do not contain $v$ and $\mathcal{U}$ is the complement of $\mathcal{L}$ in $\mathcal{T}\backslash v$.

The set-up of Lemma~\ref{lem:subpolytope_exp_oth_vert} gives us that $\mathcal{T}$ contains a cyclic subpolytope $C(\U, n)$. We let $\mathcal{T}_{\U}$ be the induced triangulation of this subpolytope in $\mathcal{T}$. If $v \notin \U$, then $\mathcal{T}_{\U} \subseteq \mathcal{T}^{\circ}$, so $C(S, n)$ is a subpolytope of $\preim{\mathcal{T}}$. Hence, we assume that $v \in \U$. There are then three options:
\begin{enumerate}
\item $\mathcal{T}_{\U}\backslash v \subseteq \CL$.\label{op:v_below}
\item $\mathcal{T}_{\U}\backslash v \subseteq \CU$.\label{op:v_above}
\item $\mathcal{T}_{\U}\backslash v$ has non-empty intersection with both $\CL$ and $\CU$.\label{op:v_intersect}
\end{enumerate}
In case (\ref{op:v_below}) we have that $C((\U \setminus v) \cup y, n)$ is a subpolytope of $\preim{\mathcal{T}}$. In case (\ref{op:v_above}), we have that $C((\U \setminus v) \cup x, n)$ is a subpolytope of $\preim{\mathcal{T}}$.

In case (\ref{op:v_intersect}), let $\CL_{\U} = \CL \cap \mathcal{T}_{\U}$ and $\CU_{\U} = \CU \cap \mathcal{T}_{\U}$. Then $\CL_{\U}$ is a lower set of the restriction of $\vorder$ to $\mathcal{T}_{\U}\backslash v$. We then obtain a section $\sect{\CL_{\U}}$ of $\mathcal{T}_{\U}\backslash v$, and it is straightforward to see that $\sect{\CL_{\U}}$ consists of the $(2d - 1)$-simplices of $\synth$ which are also $(2d - 1)$-simplices of $\mathcal{T}_{\U}\backslash v$. By Proposition~\ref{prop:sections=expansions}, we have that the section $\sect{\CL_{\U}}$ of $\mathcal{T}_{\U} \backslash v$ gives us a triangulation $\preim{\mathcal{T}}_{\U}$ of $C(\U_{v+}, n)$. Moreover, the triangulation $\preim{\mathcal{T}}_{\U}$ of $C(\U_{v+}, n)$ has simplices \[\mathcal{T}_{\U}^{\circ} \cup (\sect{\CL_{\U}} \ast \{x, y\}) \cup (\CU_{\U}\ast x) \cup (\CL_{\U} \ast y).\] It is then clear that $\mathcal{T}_{\U}^{\circ} \subseteq \mathcal{T}^{\circ}$, $\sect{\CL_{\U}} \subseteq \synth$, $\CU_{\U} \subseteq \CU$, and $\CL_{\U} \subseteq \CL$. Hence $\preim{\mathcal{T}}_{\U}$ is a subtriangulation of $\preim{\mathcal{T}}$, which gives us that $C(\U_{v+}, n)$ is a subpolytope of~$\preim{\mathcal{T}}$.
\end{proof}

\subsection{Order-preservation}

We finish the section by showing that the contraction $\xvy$ is order-preserving with respect to the second order. To do this, we show how this operation may be interpreted combinatorially, and then use this interpretation to show that it is order-preserving.

\begin{lemma}\label{lem:cont_comb_int}
Let $\preim{\mathcal{T}}$ be a triangulation of $C([m]_{v+}, n)$ with $v \in [2, m - 1]$ and let $\mathcal{T} = \preim{\mathcal{T}}\xvy$. Then \[\simp(\mathcal{T}) = \set{A \in \nonconsec{m}{\floor{n/2}} \st A = \preim{A}\xvy \text{ for some } \preim{A} \in \simp(\preim{\mathcal{T}})}.\]
\end{lemma}

Note that this lemma applies to both $n$ even and $n$ odd. If $n$ is odd and $A \in \nonconsec{m}{\floor{n/2}}$ is such that $A = \preim{A}\xvy$ for some $\preim{A} \in \simp(\preim{\mathcal{T}})$, then we automatically have that $A \in \chainset{m}{\floor{n/2}}$, since $\preim{A} \in \chainset{[m]_{v+}}{\floor{n/2}}$.

\begin{proof}
It is immediate that \[\simp(\mathcal{T}) \supseteq \set{A \in \nonconsec{m}{\floor{n/2}} \st A = \preim{A}\xvy \text{ for some } \preim{A} \in \simp(\preim{\mathcal{T}})}\] from the definitions of $\simp(\mathcal{T})$ and $\preim{\mathcal{T}}\xvy$.

We now show that \[\simp(\mathcal{T}) \subseteq \set{A \in \nonconsec{m}{\floor{n/2}} \st A = \preim{A}\xvy \text{ for some } \preim{A} \in \simp(\preim{\mathcal{T}})}.\] If $A \in \simp(\mathcal{T})$, then there must exist a simplex $\preim{A}$ of $\preim{\mathcal{T}}$ such that $\preim{A}\xvy = A$. Without loss of generality, we may assume that $\{x,y\} \not\subseteq \preim{A}$, since in this case we may remove either $x$ or $y$ from $\preim{A}$ and still have $\preim{A}\xvy = A$. But then we must have that $\preim{A}$ is an internal $\floor{n/2}$-simplex, since $A$ is an internal $\floor{n/2}$-simplex. Hence, $\preim{A} \in \simp(\preim{\mathcal{T}})$, as desired.
\end{proof}

We can now show that $\xvy$ is order-preserving with respect to the second order.

\begin{proof}[Proof of Lemma~\ref{lem:v_cont_op}]
We split into two cases depending on whether $n$ is odd or even so that we can use the combinatorial interpretations of the second higher Stasheff--Tamari order from Theorem~\ref{thm:hst_char_even} and Theorem~\ref{thm:hst_char_odd}.

We first let $n = 2d$. We show that if $\mathcal{T} \not\leqslant_{2} \mathcal{T}'$, then $\preim{\mathcal{T}} \not\leqslant_{2} \preim{\mathcal{T}}'$. Suppose that there exists $B \in \simp(\mathcal{T}')$ and $A \in \simp(\mathcal{T})$ such that $B \wr A$. Then we have that $\preim{B} \in \simp(\preim{\mathcal{T}}')$ and $\preim{A} \in \simp(\preim{\mathcal{T}})$, with $A = \preim{A}\xvy$ and $B = \preim{B}\xvy$, by Lemma~\ref{lem:cont_comb_int}. Then we also must have have $\preim{B} \wr \preim{A}$, since at most one of $A$ and $B$ can contain $v$. This implies that $\preim{\mathcal{T}} \not\leqslant_{2} \preim{\mathcal{T}}'$, as desired.  

We now let $n = 2d + 1$. We shall show that if $\preim{\mathcal{T}} \leqslant_{2} \preim{\mathcal{T}}'$, then $\simp(\mathcal{T}) \supseteq \simp(\mathcal{T}')$. Let $A \in \simp(\mathcal{T}')$. Then $A = \preim{A}\xvy$ for some $\preim{A} \in \simp(\preim{\mathcal{T}}')$ by Lemma~\ref{lem:cont_comb_int}. Since $\simp(\preim{\mathcal{T}}) \supseteq \simp(\preim{\mathcal{T}}')$, we then have that $\preim{A} \in \simp(\preim{\mathcal{T}})$, which implies that $\preim{A}\xvy = A \in \simp(\mathcal{T})$, as desired.
\end{proof}	

\addtocontents{toc}{\protect\setcounter{tocdepth}{0}} 
\printbibliography

\end{document}